\documentclass[a4paper, oneside, reqno]{amsart}

\usepackage[a4paper]{geometry}
\usepackage[utf8]{inputenc}
\usepackage[english]{babel}
\usepackage[T1]{fontenc}
\usepackage{textcomp}
\usepackage{lmodern}
\usepackage{graphicx}

\usepackage{todonotes}

\usepackage{amsmath}
\usepackage{amsfonts}
\usepackage{amsthm}
\usepackage{latexsym}
\usepackage{amssymb}
\usepackage[all,cmtip]{xy}

\newenvironment{psmallmatrix}{\left(\begin{smallmatrix}}{\end{smallmatrix}\right)}

\DeclareMathOperator{\Hom}{Hom}

\DeclareMathOperator{\Fundg}{Fun_{dg}}

\DeclareMathOperator{\Hqe}{\mathsf{Hqe}}
\DeclareMathOperator{\cone}{C}

\DeclareMathOperator{\coker}{coker}

\DeclareMathOperator{\res}{Res}
\DeclareMathOperator{\Ind}{Ind}

\DeclareMathOperator{\Inj}{Inj}
\DeclareMathOperator{\Proj}{Proj}
\DeclareMathOperator{\DGInj}{DGInj}
\DeclareMathOperator{\DGProj}{DGProj}

\DeclareMathOperator{\MC}{MC}
\DeclareMathOperator{\Tot}{Tot}
\DeclareMathOperator{\Tw}{Tw}
\DeclareMathOperator{\Img}{Im}
\DeclareMathOperator{\pr}{pr}
\DeclareMathOperator{\incl}{incl}

\DeclareMathOperator{\mmod}{mod}

\newcommand{\op}{\mathrm{op}}

\newcommand{\cat}{\mathbf}

\newcommand{\opp}[1]{{#1}^{\mathrm{op}}}
\newcommand{\kat}{\mathsf}

\newcommand{\qis}{\overset{\mathrm{qis}}{\approx}}

\newcommand{\pretr}[1]{\mathrm{pretr}(#1)}

\newcommand{\basering}[1]{\mathbf{#1}}

\newcommand{\Mod}[1]{\mathsf{Mod}(#1)}
\newcommand{\Modfp}[1]{\mathsf{mod}(#1)}

\newcommand{\comp}[1]{\mathsf{C}(#1)}
\newcommand{\hocomp}[1]{\mathsf{K}(#1)}
\newcommand{\dercomp}[1]{\mathsf{D}(#1)}
\newcommand{\dercompmin}[1]{\mathsf{D}^{-}(#1)}
\newcommand{\compdg}[1]{\mathsf{C}_\mathrm{dg}(#1)}

\newcommand{\hproj}[1]{\mathrm{h\textrm{-}proj}(#1)}
\newcommand{\hprojmin}[1]{\mathrm{h\textrm{-}proj}^{-}(#1)}

\newcommand{\shiftid}[3]{1_{({#1},{#2},{#3})}}

\newcommand{\bd}{\mathrm{b}}
\newcommand{\holim}{\mathrm{ho}\!\varprojlim}
\newcommand{\hocolim}{\mathrm{ho}\!\varinjlim}

\newtheorem{thm}{Theorem}[section]

\newtheorem{prop}[thm]{Proposition}
\newtheorem{coroll}[thm]{Corollary}

\newtheorem{lemma}[thm]{Lemma}

\theoremstyle{remark}
\newtheorem{remark}[thm]{Remark}
\newtheorem{example}[thm]{Example}

\theoremstyle{definition}
\newtheorem{defin}[thm]{Definition}

\numberwithin{equation}{section}

\entrymodifiers={+!!<0pt,\fontdimen22\textfont2>}   

\title{t-structures and twisted complexes on derived injectives}

\author{Francesco Genovese}
\address[Francesco Genovese]{Universiteit Antwerpen, Departement Wiskunde-Informatica, Middelheimcampus,
	Middelheimlaan 1,
	2020 Antwerp, Belgium}
\email{Francesco.Genovese@uantwerpen.be}

\author{Wendy Lowen} 
\address[Wendy Lowen]{Universiteit Antwerpen, Departement Wiskunde-Informatica, Middelheimcampus,
Middelheimlaan 1,
2020 Antwerp, Belgium}
\address{Laboratory of Algebraic Geometry, National Research University, Higher School of Economics, Moscow, Russia}
\email{wendy.lowen@uantwerpen.be}

\author{Michel Van den Bergh}
\address[Michel Van den Bergh]{Universiteit Hasselt\\ Campus Diepenbeek\\ Agoralaan Gebouw D \\ 3590 Diepenbeek \\Belgium}
\email{michel.vandenbergh@uhasselt.be}

\thanks{The authors acknowledge the support of the Research Foundation Flanders (FWO) under Grant No G.0D86.16N, and of the Russian Academic Excellence Project `5-100'. This project has received funding from the European Research Council (ERC) under the European Union’s Horizon 2020 research and innovation programme (grant agreement No. 817762).}

\keywords{pretriangulated dg-categories, t-structures, derived injectives, twisted complexes}
\subjclass[2010]{18E30, 18G05, 18G35}

\let\oldmarginpar\marginpar
\def\marginpar#1{\oldmarginpar{\raggedright \tiny \baselineskip 0pt \lineskip 0pt #1}}
\def\HH{\operatorname{HH}}

\begin{document}
\begin{abstract}
In the paper \emph{Deformation theory of abelian categories}, the last two authors proved that an abelian category with enough injectives can be reconstructed as the category of finitely presented modules over the category of its injective objects. We show a generalization of this to pretriangulated dg-categories with a left bounded non-degenerate t-structure with enough derived injectives, the latter being derived enhancements of the injective objects in the heart of the t-structure. Such dg-categories (with an additional hypothesis of closure under suitable products) can be completely described in terms of left bounded twisted complexes of their derived injectives.
\end{abstract}
\maketitle
\tableofcontents
\section{Introduction}
This paper is the first one in an ongoing project to develop the
deformation theory of triangulated categories with t-structure. The current paper is intended as the
foundation for \cite{deftria1,deftria2} in which the actual deformation theory is developed.
Taken together these papers should be viewed as sequels to
\cite{lowen-vandenbergh-deformations-abelian,lowen-vdb-hochschild} which are about
the deformation theory of abelian categories. An abelian category can always be viewed as the heart of the tautological
t-structure on its (triangulated) derived category and this
provides the link of the current ``triangulated'' setting
with the earlier ``abelian'' setting.

\medskip

To be more concrete let $\cat k$ be a field and let $\mathfrak A$ be a $\cat k$-linear abelian category. In \cite{lowen-vdb-hochschild} the last two
  authors defined the Hochschild cohomology
  $\HH^\ast(\mathfrak A)$ of $\mathfrak A$ and showed that $\HH^{2,3}(\mathfrak A)$ provides an obstruction theory\footnote{The methods in loc.\ cit.\ may also be used to give natural deformation theoretic interpretations for the lower groups $\HH^{0,1}(\mathfrak A)$.} for the (suitably defined)
 deformations of $\mathfrak A$.  
Restricting ourselves for simplicity to first order deformations we have in particular that
$ \HH^2(\mathfrak A)$
  parametrizes the deformations of $\mathfrak
  A$ over the dual numbers $D_0:=\cat k[\epsilon]/(\epsilon^2)$.

It seems natural to look for a deformation theoretic interpretation
of the higher Hochschild cohomology groups of $\mathfrak A$. Indeed based on general principles the Hochschild cohomology groups $\HH^n(\mathfrak A)$
for $n\ge 2$ ``should'' correspond to deformations of $\mathfrak A$  over the DG-algebra $D_{2-n}=\cat k[\epsilon]/(\epsilon^2)$ where now $|\epsilon|=2-n$. However it is impossible
to realize this objective in the abelian world as there is no sensible notion of a $D_{2-n}$-linear abelian category. 
But we will show that it is possible do it in the triangulated world! Indeed the theory developed in the current paper allows one to associate to a class in $\eta\in \HH^n(\mathfrak A)$ for $n\ge 3$ a triangulated category $D^+(\mathfrak A)_{\eta}$ with t-structure (whose heart happens to be also $\mathfrak A$) which is linear over $D_{2-n}$ and which for all practical purposes behaves as a deformation of $D^+(\mathfrak A)$ corresponding to $\eta$ (see \cite{deftria1}).
In a subsequent paper~\cite{deftria2} we will show that this procedure is in fact reversible and that all deformations of $D^+(\mathfrak A)$ over $D_{2-n}$ (for a suitable notion of deformation) 
are of the form $D^+(\mathfrak A)_\eta$. 

\medskip

Note however that it may appear that we are actually solving  a  non-problem. 
Indeed, unsurprisingly the abstract theory of
  triangulated categories is too weak for us and we work instead
  with pretriangulated dg-categories \cite{bondal-kapranov} which have in particular a standard notion of Hochschild cohomology.  So let
  $\cat A$ be a pretriangulated dg-category. 
  If  $\tilde{\eta}$ is a Hochschild cocycle representing a class
  $\eta\in \HH^\ast(\cat A)$ then we may use it to deform the
  DG-category $\cat A$ \cite{kuznetsov-height}\footnote{We are skipping some technicalities. Either one has to replace $\cat A$ by a cofibrant model, or else one has to use $A_\infty$-categories.} much in the same way as we deform algebras, and so in particular the presence of a t-structure seems to be irrelevant! The catch however is that in general $\tilde{\eta}$ will
  have curvature and hence the same will be true for the corresponding deformation of $\cat A$
  (roughly speaking $d^2\neq 0$)
  \cite{dedeken-lowen-deformations, LowenVdBFormal}.  Homological algebra over curved
  dg-categories is possible \cite{dedeken-lowen-deformations,Positselski1} but
  presents rather serious technical difficulties. One may attempt to solve this
``curvature problem''
 by
  replacing $\cat A$ by a Morita equivalent dg-category $\cat A'$
  (which has the same Hochschild cohomology as $\cat A$) such that
  over $\cat A'$, $\eta$ may be represented by a cocycle without
  curvature but this appears not to be possible in general
  \cite{keller-lowen-nicolas-curved}. Part of the motivation for the papers
  \cite{deftria1,deftria2} is now precisely to show that the curvature
  problem can be solved in a natural way for triangulated categories with
  t-structure.

\medskip

Now we describe more concretely the content of the current paper. Let
us first recall the abelian setting
\cite{lowen-vandenbergh-deformations-abelian,lowen-vdb-hochschild}. Assume that
$\mathfrak A$ is an abelian category with enough injectives and let
\begin{equation}
\label{eq:inj}
\cat E=\Inj \mathfrak A\,.
\end{equation}
Then $\cat E$ has weak cokernels and it particular the
category of finitely presented left $\cat E$-modules (denoted by
$\mmod(\cat E^{\mathrm{op}})$) is abelian (in other words $\cat E$ is
left coherent). Moreover the restricted Yoneda functor $A \mapsto \cat E(A,-)$
gives an equivalence of categories
\begin{equation}
\label{eq:coherent}
\mathfrak A \cong \opp{\mmod(\opp{\cat E}) }\,.
\end{equation}
Furthermore the relations \eqref{eq:inj} and \eqref{eq:coherent} are  in fact reversible.
In other words we may start with a Karoubian additive coherent category $\cat E$, 
put $\mathfrak A:=\opp{\mmod({\opp{\cat E}})}$ and then we find $\cat E\cong \Inj \mathfrak A$.
Elaborating on this one finds that there is an equivalence of categories
\begin{multline}
\label{eq:karoubian}
\{\text{Abelian categories with enough injectives, with functors possessing an exact left adjoint}\}\\
\cong \{\text{Karoubian additive coherent categories}\}\qquad\qquad\qquad\qquad\qquad
\end{multline}
and this provides
a natural path towards the deformation theory of abelian categories. For example if $\mathfrak A$ is linear over a field
$\cat k$ then we put $\HH^*_{\cat k}(\mathfrak A):=\HH^\ast_{\cat k}(\cat E)$. 
\medskip

The main result in the current paper is an analogue 
of \eqref{eq:coherent} in the triangulated setting.  From now on we fix a ground field $\basering k$ and all objects will be $\basering k$-linear.
Let $\cat T$ be a triangulated category equipped with a t-structure with heart $\cat T^{\heartsuit}$. We say that $\cat T$ has enough \emph{derived injectives}
if $\cat T^{\heartsuit}$ has enough injectives and for every injective $I$ in $\cat T^{\heartsuit}$ there exists an object
$L(I)\in \cat T$ such that ${\cat T}(-,L(I))\cong \cat T^{\heartsuit} (H^0(-),I)$. From this definition it is clear
that the category of derived injectives in $\cat T$ is closed under (existing) products.

\medskip

We now state our main results. They are triangulated analogues of the results outlined above in the abelian case.\footnote{For compatibility with the abelian case we state our results
here under the assumption that there are enough derived injectives.
However in the body of the paper the results will be stated for triangulated categories
with enough \emph{derived projectives}.} 
\begin{prop}[Dual version of Lemma \ref{lemma:dgproj_hlc_karoubian}, Definition \ref{def:hlc}]
Assume that $\cat T$ is a triangluated category with $t$-structure which has enough derived injectives and furthermore that it is ``enhanced'' in the sense
of \cite{bondal-kapranov}. I.e. $\cat T=H^0(\cat A)$ where $\cat A$ is a pretriangulated dg-category\footnote{Note that the notation $H^0$ denotes both 
the ordinary linear category associated to a dg-category and the cohomological functor associated to a t-structure. Since both notations are quite standard this dual use appears difficult to avoid.}
 (see \S\ref{subsec:dgfacts}). Let
$\cat J$ be the full sub-dg-category of $\cat A$ spanned by the derived injective objects. 

The dg-category $\cat J$ is \emph{(left) homotopically locally coherent} (hlc). I.e.
\begin{itemize}
\item $\cat J$ is cohomologically concentrated in nonpositive degrees: for all $A,A' \in \cat J$, we have $H^i(\cat J(A,A'))=0$ for all $i>0$.
\item $H^0(\cat J)$ is an additive, left coherent $\basering k$-linear category.
\item For all $i \in \mathbb Z$ and all $A \in \cat J$, the left $H^0(\cat J)$-module $H^i(\cat J(A, -))$ is finitely presented.
\end{itemize}
Moreover the category $H^0(\cat J)$ is Karoubian.
\end{prop}
The following theorem  provides a method for constructing triangulated categories with a t-structure and with enough - prescribed - derived injectives.
\begin{thm}[Dual version of Theorem  \ref{thm:construction}: ``construction''] 
\label{thm:main} Let $\cat J$ be a hlc dg-category such that $H^0(\cat J)$ is Karoubian. 
Then the dg-category $\Tw^+(\cat J) = \Tw^-(\opp{\cat J})^{\op}$ of bounded below twisted complexes over $\cat J$ (see \S\ref{sec:twcom})
has a non-degenerate t-structure whose heart is the category $\Modfp{H^0(\cat J)^\op}^{\op}$, has enough derived
injectives, and the dg-category of derived injectives is the closure of $\cat J \hookrightarrow \Tw^+(\cat J)$ under isomorphisms in $H^0(\Tw^+(\cat J))$.

Moreover, if $H^0(\cat J)$ is closed under countable products, then the t-structure on $\Tw^+(\cat J)$ is closed under countable products, that is, the aisles $\Tw^+(\cat J)_{\geq M}$ are closed under countable products.
\end{thm}
If $\cat B$ is a dg-category then we denote by $\hproj{\cat B}$ the category of h-projective (see \eqref{eq:h-projective}) right $\cat B$-modules. Theorem \ref{thm:main} is established
by showing that the 
totalisation dg-functor induces a quasi-equivalence:
\begin{equation*}
\Tot \colon \Tw^+(\cat  J) \to \hprojmin{\cat J^\op}^{\mathrm{hfp},\op}.
\end{equation*}
where, the dg-category $\hprojmin{\cat J^\op}^{\mathrm{hfp}}$ is given by
\begin{equation*}
\hprojmin{\cat J^\op}^{\mathrm{hfp}} = \{M \in \hproj{\cat J^\op} : H^i(M) \in \Modfp{H^0(\cat J^\op)} \ \ \forall\, i, H^i(M)=0\text{ for } i\gg 0\}.
\end{equation*}
The following theorem explains how a pretriangulated dg-category may be reconstructed from its category of derived injectives.
\begin{thm}[Dual version of Theorem \ref{thm:comparison}:
  ``reconstruction''] \label{thm:comparison_inj} Let $\cat A$ be
  a  pretriangulated dg-category with a non-degenerate left bounded t-structure, with
enough derived injectives, and which is closed under countable products (namely, the aisles $H^0(\cat A)_{\geq M}$ are closed under countable products). Let $\cat J$ be the
  dg-category of derived injectives. The restricted Yoneda-functor
\begin{equation*}
{\cat A}^{\op} \to \hprojmin{\opp{\cat J}}:A\mapsto {\cat A}(A,-)
\end{equation*}
is t-exact and induces a quasi-equivalence between the
dg-categories $\opp{\cat A}$ and $\hprojmin{\opp{\cat
    J}}^{\mathrm{hfp}} \approx \Tw^-(\opp{\cat J})$. In particular we obtain
a quasi-equivalence between $\cat A$ and $\Tw^+({\cat J})$.
\end{thm}

The above ``construction'' and ``reconstruction'' theorems can be enhanced to a functorial correspondence. We denote by $\Hqe$ the homotopy category of (small) dg-categories, namely the localization of the category of (small) dg-categories along quasi-equivalences. We further define categories as follows (dual versions of Definition \ref{defin:Hqe_DGProj} and Definition \ref{defin:Hqe_tmin}):
\begin{itemize}
\item The category $\Hqe^{\mathrm{DGInj}}$ has objects the dg-categories $\cat J$ which are (left) hlc and such that $H^0(\cat J)$ is Karoubian; a morphism $F \colon \cat J \to \cat J'$ in $\Hqe^{\mathrm{DGInj}}$ is a morphism in $\Hqe$ such that for all $J' \in \cat J'$, the $H^0(\opp{\cat J})$-module $H^0(\cat J')(J',F(-))$ is finitely presented. We also denote by $\Hqe^{\mathrm{DGInj}}_{\Pi}$ the full subcategory of $\Hqe^{\mathrm{DGInj}}$ of dg-categories $\cat J$ such that $H^0(\cat J)$ is closed under countable products.
\item The category $\Hqe^{\mathrm{t+}}$ has objects the dg-categories $\cat A$ endowed with a non-degenerate left bounded t-structure with enough derived injectives; a morphism in $\Hqe^{\mathrm{t+}}$ is a morphism in $\Hqe$ which has a t-exact left adjoint. We also denote by $\Hqe^{\mathrm{t+}}_{\Pi}$ the full subcategory of $\Hqe^{\mathrm{t+}}$ of dg-categories $\cat A$ with a t-structure which is closed under countable products (i.e. the aisles $H^0(\cat A)_{\geq M}$ are closed under countable products).
\end{itemize}
Then, we have the following theorem.
\begin{thm}[Dual version of Theorem \ref{thm:correspondence}: ``correspondence'']
The mapping $\cat J \mapsto \Tw^+(\cat J)$ for $\cat J \in \Hqe^\mathrm{DGInj}$ gives rise to a fully faithful functor
\begin{equation*}
\Tw^+ \colon \Hqe^\mathrm{DGInj} \to \Hqe^{\mathrm{t+}},
\end{equation*}
which induces an equivalence of categories
\begin{equation*}
\Tw^+ \colon \Hqe^\mathrm{DGInj}_{\Pi} \to \Hqe^{\mathrm{t+}}_{\Pi}.
\end{equation*}
The inverse is given by taking derived injectives:
\begin{equation*}
\DGInj \colon \Hqe^{\mathrm{t+}}_{\Pi} \to \Hqe^\mathrm{DGInj}_{\Pi}.
\end{equation*}
\end{thm}
\section{Preliminaries}
We fix once and for all a ground field $\basering k$. Every category will be assumed to be $\basering k$-linear. Moreover, we shall work within a fixed universe $\mathcal U$, and every category $\cat A, \cat B, \cat Q,\ldots$ we shall fix will be $\mathcal U$-small.
\subsection{Dg-categories} \label{subsec:dgfacts}
We assume the reader to be acquainted with triangulated categories and dg-categories, see for example \cite{keller-dgcat} or \cite{toen-dgcat}. We recollect here some notation and terminology we shall need throughout the paper.

\subsubsection{} The (locally $\mathcal U$-small) dg-category of $\mathcal U$-small cochain complexes over $\basering k$ is denoted by $\compdg{\basering k}$. 

For any pair of dg-categories $\cat A$ and $\cat B$, we have the ($\mathcal U$-small) dg-category of dg-functors $\Fundg(\cat A, \cat B)$, which is the internal hom in the symmetric monoidal category $\kat{dgCat}$ of ($\mathcal U$-small) dg-categories, namely it satisfies the natural isomorphism: 
\begin{equation*}
\Fundg(\cat A \otimes \cat B,\cat C) \cong \Fundg(\cat A, \Fundg(\cat B,\cat C)),
\end{equation*}
for all $\cat A, \cat B, \cat C \in \kat{dgCat}$. The dg-category $\cat A \otimes \cat B$ is the tensor product of $\cat A$ and $\cat B$, and it is $\mathcal U$-small. 

The dg-category of \emph{(right) $\cat A$-dg-modules} is defined by 
\begin{equation*}
\compdg{\cat A} = \Fundg(\opp{\cat A}, \compdg{\basering k}),
\end{equation*}
whereas \emph{left $\cat A$-dg-modules} are by definition $\opp{\cat A}$-dg-modules. Moreover, we set:
\begin{align*}
\comp{\cat A} &= Z^0(\compdg{\cat A}), \\
\hocomp{\cat A} &= H^0(\compdg{\cat A}).
\end{align*}

The \emph{derived category} $\dercomp{\cat A}$ of $\cat A$ is the localization of $\comp{\cat A}$ (or equivalently $\hocomp{\cat A}$) along quasi-isomorphisms. We remark that $\compdg{\cat A}, \comp{\cat A}, \hocomp{\cat A}, \dercomp{\cat A}$ are all $\mathcal U$-locally small.

 Normally, we shall use the symbol ``$\approx$'' meaning ``isomorphic in the homotopy category $H^0(\cat B)$'' of a suitable dg-category $\cat B$. In particular, for two given $M,N \in \compdg{\cat A}$, we write $M \approx N$ whenever $M \cong N$ in $\hocomp{\cat A}$ and sometimes $M \qis N$ whenever $M \cong N$ in $\dercomp{\cat A}$.
\subsubsection{} A dg-functor $F \colon \cat A \to \cat B$ between dg-categories is a \emph{quasi-equivalence} if it induces quasi-isomorphisms between the hom-complexes, and $H^0(F) \colon H^0(\cat A) \to H^0(\cat B)$ is essentially surjective. The category $\kat{dgCat}$ of $\mathcal U$-small dg-categories has a model structure whose weak equivalences are the quasi-equivalences (see \cite{tabuada-dgcat}). We denote by $\Hqe$ the homotopy category of dg-categories, namely the localization of $\kat{dgCat}$ along quasi-equivalences. Two dg-categories are \emph{quasi-equivalent} if they are isomorphic in $\Hqe$. We also say that a dg-category $\cat A$ is \emph{essentially $\mathcal U$-small} if it is quasi-equivalent to a $\mathcal U$-small dg-category.
\subsubsection{} A dg-module $P \in \compdg{\cat A}$ is \emph{h-projective} if
\begin{equation}
\label{eq:h-projective}
\hocomp{\cat A}(P,X) = 0,
\end{equation}
for all acyclic $\cat A$-dg-modules $X$; equivalently, if the localization functor $\hocomp{\cat A} \to \dercomp{\cat A}$ induces an isomorphism
\begin{equation} \label{eq:hproj_loc_iso}
\hocomp{\cat A}(P,X) \xrightarrow{\sim} \dercomp{\cat A}(P,X),
\end{equation}
for all $X \in \compdg{\cat A}$. The full dg-subcategory of $\compdg{\cat A}$ of h-projective dg-modules is denoted by $\hproj{\cat A}$. The restriction of the localization functor
\begin{equation*}
H^0(\hproj{\cat A}) \to \dercomp{\cat A}
\end{equation*}
is an equivalence, so $\hproj{\cat A}$ is a dg-enhancement of $\dercomp{\cat A}$.

Notice that for any $A \in \cat A$, the representable dg-module $\cat A(-,A)$ is h-projective by the dg-Yoneda lemma. So, the Yoneda embedding gives rise to a dg-functor
\begin{equation}
h_{\cat A} \colon \cat A \hookrightarrow \hproj{\cat A},
\end{equation}
which in turn induces the so-called \emph{derived Yoneda embedding}:
\begin{equation} \label{eq:derivedYoneda}
H^0(\cat A) \hookrightarrow \dercomp{\cat A}.
\end{equation}
\subsubsection{} Denote by $\pretr{\cat A}$ the smallest full dg-subcategory of $\compdg{\cat A}$ which contains (the Yoneda image of) $\cat A$ and is closed under taking shifts of dg-modules and mapping cones of closed degree $0$ morphisms. We say that $\cat A$ is \emph{strongly pretriangulated} (respectively, \emph{pretriangulated})  if the Yoneda embedding
\begin{equation*}
\cat A \hookrightarrow \pretr{\cat A}
\end{equation*}
is a dg-equivalence (respectively, a quasi-equivalence). The dg-category $\pretr{\cat A}$ is itself strongly pretriangulated and it is called the \emph{pretriangulated hull} of $\cat A$. We remark that $\pretr{\cat A}$ is essentially $\mathcal U$-small: in fact, it is equivalent to the $\mathcal U$-small dg-category of bounded one-sided twisted complexes on $\cat A$ (see \cite[Definition 4.6]{bondal-larsen-lunts}).
$\cat A$ is a strongly pretriangulated dg-category if and only if it is closed under \emph{pretriangles}, which are sequences of the form
\begin{equation}
\begin{gathered}
\xymatrix{
A \ar[r]^f & B \ar@<-.5ex>@{<-}[r]_-{s} \ar@<.5ex>[r]^-{j}  & \cone(f) \ar@<-.5ex>@{<-}[r]_-{i} \ar@<.5ex>[r]^-{p} & A[1],
}
\end{gathered}
\end{equation}
where $f \colon A \to B$ is a closed degree $0$ morphism in $\cat A$. The object $\cone(f) \in \cat A$ is the \emph{cone} of $f$, and $A[1] \in \cat A$ is the \emph{shift} of $A$. They are objects representing respectively the usual mapping cone of $f_* \colon \cat A(-,A) \to \cat A(-,B)$ and the shift of $\cat A(-,A)$ in $\compdg{\cat A}$. The shifts $A[m]$ of $A$ come with closed invertible degree $n-m$ maps (``shifted identity morphisms'')
\begin{equation}
\shiftid{A}{n}{m} \colon A[n] \to A[m]
\end{equation}
which satisfy $\shiftid{m}{n}{A} \circ \shiftid{n}{m}{A} = \shiftid{A}{n}{n} = 1_{A[n]}$  The maps $i,j,p,s$ characterise the cone (and the pretriangle) as follows: they are of degree $0$ and they describe $\cone(f)$ as the biproduct $A[1] \oplus B$ in the underlying graded category of $\cat A$. Moreover, they satisfy:
\begin{equation*}
dj=0, \quad dp=0, \quad di=j  f  \shiftid{A}{1}{0}, \quad ds= -f \shiftid{A}{1}{0} p.
\end{equation*}
This allows us use matrix notation as follows when describing maps to and from a cone:
\begin{align*}
u & = (u_1, u_2) \colon \cone(f) \to D, \\
v &= \begin{pmatrix} v_1 \\ v_2 \end{pmatrix} \colon D \to \cone(f).
\end{align*}
We can also write down explicit formulas for the differentials:
\begin{equation} \label{eq:differentials_cone}
du = (d u_1 - (-1)^{|u|} u_2 f \shiftid{A}{1}{0}, d u_2), \quad dv = \begin{pmatrix} dv_1 \\ d v_2 + f \shiftid{A}{1}{0}v_1 \end{pmatrix},
\end{equation}
where $|u|$ is the degree of $u$.
 \subsection{Quasi-functors}
 The morphisms in the localization $\Hqe$ of $\kat{dgCat}$ along quasi-equivalences can be described as isomorphism classes of \emph{quasi-functors} (see \cite{toen-morita} and \cite{canonaco-stellari-internalhoms}). Roughly speaking, quasi-functors are ``homotopy coherent dg-functors'', and they are defined as particular dg-bimodules.

\subsubsection{} Let $\cat A$ and $\cat B$ be dg-categories. An \emph{$\cat A$-$\cat B$-dg-bimodule} is a right $\cat B \otimes \opp{\cat A}$-dg-module, namely a dg-functor
\begin{equation*}
T \colon \opp{\cat B} \otimes \cat A \to \compdg{\cat k}.
\end{equation*}
We shall sometimes use the ``Einstein notation'', writing
\begin{equation*}
T(B,A)=T^B_A,
\end{equation*}
putting the contravariant variables above and the covariant ones below. We shall also write:
\begin{align*}
T_A &= T(-,A) \in \compdg{\cat B}, \\
T^B &= T(B,-) \in \compdg{\opp{\cat A}}.
\end{align*}
For any dg-category $\cat A$ we have the \emph{diagonal bimodule} $h=h_{\cat A} \in \compdg{\cat A \otimes \opp{\cat A}}$, defined by
\begin{equation*}
h^A_B = \cat A(A,B).
\end{equation*}
This notation is consistent with the chosen name $h_{\cat A}$ for the Yoneda embedding of $\cat A$: in fact, the Yoneda embedding is precisely the functor which maps
\[
A \mapsto h_{\cat A}(A) = \cat A(-,A) = h_A, \quad A \in \cat A.
\]
\subsubsection{} A \emph{quasi-functor} $T \colon \cat A \to \cat B$ between two dg-categories is a dg-bimodule $T \in \compdg{\cat B \otimes \opp{\cat A}}$ with the property of being \emph{right quasi-representable}, namely: for all $A \in \cat A$, there exists an object $F(A) \in \cat B$ such that $T_A \cong h_{F(A)}$ in $\dercomp{\cat B}$. From this, we see that a quasi-functor $T$ induces a genuine functor $H^0(T) \colon H^0(\cat A) \to H^0(\cat B)$. Two quasi-functors $T,S$ are \emph{isomorphic} if they are isomorphic in the derived category $\dercomp{\cat B \otimes \opp{\cat A}}$. As already said, isomorphism classes of quasi-functors can be identified with the morphisms in the homotopy category $\Hqe$ of dg-categories. From the general model-categorical machinery, we also know that a morphism $\cat A \to \cat B$ in $\Hqe$ can be represented by a dg-functor whenever the domain dg-category $\cat A$ is cofibrant; moreover, any dg-category $\cat A$ has a \emph{cofibrant replacement} $Q(\cat A)$ which comes with a quasi-equivalence $Q(\cat A) \to \cat A$.
\subsubsection{} There is a notion of \emph{adjunction of quasi-functors}, investigated in \cite{genovese-adjunctions}. Given two quasi-functors $T,S \colon \cat A \leftrightarrows \cat B$, we see that $T \dashv S$ if and only if there is an isomorphism in $\dercomp{\basering k}$
\begin{equation} \label{eq:qfunct_adjoints}
\compdg{\cat B}(Q(T)_A,h_B) \qis \compdg{\cat A}(h_A,S_B) \cong S^A_B,
\end{equation}
``natural'' in $A$ and $B$, in the precise sense that the bimodules
\begin{align*}
(A,B) & \mapsto \compdg{\cat B}(Q(T)_A,h_B), \\
(A,B) & \mapsto \compdg{\cat A}(h_A,S_B) \cong S^A_B
\end{align*}
are isomorphic in $\dercomp{\cat A \otimes \opp{\cat B}}$. Here $Q(T)$ is an h-projective resolution of $T$ as an $\cat A$-$\cat B$-dg-bimodule. Recall from \cite[Lemma 3.4]{canonaco-stellari-internalhoms} that in particular $Q(T)_A \in \hproj{\cat B}$ for all $A \in \cat A$.  It is worth mentioning that in case $T$ is such that $T_A \in \hproj{\cat B}$ for all $A \in \cat A$, there is an isomorphism in $\dercomp{\basering k}$
\begin{equation*}
\compdg{\cat B}(Q(T)_A,h_B) \qis \compdg{\cat B}(T_A,h_B),
\end{equation*}
``natural'' in $A$ and $B$, so the adjunction $T \dashv S$ is given by an isomorphism in $\dercomp{\basering k}$
\begin{equation*}
\compdg{\cat B}(T_A,h_B) \qis \compdg{\cat A}(h_A,S_B) \cong S^A_B.
\end{equation*}

\section{Homotopy colimits and t-structures}
\subsection{Homotopy colimits in triangulated categories} \label{subsec:hocolim_trcat} We start by recalling the notion of \emph{homotopy colimit} of a sequence in a fixed $\basering k$-linear triangulated category $\cat T$. We shall tacitly assume that any coproduct (direct sum) we write exists in $\cat T$.
\begin{defin} \label{def:holim_tr}
Let $(A_n \xrightarrow{j_{n,n+1}} A_{n+1})_{n \geq 0}$ be a sequence of maps in $\cat T$. The \emph{homotopy colimit} $\hocolim_n A_n$ is defined as the object (uniquely determined up to isomorphism) sitting in the following distinguished triangle:
\begin{equation}
\bigoplus_n A_n \xrightarrow{1-\mu} \bigoplus_n A_n \to \hocolim_n A_n,
\end{equation}
where $\mu$ is the map induced by
\begin{equation*}
A_n \xrightarrow{j_{n,n+1}} A_{n+1} \xrightarrow{\incl_{n+1}} \bigoplus_n A_n.
\end{equation*}

A \emph{homotopy limit} is defined as a homotopy colimit in $\opp{\cat T}$. Explicity, assume that every direct product we shall write exists in $\cat T$, and let $(A_{n+1} \xrightarrow{\pi_{n+1,n}} A_n)_{n \geq 0}$ be a sequence of maps in $\cat T$. Then, the homotopy limit $\holim_n A_n$ is defined as the object (uniquely determined up to isomorphism) sitting in the following distinguished triangle:
\begin{equation}
\holim_n A_n \to \prod_n A_n \xrightarrow{1-\nu} \prod_n A_n,
\end{equation}
where $\nu$ is the map induced by
\begin{equation*}
\prod_n A_n \xrightarrow{\pr_{n+1}} A_{n+1} \xrightarrow{\pi_{n+1,n}} A_n.
\end{equation*}
\end{defin}
In the following discussion we shall concentrate on homotopy colimits; changing $\cat T$ with $\opp{\cat T}$ gives the formal analogous facts about homotopy limits.

Being defined as $\cone(1-\mu)$ in $\cat T$, the homotopy limit is not functorial. Still, it satisfies a weak universal property involving existence but not unicity. First, there are natural maps $j_n \colon A_n \to \hocolim A_n$ such that the diagram
\begin{equation} \label{eq:hocolim_tr_inclusions}
\begin{gathered}
\xymatrix{
A_n \ar[r]^-{j_{n,n+1}} \ar[dr]_{j_n} & A_{n+1} \ar[d]^{j_{n+1}} \\ &  \hocolim_n A_n
}
\end{gathered}
\end{equation}
is commutative: these maps are just the components of the map $\oplus_n A_n \xrightarrow{\oplus j_n} \hocolim_n A_n$, and the above commutativity is equivalent to saying that the composition
\begin{equation*}
\bigoplus_n A_n \xrightarrow{1-\mu} \bigoplus_n A_n \xrightarrow{\oplus j_n} \hocolim_n A_n
\end{equation*}
is zero. Moreover, for any family of maps $f_n \colon A_n \to X$ such that the diagram
\begin{equation*}
\begin{gathered}
\xymatrix{
A_n \ar[r]^-{j_{n,n+1}} \ar[dr]_{f_n} & A_{n+1} \ar[d]^{f_{n+1}} \\ &  X
}
\end{gathered}
\end{equation*}
is commutative (that is, the composition
\begin{equation*}
 \bigoplus_n A_n \xrightarrow{1-\mu} \bigoplus_n A_n \xrightarrow{\oplus f_n} X 
\end{equation*}
is zero), there is a map $f \colon \hocolim_n A_n \to X$ such that $f_n = f \circ j_n $ for all $n$:
\begin{equation}
\begin{gathered}
\xymatrix{
A_n \ar[r]^-{f_n} \ar[d]_{j_n} & X \\ \hocolim_n A_n. \ar[ur]_{f}
}
\end{gathered}
\end{equation}
Such $f$ is obtained non-uniquely by observing that in the exact sequence
\begin{equation*}
\cat T(\hocolim_n A_n, X) \xrightarrow{{(\oplus j_n)}^*} \cat T(\bigoplus_n A_n, X) \xrightarrow{(1-\mu)^*} \cat T(\bigoplus_n A_n, X)
\end{equation*}
the element $(f_n) \in \cat T(\bigoplus_n A_n, X)$ is in $\ker (1-\nu)^* = \Img((\oplus j_n)^*)$.
\subsection{Homotopy colimits in dg-categories} Now, let $\cat A$ be a dg-category. First, we discuss \emph{strictly dg-functorial homotopy (co)limits} of dg-modules.
\begin{defin} \label{def:holim_dgmod}
Let $(M_{n+1} \xrightarrow{p_{n+1,n}} M_n)_{n \geq 0}$ be a sequence of closed degree $0$ maps in $\compdg{\cat A}$. Its \emph{(strictly dg-functorial) homotopy limit} is defined as the shifted mapping cone $\holim_n M_i = \cone(1-\nu)[-1]$, sitting in the following pretriangle:
\begin{equation*}
\holim_n M_n \to \prod_n M_n \xrightarrow{1-\nu} \prod_n M_n,
\end{equation*}
where $\nu$ is the (closed, degree $0$) map induced by
\begin{equation*}
\prod_n M_n \xrightarrow{\pr_{n+1}} M_{n+1} \xrightarrow{p_{n+1,n}} M_n.
\end{equation*}

Dually, let $(N_n \xrightarrow{j_{n,n+1}} N_{n+1})_{n \geq 0}$ be a sequence in $\compdg{\cat A}$. Its \emph{(strictly dg-functorial) homotopy colimit} is defined as the mapping cone $\hocolim_n N_n = \cone(1-\mu)$, sitting in the following pretriangle:
\begin{equation*}
\bigoplus_n N_n \xrightarrow{1-\mu} \bigoplus_n N_n \to \hocolim_n N_n,
\end{equation*}
where $\mu$ is the (closed, degree $0$) map induced by
\begin{equation*}
N_n \xrightarrow{j_{n,n+1}} N_{n+1} \xrightarrow{\incl_{n+1}}  \bigoplus_n N_n.
\end{equation*}
\end{defin} 
It is immediate to check that there are (strict) isomorphisms of complexes:
\begin{equation} \label{eq:strictholim_dgmod}
\begin{split}
\compdg{\cat A}(X, \holim_n M_n)  & \cong \holim_n \compdg{\cat A}(X,M_n), \\
\compdg{\cat A}(\hocolim_n N_n, X) & \cong \holim_n \compdg{\cat A}(N_n,X),
\end{split}
\end{equation}
both natural in $X \in \compdg{\cat A}$.

The homotopy (co)limits $\holim_n M_n$ and $\hocolim_n N_n$ are defined as mapping cones in $\compdg{\cat A}$, so we know how to describe maps to and from them. In particular, let $X \in \compdg{\cat A}$ and let
\begin{equation*}
\begin{gathered}
\xymatrix{
X  \ar[dr]^{(k_n)} \ar[d]^{(f_n)} \\
\prod_n M_n \ar[r]^{1-\nu} & \prod_n M_n,
}
\end{gathered}
\end{equation*}
be a diagram where $(f_n)_n$ is closed of degree $i$ and $(k_n)_n$ is of degree $i-1$ such that $d((k_n)_n) + (1-\nu) \circ (f_n) = 0$, namely
\begin{equation*}
dk_n = p_{n+1,n} \circ f_{n+1} - f_n.
\end{equation*}
Then, there is an induced closed degree $i$ morphism
\begin{equation}
\holim_n (f_n,k_n) \colon X \to \holim_n M_n.
\end{equation} 
Dually, let $Y \in \compdg{\cat A}$ and let 
\begin{equation*}
\begin{gathered}
\xymatrix{
\bigoplus_n N_n \ar[r]^{1-\mu} \ar[dr]^{\oplus l_n} & \bigoplus_n N_n \ar[d]^{\oplus g_n} \\
& Y,
}
\end{gathered}
\end{equation*}
be a diagram where $\oplus g_n$ is closed of degree $i$ and $\oplus l_n$ is of degree $i-1$ such that $d(\oplus l_n) = (-1)^i (\oplus g_n) \circ (1-\mu)$, namely:
\begin{equation*}
d l_n = (-1)^i (g_n - g_{n+1} \circ j_{n,n+1}).
\end{equation*}
Then, there is an induced degree $i$ morphism
\begin{equation} \label{eq:hocolim_inducedmap}
\hocolim_n (g_n,l_n) \colon \hocolim_n N_n \to Y.
\end{equation}

We can now give the following definition.
\begin{defin} \label{def:dg_hocolim}
Let $(A_n \xrightarrow{j_{n,n+1}} A_{n+1})_{n \geq 0}$ be a sequence of closed degree $0$ maps in $\cat A$, and let $\holim_n \cat A(A_n,-)$ be the strictly dg-functorial homotopy limit of the induced sequence $(\cat A(A_{n+1},-) \xrightarrow{j^*_{n,n+1}} \cat A(A_n,-))_n$ in $\compdg{\cat A}$. The \emph{(dg-functorial) homotopy colimit} of $(A_n \to A_{n+1})_n$ is an object $\hocolim_n A_n$ together with an isomorphism
\begin{equation*}
\cat A(\hocolim_n A_n,-) \to \holim_n \cat A(A_n,-)
\end{equation*}
in the derived category $\dercomp{\opp{\cat A}}$.
\end{defin}
\begin{remark} \label{remark:hocolim_univprop}
Let $(A_n \xrightarrow{j_{n,n+1}} A_{n+1})_{n \geq 0}$ be as in the above Definition \ref{def:dg_hocolim}, and let $B \in \cat A$. An element in $Z^i(\holim_n \cat A(A_n,B))$ is explicitly given by a family $(f_n,k_n)_{n \geq 0}$ where $f_n \colon A_n \to B$ is closed of degree $i$ and $k_n \colon A_n \to B$ is a ``homotopy'' of degree $i-1$ such that $d k_n=f_{n+1}j_{n,n+1} - f_n$ for all $n$. In other words, the diagram
\begin{equation*}
\begin{gathered}
\xymatrix{
A_n \ar[r]^-{j_{n,n+1}} \ar[dr]_{f_n} & A_{n+1} \ar[d]^{f_{n+1}} \\ &  B
}
\end{gathered}
\end{equation*}
is commutative up to $d k_n$. By the Yoneda lemma, giving $(f_n,k_n)_n \in Z^i(\holim_n \cat A(A_n,B))$ is the same as giving a closed degree $i$ morphism in $\compdg{\opp{\cat A}}$:
\begin{equation}
\holim_n (f^*_n,k^*_n) \colon \cat A(B,-) \to \holim_n \cat A(A_n,-).
\end{equation}

Again by the Yoneda lemma, note that we have natural isomorphisms:
\begin{align*}
\hocomp{\opp{\cat A}}(\cat A(B,-), \holim_n \cat A(A_n,-)[i]) & \cong \dercomp{\opp{\cat A}}(\cat A(B,-), \holim_n \cat A(A_n,-)[i])\\ 
& \cong H^i(\holim_n \cat A(A_n,B)),
\end{align*}
Hence, we can restate the definition of homotopy colimit as follow. The homotopy colimit of $(A_n \xrightarrow{j_{n,n+1}} A_{n+1})_{n \geq 0}$ is a pair
\begin{equation*}
(\hocolim_n A_n \in \cat A, [(j_n,h_n)_n] \in  H^0(\holim_n \cat A(A_n,\hocolim_n A_n)))
\end{equation*}
such that the induced map
\begin{equation*}
\holim_n (j^*_n,h^*_n) \colon \cat A(\holim_n A_n,-) \to \holim_n \cat A(A_n,-).
\end{equation*}
is a quasi-isomorphism. This means that whenever we are given $B \in \cat A$ with a class $[(f_n,k_n)_n] \in  H^i(\holim_n \cat A(A_n,B))$, there is a unique $[f] \in H^i(\cat A(\hocolim_n A_n, B))$ such that 
\begin{equation*}
[(f \circ j_n, f \circ h_n)_n] = [(f_n,k_n)_n].
\end{equation*}
 In other words, given $B \in \cat A$ and a morphism
\begin{equation*}
\holim_n (f^*_n,k^*_n) \colon \cat A(B,-) \to \holim_n \cat A(A_n,-)[i]
\end{equation*}
in $\dercomp{\opp{\cat A}}$, there is a unique $[f] \in H^i(\cat A(\hocolim_n A_n, B))$ such that the diagram
\begin{equation} \label{eq:hocolim_univprop_diagram}
\begin{gathered}
\xymatrix{
\cat A(B,-) \ar[r]^-{\holim_n (f^*_n,k^*_n)} \ar@{.>}[d]^-{f^*} & \holim_n \cat A(A_n,-) \\
\cat A(\hocolim_n A_n,-). \ar[ur]_{\holim_n (j^*_n,h^*_n)}
}
\end{gathered}
\end{equation}
is commutative in $\dercomp{\opp{\cat A}}$. In particular, notice that $[f \circ j_n] =[f_n]$ in $H^i(\cat A(A_n,B))$, for all $n$.
\end{remark}
Now, assume that $\cat A$ is a pretriangulated dg-category and let $(A_n \xrightarrow{j_{n,n+1}} A_{n+1})_{n \geq 0}$ be a sequence of closed degree $0$ maps in $\cat A$ such that the coproduct $\bigoplus_n A_n$ exists in $H^0(\cat A)$. Then, the homotopy colimit of $(A_n \xrightarrow{j_{n,n+1}} A_{n+1})_{n \geq 0}$ exists in $\cat A$. First, we note that the dg-module $\prod_n \cat A(A_n,-)$ is quasi-representable. Indeed, we have closed degree $0$ maps $\incl_n \colon A_n \to \bigoplus_n A_n$ such that
\begin{equation*}
H^0(\cat A)( \bigoplus_n A_n,-) \xrightarrow{( [\incl_n]^*)_n} \prod_n H^0(\cat A)(A_n,-)
\end{equation*}
is an isomorphism of left $H^0(\cat A)$-modules. Clearly, the maps $\incl_n$ induce a morphism in $\compdg{\opp{\cat A}}$:
\begin{equation*}
\cat A(\bigoplus_n A_n,-) \xrightarrow{(\incl_n^*)_n} \prod_n \cat A(A_n,-).
\end{equation*}
This is actually a quasi-isomorphism. By shifting, it is enough to check that it induces an isomorphism in $H^0$, and indeed
\begin{equation*}
H^0(\cat A(A_n,-)) \to H^0(\prod_n \cat A(A_n,-)) \xrightarrow{\sim} \prod_n H^0(\cat A(A_n,-))
\end{equation*}
is precisely the above $H^0(\cat A)( \bigoplus_n A_n,-) \xrightarrow{( [\incl_n]^*)_n} \prod_n H^0(\cat A)(A_n,-)$.

Now, since $\cat A$ is pretriangulated, the sequence $(A_n \xrightarrow{[j_{n,n+1}]} A_{n+1})_n$ in $H^0(\cat A)$ has a homotopy colimit $\hocolim_n A_n$ in the sense of Definition \ref{def:holim_tr}. Moreover, we have a diagram with distinguished rows in $\dercomp{\opp{\cat A}}$:
\begin{equation*}
\begin{gathered}
\xymatrix{
\cat A(\hocolim_n A_n,-) \ar[r] \ar@{.>}[d]^\sim & \cat A(\bigoplus_n A_n,-) \ar[r] \ar[d]^\sim & \cat A(\bigoplus_n A_n,-) \ar[d]^\sim \\
\holim_n \cat A(A_n,-) \ar[r] & \prod_n \cat A(A_n,-) \ar[r] & \prod_n \cat A(A_n,-).
}
\end{gathered}
\end{equation*}
From this, we get an isomorphism $\cat A(\hocolim_n A_n,-) \xrightarrow{\sim} \holim_n \cat A(A_n,-)$ in $\dercomp{\opp{\cat A}}$. We sum up what we found:
\begin{lemma} \label{lemma:dgcat_hocolim}
Let $\cat A$ be a pretriangulated dg-category, and let $(A_n \xrightarrow{j_{n,n+1}})_{n \geq 0}$ be a sequence of closed degree $0$ maps such that $\bigoplus_n A_n$ exists in $H^0(\cat A)$. Then, if $\hocolim_n A_n$ is the homotopy colimit of the sequence $(A_n \xrightarrow{[j_{n,n+1}]})_{n \geq 0}$ in $H^0(\cat A)$ in the sense of Definition \ref{def:holim_tr}, there is an isomorphism
\begin{equation*}
\cat A(\hocolim_n A_n,-) \to \holim_n \cat A(A_n,-)
\end{equation*}
in $\dercomp{\opp{\cat A}}$. In other words, $\cat A$ is closed under (dg-functorial) homotopy colimits which, as objects, are described by (non-functorial) homotopy colimits in the homotopy category.
\end{lemma} 
\subsection{t-structures and homotopy colimits} \label{subsection_tstruct_hocolim}
Now, let $\cat T$ be a triangulated category endowed with a t-structure $(\cat T_{\leq 0}, \cat T_{\geq 0})$, with heart $\cat T^\heartsuit$ (see \cite{beilinson-pervers} for the basic reference on t-structures). The zeroth cohomology functor given by the t-structure on $\cat T$ is denoted by
\begin{equation}
H^0(-) = \tau_{\leq 0} \tau_{\geq 0} \colon \cat T \to \cat T^\heartsuit,
\end{equation}
where $\tau_{\geq 0}$ and $\tau_{\geq 0}$ are the truncation functors. We also define the $i$-th cohomology:
\begin{equation*}
H^i(-) = H^0(-[i]) = \tau_{\leq i} \tau_{\geq i}.
\end{equation*}
We define full subcategories of $\cat T$ as follows:
\begin{align}
\cat T^- &= \bigcup_{n \geq 0} \cat T_{\leq n}, \\
\cat T^+ &= \bigcup_{n \geq 0} \cat T_{\geq -n}.
\end{align}
We have inclusions
\begin{equation} \label{eq:aisles_incl_cohom}
\begin{split}
\cat T_{\leq n} &\subseteq \{A \in \cat T : H^i(A) = 0 \quad \forall\, i > n \},  \\
\cat T_{\geq n} & \subseteq \{A \in \cat T : H^i(A)=0 \quad \forall\, i < n \}.
\end{split}
\end{equation}
We recall that an exact functor $F \colon \cat T \to \cat T'$ between triangulated categories with t-structures is \emph{t-exact} if it preserves the aisles: $F(\cat T_{\leq n}) \subseteq \cat T'_{\leq n}$ and $F(\cat T_{\geq n}) \subseteq \cat T'_{\geq n}$; equivalently, if it commutes with the truncation functors.
\begin{lemma} \label{lemma:Tplus_triang}
$\cat T^+$ and $\cat T^-$ are strictly full triangulated subcategories of $\cat T$ closed under direct summands, and the t-structure on $\cat T$ induces t-structures on $\cat T^+$ and $\cat T^-$ such that the inclusions $\cat T^+ \hookrightarrow \cat T$ and $\cat T^- \hookrightarrow \cat T$ are t-exact. In particular, the hearts of both $\cat T^+$ and $\cat T^-$ coincide with $\cat T^\heartsuit$.
\end{lemma}

We say that the t-structure $(\cat T_{\leq 0}, \cat T_{\geq 0})$ is \emph{right bounded} (or \emph{bounded from above}) if the inclusion $\cat T^- \hookrightarrow \cat T$ is an equivalence. Dually, we say that it is \emph{left bounded} (or \emph{bounded from below}) if the inclusion $\cat T^+ \hookrightarrow \cat T$ is an equivalence.

We say that the t-structure $(\cat T_{\leq 0}, \cat T_{\geq 0})$ is \emph{left separated} if 
\begin{equation}
\bigcap_{n \geq 0} \cat T_{\leq -n} = 0.
\end{equation} 
Dually, we say that it is \emph{right separated} if 
\begin{equation}
\bigcap_{n \geq 0} \cat T_{\geq n} = 0.
\end{equation}
Finally, we say that the t-structure $(\cat T_{\leq 0}, \cat T_{\geq 0})$ is \emph{non-degenerate} if it is both right and left separated. In this case (see \cite[Proposition 1.3.7]{beilinson-pervers}) we have that $H^i(A) = 0$ for all $i$ if and only if $A \cong 0$ in $\cat T$, and the inclusions \eqref{eq:aisles_incl_cohom} are actual equalities:
\begin{equation} \label{eq:tstruct_nondeg_aisles}
\begin{split}
\cat T_{\geq n} = \{ A \in \cat T : H^i(A) = 0 \quad \forall\, i <n \}, \\
\cat T_{\leq n} = \{ A \in \cat T : H^i(A) = 0 \quad \forall\, i >n \}.
\end{split}
\end{equation}
\begin{remark} 
The t-structure $(\cat T_{\leq 0}, \cat T_{\geq 0})$ is left separated if and only if the induced t-structure on $\cat T^-$ is non-degenerate.  Dually, the t-structure $(\cat T_{\leq 0}, \cat T_{\geq 0})$ is right separated if and only if the induced t-structure on $\cat T^+$ is non-degenerate.
\end{remark}
We are going to work with homotopy colimits inside triangulated categories with a t-structure; we now explain a useful assumption which ensures the existence of the homotopy colimits we shall need.
\begin{defin} \label{def:tstruct_closed_directsums}
We say that the t-structure $(\cat T_{\leq 0}, \cat T_{\geq 0})$ is \emph{closed under countable coproducts} if the aisle $\cat T_{\leq 0}$ is closed under countable coproducts.

Dually, we say that the t-structure $(\cat T_{\leq 0}, \cat T_{\geq 0})$ is \emph{closed under countable products} if $\cat T_{\geq 0}$ is closed under countable products.
\end{defin}
\begin{remark} \label{remark:tstruct_closed_directsums}
Being left adjoints, the inclusions $\cat T_{\leq M} \hookrightarrow \cat T$ are cocontinuous. Hence, the t-structure $(\cat T_{\leq 0}, \cat T_{\geq 0})$ is closed under countable coproducts if and only if any countable family of objects $\{A_n\}$ in $\cat T_{\leq 0}$ has a direct sum $\bigoplus_n A_n$ in $\cat T$.

Moreover, we can check that $\cat T_{\leq 0}$ is closed under countable coproducts if and only if $\cat T_{\leq M}$ is closed under countable coproducts for all $M \in \mathbb Z$. Indeed, if $\{A_n\}$ is a countable family in $\cat T_{\leq M}$, the shifted family $\{A_n[M]\}$ lies in $\cat T_{\leq 0}$, and assuming that $\bigoplus_n A_n[M]$ exists in $\cat T_{\leq 0}$, we immediately see that $(\bigoplus_n A_n[M])[-M]$ is the coproduct of the $A_n$ in $\cat T_{\leq M}$.

Clearly, the above discussion dualizes directly to t-structures which are closed under countable products.
\end{remark}

We now prove a lemma which describes the t-structure cohomology of homotopy colimits of particular sequences which are eventually constant in cohomology.
\begin{lemma} \label{lemma:tcohom_holim}
Let $\cat T$ be a triangulated category with a non-degenerate t-structure. Let 
\begin{equation*}
(X_{-k} \xrightarrow{j_{-k,-k-1}} X_{-k-1})_{k \geq 0}
\end{equation*}
be a sequence of maps in $\cat T$ and assume that the direct sum $\bigoplus_{k} X_{-k}$ exists in $\cat T$. Let $X= \hocolim_k X_{-k}$, together with the natural maps $j_{-k} \colon X_{-k} \to X$ (see \eqref{eq:hocolim_tr_inclusions}). If for all $n\geq 0$ the induced morphism
\begin{equation*}
H^i(j_{-n,-n-1}) \colon H^i(X_{-n}) \to H^i(X_{-n-1})
\end{equation*}
is an isomorphism for $i>-n$ and an epimorphism for $i=-n$, then for all $n\geq 0$ the induced morphism
\begin{equation*}
H^i(j_{-n}) \colon H^i(X_{-n}) \to H^i(X)
\end{equation*}
is an isomorphism for $i>-n$ and an epimorphism for $i=-n$.
\end{lemma}
\begin{proof}
Set $C_{-n,-n-1} = \cone(j_{-n,-n-1})[-1]$ and $C_{-n} = \cone(j_{-n})[-1]$. Then, the hypothesis is equivalent to
\begin{equation*}
\forall\, n\geq 0,\ H^i(C_{-n,-n-1}) = 0 \quad \forall\, i>-n,
\end{equation*}
and the thesis is equivalent to
\begin{equation*}
\forall\, n\geq 0,\ H^i(C_{-n}) = 0 \quad \forall\,  i > -n.
\end{equation*}
By \eqref{eq:tstruct_nondeg_aisles}, the hypothesis is
\begin{equation*}
\forall\, n\geq 0,\ \cat T(C_{-n,-n-1}, Z) = 0 \quad \forall\, Z \in \cat T_{>-n},
\end{equation*}
and the thesis is
\begin{equation*}
\forall\, n\geq 0,\ \cat T(C_{-n},Z) = 0 \quad \forall\, Z \in \cat T_{>-n}.
\end{equation*}

Let $n \geq 0$ and let $Z \in \cat T_{>-n}$.  For $-k < -n$, consider the exact sequence:
\begin{equation*}
\cat T(C_{-k,-k-1}[1], Z) \to \cat T(X_{-k-1}, Z) \to \cat T(X_{-k},Z) \to \cat T(C_{-k,-k-1}, Z).
\end{equation*}
By hypothesis, we have $\cat T(C_{-k,-k-1},Y)=0$ for all $Y \in \cat T_{>-k}$, in particular for $Y=Z \in \cat T_{>-n} \subseteq \cat T_{>-k}$; also $\cat T(C_{-k,-k-1}[1], Z) \cong \cat T(C_{-k,-k-1},Z[-1]) = 0$, for $Z[-1] \in \cat T_{>-n}$ like $Z$. So, $\cat T(X_{-k-1},Z) \to \cat T(X_{-k},Z)$ is an isomorphism, and we get a chain of isomorphisms
\begin{equation*}
\cdots \xrightarrow{\sim} \cat T(X_{-n-2}, Z) \xrightarrow{\sim} \cat T(X_{-n-1},Z) \xrightarrow{\sim} \cat T(X_{-n},Z).
\end{equation*}
Hence, $\cat T(X_{-n},Z)$ is the inverse limit $\varprojlim_k \cat T(X_{-k},Z)$, together with the maps $\cat T(X_{-n},Z) \to \cat T(X_{-k},Z)$ obtained composing the suitable morphisms ${j^*_{-i,-i-1}}$ or their inverses. Moreover, since the sequence $(\cat T(Z,X_{-k}))_k$ is definitely constant, the morphism
\begin{equation*}
\prod_k \cat T(X_{-k},Z) \xrightarrow{1-\nu} \prod_k \cat T(X_{-k},Z)
\end{equation*}
is surjective, and hence the following sequence (exhibiting $\cat T(X_{-n},Z)$ as the above inverse limit) is exact:
\begin{equation*}
0 \to \cat T(X_{-n},Z) \to \prod_k \cat T(X_{-k},Z) \xrightarrow{1-\nu} \prod_k \cat T(X_{-k},Z) \to 0,
\end{equation*}
where $\cat T(X_{-n},Z) \to \prod_k \cat T(X_{-k},Z)$ is induced by the above maps $\cat T(X_{-n},Z) \to \cat T(X_{-k},Z)$. On the other hand, we have the distinguished triangle
\begin{equation*}
\bigoplus_k X_{-k} \xrightarrow{1-\mu} \bigoplus_k X_{-k} \xrightarrow{\oplus j_{-k}} X.
\end{equation*}
Upon identifying $\prod_k \cat T(X_{-k},Z) = \cat T(\bigoplus_k X_{-k},Z)$, we get a commutative diagram with exact rows:
\begin{equation*}
\begin{gathered}
\xymatrix@C-=0.5cm{
\prod_k \cat T(X_{-k}[1],Z) \ar[r]& \prod_k \cat T(X_{-k}[1],Z) \ar[r] & \cat T(X,Z) \ar[d]^{j^*_{-n}} \ar[r] & \prod_k \cat T(X_{-k},Z) \ar@{=}[d] \ar[r] & \prod_k \cat T(X_{-k}, Z) \ar@{=}[d] \\
&  0 \ar[r] & \cat T(X_{-n},Z) \ar[r] & \prod_k \cat T(X_{-k},Z) \ar[r] & \prod_k \cat T(X_{-k},Z).
}
\end{gathered}
\end{equation*}
Notice that $\cat T(X_{-k}[1],Z) \cong \cat T(X_{-k},Z[-1])$ and $Z[-1] \in \cat T_{>-n}$ since $Z \in \cat T_{>-n}$, therefore the morphism $\prod_k \cat T(X_{-k}[1],Z) \to \prod_k \cat T(X_{-k}[1],Z)$ is surjective, so the map $\prod_k \cat T(X_{-k}[1],Z) \to \cat T(X,Z)$ is the zero map, and $\cat T(X,Z) \to \prod_k \cat T(X_{-k},Z)$ is monic. Then, by exactness, that is a kernel of $\prod_k \cat T(X_{-k},Z) \to \prod_k \cat T(X_{-k},Z)$, just as $\cat T(X_{-n},Z) \to \prod_k \cat T(X_{-k},Z)$. We deduce that
\begin{equation} \label{eq:j_n_iso}
j^*_{-n} \colon \cat T(X,Z) \to \cat T(X_{-n},Z)
\end{equation}
is an isomorphism. This is true for all $n \geq 0$ and for all $Z \in \cat T_{>-n}$. In particular, we also have that
\begin{equation} \label{eq:j_n+1_iso}
j^*_{-n-1} \colon \cat T(X,Z[1]) \to \cat T(X_{-n-1},Z[1])
\end{equation}
is an isomorphism for the same $n$ and $Z$, since $Z[1] \in \cat T_{>-n-1}$.

Next, we consider the following commutative diagram with exact rows:
\begin{equation*}
\begin{gathered}
\xymatrix@C-=0.5cm{
\cat T(X,Z) \ar[r]^{j^*_{-n}} \ar[d]^{j^*_{-n-1}} & \cat T(X_{-n},Z) \ar@{=}[d] \ar[r] & \cat T(C_{-n},Z) \ar[d] \ar[r] & \cat T(X[-1],Z) \ar[d]^{{j_{-n-1}[-1]}^*} \ar[r] & \cat T(X_{-n}[-1],Z) \ar@{=}[d] \\
\cat T(X_{-n-1},Z) \ar[r]^\sim & \cat T(X_{-n},Z) \ar[r] & \cat T(C_{-n,-n-1},Z) \ar[r] & \cat T(X_{-n-1}[-1],Z) \ar[r] & \cat T(X_{-n}[-1],Z),
}
\end{gathered}
\end{equation*}
where the map $\cat T(C_{-n},Z) \to \cat T(C_{-n,-n-1},Z)$ is induced by
\begin{equation*}
\begin{gathered}
\xymatrix{
C_{-n,-n-1} \ar@{.>}[d] \ar[r] & X_{-n} \ar@{=}[d] \ar[r]^-{j_{-n,-n-1}} & X_{-n-1} \ar[d]^{j_{-n-1}}  \\
C_{-n} \ar[r] & X_{-n} \ar[r]^{j_{-n}}  & X.
}
\end{gathered}
\end{equation*}
Since $Z \in \cat T_{>-n} \subseteq \cat T_{>-n-1}$, $j^*_{-n-1}$ is an isomorphism (see \eqref{eq:j_n_iso} above). Moreover, \eqref{eq:j_n+1_iso} holds and $j_{-n-1}[-1]^*$ is an isomorphism. By the five lemma, we conclude that
\begin{equation*}
\cat T(C_{-n},Z) \to \cat T(C_{-n,-n-1},Z)
\end{equation*}
is an isomorphism, hence $\cat T(C_{-n},Z) = 0$ as required.
\end{proof}
\begin{coroll} \label{coroll:tcohom_holim}
Let $\cat T$ be a triangulated category with a non-degenerate t-structure. Assume we are given a sequence $(X_{-k} \xrightarrow{j_{-k,-k-1}} X_{-k-1})_{k \geq 0}$ such that $\bigoplus_{k} X_{-k}$ exists in $\cat T$. Assume moreover that for all $n \geq 0$ $H^i(j_{-n,-n-1}) $ is an isomorphism for $i > - n$ and an epimorphism for $i=n$, as in Lemma \ref{lemma:tcohom_holim}. Next, let $Y \in \cat T$ be an object and let $f_{-n} \colon X_{-n} \to Y$ be maps such that the diagram
\begin{equation*}
\begin{gathered}
\xymatrix{
X_{-n} \ar[d]_{j_{-n,-n-1}} \ar[dr]^-{f_n} \\
X_{-n-1} \ar[r]_-{f_{-n-1}} & Y 
}
\end{gathered}
\end{equation*}
is commutative for all $n \geq 0$. Moreover, assume that for all $i \in \mathbb Z$, the induced maps
\begin{equation*}
H^i(f_{-n}) \colon H^i(X_{-n}) \to H^i(Y)
\end{equation*}
are isomorphisms for all $n > M(i)$ sufficiently large. Then, any morphism $f \colon \hocolim_n X_{-n} \to Y$ satisfying $f \circ j_{-n} = f_{-n}$ is such that
\begin{equation*}
H^i(f) \colon H^i(\hocolim_n X_{-n}) \to H^i(Y)
\end{equation*}
is an isomorphism for all $i \in \mathbb Z$. In particular, $f \colon \hocolim_n X_{-n} \to Y$ is an isomorphism in $\cat T$.
\end{coroll}
\begin{proof}
Fix $i \in \mathbb Z$. By Lemma \ref{lemma:tcohom_holim} we know that
\begin{equation*}
H^i(j_{-n}) \colon H^i(X_{-n}) \to H^i(\hocolim_n X_{-n})
\end{equation*}
is an isomorphism for all $-n < i$. So, take $n > \max(M(i),-i)$ sufficiently large, so that 
\begin{equation*}
H^i(f_{-n}) \colon H^i(X_{-n}) \to H^i(Y)
\end{equation*}
is also an isomorphism. Since we have by hypothesis
\begin{equation*}
H^i(f) \circ H^i(j_{-n}) = H^i(f_{-n}),
\end{equation*}
we conclude that $H^i(f)$ is an isomorphism, as desired.
\end{proof}
\begin{remark} \label{remark:hocolim_aisle}
Let $(X_{-k} \xrightarrow{j_{-k,-k-1}} X_{-k-1})_{k \geq 0}$ be a sequence such that $X_{-k} \in \cat T_{\leq M}$ for all $k\geq 0$, for some $M \in \mathbb Z$, and assume that the t-structure  $\bigoplus_k X_{-k}$ exists in $\cat T$. then, $\bigoplus_k X_{-k}$ and $\hocolim_k X_{-k}$ lie in $\cat T_{\leq M}$. Indeed, given $Z \in \cat T_{\geq M+1}$, we first have:
\begin{equation*}
\cat T(\bigoplus_k X_{-k}, Z) = \prod_k \cat T(X_{-k}, Z) = 0.
\end{equation*}
Moreover, the space $\cat T(\hocolim_k X_{-k}, Z)$ sits in the following exact sequence:
\begin{equation*}
\cat T(\bigoplus_k X_{-k}, Z[-1]) \to \cat T(\hocolim_k X_{-k}, Z) \to \cat T(\bigoplus_k X_{-k}, Z),
\end{equation*}
and by hypothesis we find out that $\cat T(\hocolim_k X_{-k}, Z) =0$.
\end{remark}
\begin{remark} \label{remark:tcohom_holim_leftsep_applicable}
Let $\cat T$ be closed under countable coproducts.  The above discussion actually shows that the t-structure $(\cat T_{\leq 0}, \cat T_{\geq 0})$ is closed under countable coproducts (Definition \ref{def:tstruct_closed_directsums}). Now, assume that $\cat T$ is closed under countable coproducts and has a left separated t-structure. Then, we deduce that $\cat T^-$ has a non-degenerate t-structure which is closed under countable coproducts. In particular, Lemma \ref{lemma:tcohom_holim} and Corollary \ref{coroll:tcohom_holim} can be applied in $\cat T^-$ to sequences $(X_{-k} \xrightarrow{j_{-k,-k-1}} X_{-k-1})_{k \geq 0}$ such that $X_{-k} \in \cat T_{\leq M}$ for all $k\geq 0$, for some $M \in \mathbb Z$.
\end{remark}

\section{Bounded above twisted complexes}
\label{sec:twcom}
Twisted complexes on dg-categories were introduced in \cite{bondal-kapranov}. In this section we present a slightly different flavour of this notion which brings us to the definition of \emph{bounded above twisted complexes on a dg-category with cohomology in nonpositive degrees}.
\subsection{The dg-category of Maurer-Cartan objects}
In this subsection we fix a $\basering k$-linear dg-category $\cat A$.
\begin{defin}
Let $\cat B \subseteq \compdg{\cat A}$ be a full dg-subcategory. We define the \emph{dg-category $\MC(\cat B)$ of Maurer-Cartan objects of $\cat B$}  as follows:
\begin{itemize}
\item Objects of $\MC(\cat B)$ are pairs $(M,q)$, where $M$ is an object of $\cat B$, and $q \colon M \to M$ is a degree $1$ morphism such that $dq + q^2 = 0$.
\item A degree $p$ morphism $f \colon (M,q) \to (M',q')$ is a degree $p$ morphism $M \to M'$ in $\cat B$. The differential of $f$ is defined by:
\begin{equation}
d_{\MC(\cat B)}(f) = d_{\cat B}f + q'f - (-1)^p fq.
\end{equation}
\end{itemize}
\end{defin}
It is easy to check that $\MC(\cat B)$ is indeed a dg-category. If $\cat B$ is $\mathcal U$-small, then also $\MC(\cat B)$ is $\mathcal U$-small. There is a \emph{totalisation} dg-functor
\begin{equation} \label{eq:totalisation}
\Tot \colon \MC(\cat B) \to \compdg{\cat A},
\end{equation}
defined as follows:
\begin{align*}
\Tot(M,q) &= (M, d_{M} + q), \\
\Tot((M,q) \xrightarrow{f} (M',q')) &= (M, d_M + q) \xrightarrow{f} (M',d_{M'} + q').
\end{align*}
In other words, an object $(M,q)$ is mapped to the dg-module whose underlying graded module is the same as $M$ but with differential changed to $d_M + q$; a morphism $f \colon (M,q) \to (M',q')$ is mapped to itself viewed as a morphism of dg-modules $(M,d_M + q) \to (M', d_{M'} + q')$. We can check:
\begin{lemma}
The above definition gives a well-defined dg-functor $\Tot$ which is fully faithful.
\end{lemma}
\begin{prop}
Assume that $\cat B \subseteq \compdg{\cat A}$ is closed under taking shifts and finite direct sums. Then, $\MC(\cat B)$ is a strongly pretriangulated dg-category.
\end{prop}
\begin{proof}
Given an object $(M,q) \in \MC(\cat B)$, its $n$-shift is given by
\begin{equation}
(M,q)[n] = (M[n], (-1)^n q[n]).
\end{equation}

Given a closed and degree $0$ morphism $f \colon (M,q) \to (M',q')$, its cone is given by
\begin{equation}
\cone(f) = (M[1] \oplus M', \begin{psmallmatrix} -q[1] & 0 \\ f \shiftid{M}{1}{0} & q' \end{psmallmatrix}). \qedhere
\end{equation}
\end{proof}
\subsection{Bounded above twisted complexes}
Here we fix a dg-category $\cat A$ whose cohomology is concentrated in nonpositive degrees, namely:
\begin{equation*}
H^i(\cat A(A,B))=0, \quad \forall\, i>0, \quad \forall\, A,B \in \cat A.
\end{equation*}
We identify $\cat A$ with its image under the Yoneda embedding $\cat A \hookrightarrow \compdg{\cat A}$. Let us denote by $\cat A^\oplus$ the closure of $\cat A$ under finite direct sums and zero objects (in $\compdg{\cat A}$). We denote by $\cat A^\leftarrow$ the full dg-subcategory of $\compdg{\cat A}$ whose objects are given by direct sums
\begin{equation}
\bigoplus_{i \in \mathbb Z} A_i[-i],
\end{equation}
where $A_i \in \cat A^\oplus$, and $A_i = 0$ for $i \gg 0$. The dg-category $\cat A^\leftarrow$ is clearly $\mathcal U$-small.
\begin{defin} \label{def:twcompl_onesided}
A \emph{(bounded above) one-sided twisted complex} on $\cat A$ is a pair $(\oplus_{i \in \mathbb Z} A_i[-i], q) \in \MC(\cat A^\leftarrow)$, where $q= (q_i^j \colon A_i[-i] \to A_j[-j])_{i,j \in \mathbb Z}$ is such that $q_i^j=0$ whenever $j \leq i$. 

A \emph{one-sided morphism of degree $p$} 
\begin{equation*}
f \colon (\oplus_{i \in \mathbb Z} A_i[-i], q) \to (\oplus_{i \in \mathbb Z} B_i[-i], q')
\end{equation*}
between one-sided twisted complexes is a morphism $f=(f_i^j \colon A_i[-i] \to B_j[-j])_{i,j \in \mathbb Z}$ of degree $p$ in $\MC(\cat A^\leftarrow)$ such that $f_i^j=0$ whenever $i-j+p >0$.

The dg-subcategory of (bounded above) one-sided twisted complexes and one-sided morphisms in $\MC(\cat A^\leftarrow)$ is denoted by $\Tw^-(\cat A)$. 

The full dg-subcategory of $\Tw^-(\cat A)$ whose objects are the \emph{bounded} one-sided twisted complexes, namely the objects of the form $(\bigoplus_i A_i[-i],q)$ with $A_i=0$ for $i \gg 0$ or $i \ll 0$, is denoted by $\Tw^-_\bd(\cat A)$.
\end{defin}
It is easily checked that the identity morphisms of one-sided twisted complexes are one-sided, and that the composition of one-sided morphisms is one-sided. Hence, $\Tw^-(\cat A)$ is actually a well-defined (non-full) $\mathcal U$-small dg-subcategory of $\MC(\cat A^\leftarrow)$. Notice that if $A \in \cat A$, then the object $(A,0)$ is a one-sided twisted complexes ($A$ lying in degree $0$). We shall often abuse notation and identify
\begin{equation} \label{eq:twcompl_inclusion}
A=(A,0) \in \Tw^-(\cat A).
\end{equation}
Moreover, $\Tot(A,0)=\Tot(A)$ is precisely the representable dg-module $\cat A(-,A)$.

Now, we would like to describe more explicitly the objects and morphisms in the dg-category $\Tw^-(\cat A)$. The idea is that a (one-sided) twisted complex $(\oplus A_i[-i],q)$ should be a complex with a ``twisted differential'' $q$, the object $A_i$ sitting in degree $i$. This involves just some care with sign conventions. We sum everything up in the following remark, leaving it to the reader to fill in the details.
\begin{remark}
An object $X \in \Tw^-(\cat A)$ can be viewed as a pair $(A_i, q_i^j)$, where $A_i \in \cat A^\oplus$, $A_i=0$ for $i \gg 0$ and $q_i^j \colon A_i \to A_j$ is a morphism of degree $i-j+1$ for all $i,j \in \mathbb Z$, such that the following identity holds:
\begin{equation}
(-1)^j dq_i^j + q^j_k q^k_i = 0,
\end{equation}
adopting the Einstein summation convention:
\begin{equation*}
q^j_k q^k_i = \sum_k q^j_k q^k_i.
\end{equation*}

A (not necessarily one-sided) morphism $f \colon (A_i, q_i^j) \to (A'_i, {q'}_i^j)$ of degree $p$ can be viewed as a matrix of morphisms $f_i^j \colon A_i \to A_j$, where $f_i^j$ has degree $i-j+p$. $f$ is one-sided if by definition we have $f_i^j=0$ if $i-j+p>0$. The differential of $f$ is given by:
\begin{equation}
(df)_i^j = (-1)^j df_i^j + {q'}_k^j f_i^k - (-1)^p f_k^j q_i^k.
\end{equation}
Notice that
\begin{equation*}
\Hom(\oplus A_i[-i], \oplus A'_j[-j]) \cong \prod_i \bigoplus_j \Hom(A_i[-i],A_j[-j]),
\end{equation*}
using the universal property of the direct sum and the Yoneda Lemma, so the matrix $(f_i^j)$ is such that, for any $i$, the terms $(f_i^j)_j$ of the $i$-th column are almost all zero. The same is true for the matrices $(q_i^j)$ and $({q'}_i^j)$, hence the sums ${q'}^j_k f^k_i$ and $f^j_k q^k_j$ are actually finite.

Given a closed degree $0$ map $f \colon (A_i, q^j_i) \to (B_i,r^j_i)$ of twisted complexes, its cone $\cone(f)$ can be described as the twisted complex
\begin{equation}
(A_{i+1} \oplus B_i, \begin{psmallmatrix} -q^{j+1}_{i+1} & 0 \\ f^j_{i+1} & r^j_i \end{psmallmatrix}).
\end{equation}
\end{remark}
\begin{lemma}
$\Tw^-(\cat A)$ is a strongly pretriangulated subcategory of $\MC(\cat A^\leftarrow)$.
\end{lemma}
\begin{proof}
We only need to check that, given a closed degree $0$ morphism in $\Tw^-(\cat A)$, namely a one-sided morphism $f \colon Q \to R$ between one-sided twisted complexes, the pretriangle in $\MC(\cat A^\leftarrow)$
\begin{equation*}
Q \xrightarrow{f} R \to \cone(f) \to Q[1]
\end{equation*}
lies in $\Tw^-(\cat A)$. But this is immediate.
\end{proof}
The reason why we defined $\Tw^-(\cat A)$ using one-sided morphisms is that the cone of a morphism between one-sided twisted complexes is not in general a one-sided twisted complex, unless this morphism is itself one-sided. The further requirement that $\cat A$ has cohomology concentrated in nonpositive degrees ensures that we are not really losing any relevant information, as we see in the following result.
\begin{prop}
Let $\cat A$ be a dg-category with cohomology concentrated in nonpositive degrees. Then, the inclusion functor
\begin{equation*}
\Tw^-(\cat A) \to \MC(\cat A^\leftarrow)
\end{equation*}
is quasi-fully faithful. In particular, the totalisation functor
\begin{equation} \label{eq:tot_twmin_def}
\Tot \colon \Tw^-(\cat A) \to \compdg{\cat A}
\end{equation}
is quasi-fully faithful.
\end{prop}
\begin{proof}
The totalisation $\MC(\cat A^\leftarrow) \to \compdg{\cat A}$ is fully faithful and both $\Tw^-(\cat A)$ and $\MC(\cat A^\leftarrow)$ are strongly pretriangulated, so we only need to show the following claims, for two given one-sided twisted complexes $Q, R \in \Tw^-(\cat A)$:
\begin{enumerate}
\item \label{proof:twmin_faithful} Given a closed degree $0$ one-sided morphism $f \colon Q \to R$, if $f= d\alpha$ for some (non necessarily one-sided) degree $-1$ morphism $\alpha \colon Q \to R$, then there exists a one-sided degree $-1$ morphism $\beta \colon Q \to R$ such that $f=d\beta$;
\item \label{proof:twmin_full} For any closed and degree $0$ (not necessarily one-sided) morphism $f \colon Q \to R$, there is a degree $-1$ morphism $\alpha \colon Q \to R$ such that $f-d\alpha$ is a one-sided morphism.
\end{enumerate}
Both claims follow from the following technical Lemma \ref{lemma:morphism_onesided_modify}.
\end{proof}
\begin{lemma} \label{lemma:morphism_onesided_modify}
Let $\cat A$ be a dg-category with cohomology concentrated in nonpositive degrees. Let $f \colon Q \to R$ be a (non necessarily one-sided) degree $p$ morphism between bounded above one-sided twisted complexes $Q = (Q_i, q^i_j), R = (R_i,r^i_j) \in \Tw^-(\cat A)$. Assume that the differential $df$ is one-sided. Then, there exists $\alpha \colon P \to Q$ of degree $p-1$ such that $f-d\alpha$ is one-sided.
\end{lemma}
\begin{proof}
By shifting, we can assume without loss of generality that $f$ has degree $0$ and that $Q_i=R_i=0$ for all $i>0$. For all $i>0$, we define $n_i= i-1$, and for $i \leq 0$ we define recursively:
\[
n_i=\min\{i-1, n_{i+1}, k : f_i^k \neq 0 \} \in \mathbb Z.
\]
Notice that $n_i < i$ and $n_{i+1} \leq n_i$ for all $i$.

For all $i > 0$ and for all $k \in \mathbb Z$, set $\alpha_i^k=0$. This verifies the (empty) conditions:
\begin{align*}
f_i^k &= (-1)^k d\alpha_i^k + r^k_s \alpha^s_i + \alpha^k_s q^s_i \quad \text{if $k < i$}, \\
\alpha^k_i &= 0 \quad \text{if $k \geq i$ or $k < n_i$}.
\end{align*}
Now, let $i \geq 0$. Assume recursively that we have defined $\alpha_j^k$ for all $j \geq i+1$ and for all $k \in \mathbb Z$, such that:
\begin{equation} \label{eq:conditions_alpha}
\begin{split}
f^k_j &= (-1)^k d\alpha^k_j + r^k_s \alpha^s_j + \alpha^k_s q^s_j \quad \text{if $k < j$}, \\
\alpha^k_j &= 0 \quad \text{if $k \geq j$ or $k < n_j$}.
\end{split}
\end{equation}
We are going to define $\alpha_i^k$ so that the conditions \eqref{eq:conditions_alpha} are satisfied. First, we set $\alpha^k_i = 0$ for all $k \geq i$ and for all $k < n_i$. Then, we let $k <i$ and we define $\alpha_i^k$ recursively. For the base step, we define $\alpha_i^{n_i}$; using that $(df)_i^{n_i} = 0$ by hypothesis since $n_i < i$, we compute:
\begin{equation*}
0 = (df)_i^{n_i} = (-1)^{n_i} df_i^{n_i} + \sum_{s < n_i} r^{n_i}_s f^s_i - \sum_{s >i} f^{n_i}_s q^s_i 
\end{equation*}
We have written explicit summation symbols for the sake of clarity: by construction $\sum_{s < n_i} r^{n_i}_s f^s_i$ vanishes, and since $s > n_s \geq n_i$ for $s>i$, we can apply the inductive hypothesis to $\sum_{s >i} f^{n_i}_s q^s_i$, finding:
\begin{equation*}
0 = (-1)^{n_i} df^{n_i}_i - (-1)^{n_i} \sum_{s > i, s < n_i} d\alpha^{n_i}_s q^s_i - \sum_{s > i, t < n_i} (r^{n_i}_t \alpha^t_s q^s_i + \alpha^{n_i}_t q^t_s q^s_i).
\end{equation*}
Next, $\sum_{s > i, t < n_i} r^{n_i}_t \alpha^t_s q^s_i$ vanishes because if $s>i$ and $t < n_i$ then $t < n_s$, and by definition $\alpha_s^t=0$. We resume forgetting the summation symbols and we note that $- q^t_s q^s_i = (-1)^t dq_i^t$, hence we go on:
\begin{align*}
0 &= (-1)^{n_i} df_i^{n_i} + (-1)^{n_i+1} d\alpha_s q^s_i +(-1)^t \alpha^{n_i}_t dq^t_i \\
&= (-1)^{n_i}df_i^{n_i} +(-1)^{n_i+1} d(\alpha^{n_i}_t q^t_i).
\end{align*}
Finally, we have that $(-1)^{n_i}f_i^{n_i} + (-1)^{n_i+1} \alpha^{n_i}_t q^t_i$ is a cocycle of positive degree. Since $\cat A$ ha cohomology concentrated in nonpositive degrees, this is a coboundary: we find $\alpha_i^{n_i}$ such that
\begin{equation*}
f_i^{n_i} = (-1)^{n_i} d\alpha_i^{n_i} + \alpha^{n_i}_s q^s_i = (-1)^{n_i} d\alpha_i^{n_i} + r^{n_i}_s \alpha^s_i + \alpha^{n_i}_s q^s_i.
\end{equation*}
For the inductive step, we assume we have defined the required $\alpha_i^h$ for $h=n_i, n_i+1,\ldots, k-1$ satisfying \eqref{eq:conditions_alpha}, and we define $\alpha_i^k$ which satisfies the analogue conditions. The technique is similar to the one used for the base step, and it is left to the reader.

At the end of this process, we get a degree $-1$ morphism $\alpha \colon Q \to R$. By constructions, it is immediate to see that
\begin{equation*}
f_i^j - (d\alpha)_i^j = 0
\end{equation*}
whenever $j<i$. Namely, $f-d\alpha$ is one-sided.
\end{proof}
\begin{remark} \label{remark:Twmin_functorial}
The construction $\cat A \mapsto \Tw^-(\cat A)$ is functorial in $\cat A$. Namely, given a dg-functor $F \colon \cat A \to \cat B$ between dg-categories with cohomology concentrated in negative degrees, there is a functorially induced dg-functor $\Tw^-(F) \colon \Tw^-(\cat A) \to \Tw^-(\cat B)$, defined as follows:
\begin{equation}
\begin{split}
\Tw^-(F)&(\oplus A_i[-i], q) = (\oplus F(A_i)[-i], F(q)), \\
\Tw^-(F)&((\oplus A_i[-i], q) \xrightarrow{f} (\oplus B_j[-j], r)) \\
&= (\oplus F(A_i)[-i], F(q)) \xrightarrow{F(f)} (\oplus F(B_j)[-j], F(r)),
\end{split}
\end{equation}
where we abused notation a little, identifying $F$ with its extension to $\cat A^\oplus$. The above definition is good since clearly $\Tw^-(F)$ maps (bounded above) one-sided twisted complexes and one-sided morphisms to (bounded above) one-sided twisted complexes and one-sided morphisms.
\end{remark}
\subsection{Twisted complexes and colimits} \label{subsec:twcompl_colim}
We shall work again with a dg-category $\cat A$ with cohomology concentrated in nonpositive degrees. We first convince ourselves that ``stupid truncations'' are well defined for twisted complexes.
Let $X=(\oplus_i A_i[-i],q) \in \Tw^-(\cat A)$. For all $n \in \mathbb Z$, define 
\begin{equation}
\sigma_{\geq n} X = (\oplus_{i \geq n} A_i[-i], \sigma_{\geq n} q) \in \Tw^-_\bd (\cat A),
\end{equation}
where $\sigma_{\geq n} q$ is obtained from $q$ by restriction:
\begin{equation}
(\sigma_{\geq n} q)_i^j = q_i^j, \quad i,j \geq n.
\end{equation}
We easily see that
\begin{align*}
(-1)^j d((\sigma_{\geq n} q)^j_i) & + (\sigma_{\geq n} q)_k^j (\sigma_{\geq n} q)_i^k \\
&= (-1)^j d(q_i^j) + q^j_k q^k_i = 0
\end{align*}
if $i,j \geq n$ (since $q$ is ``one-sided'', the sum is over $k \geq i$).

There are natural (closed, degree $0$) inclusions in $\Tw^-(\cat A)$:
\begin{equation} 
\varphi_n \colon \sigma_{\geq n} X \to X, \quad \varphi_{n,n-1} \colon \sigma_{\geq n} X \to \sigma_{\geq n-1} X
\end{equation}
such that the following diagram is commutative:
\begin{equation} \label{eq:inclusion_truncations}
\begin{gathered}
\xymatrix{
\sigma_{\geq n} X \ar[rr]^{\varphi_{n,n-1}} \ar[dr]_{\varphi_{n-1}} & & \sigma_{\geq n-1}X \ar[dl]^{\varphi_n} \\
& X.
}
\end{gathered}
\end{equation}
For the underlying graded modules, these maps are just the inclusions:
\begin{equation*}
\bigoplus_{i \geq n} A_i[-i] \to \bigoplus_{i \geq n-1} A_i[-i] \to \cdots \to \bigoplus_i A_i[-i].
\end{equation*}
They are clearly of degree $0$. Let us verify, for example, that $\varphi_n$ is closed. We compute directly:
\begin{align*}
d(\varphi_n)_i^j & = (-1)^j d(({\varphi_n})_i^j) + (\sigma_{\geq n-1} q)_k^j ({\varphi_n})_i^k - ({\varphi_n})^j_k (\sigma_{\geq n} q)^k_i \\
&= (\sigma_{\geq n-1} q)_i^j ({\varphi_n})_i^i - ({\varphi_n})^j_k (\sigma_{\geq n} q)^k_i.
\end{align*}
Now, if both $i,j \geq n$, we have $({\varphi_n})_i^i = 1, ({\varphi_n})_j^j = 1$ and $(\sigma_{\geq n-1}q)_i^j = (\sigma_{\geq n} q)_i^j = q_i^j$, so $d(\varphi_n)_i^j = 0$. On the other hand, if $i < n$, we have $({\varphi_n})_i^i = 0$ and $(\sigma_{\geq n} q)_i^j=0$, so the above expression is $0$; instead, if $j< n$ we have $({\varphi_n})^j_j = 0$, and if $i<j$ also $({\varphi_n})_i^i=0$, whereas if $i \geq j$ then $(\sigma_{\geq n-1}q)_i^j = 0$. So, in any case $(d\varphi_n)_i^j = 0$, as claimed.
\begin{remark} \label{remark:truncation-cone}
Given $X=(\oplus_i A_i[-i], q)$ as above and $n \in \mathbb Z$, there is a closed degree $0$ map
\begin{equation}
\beta_n \colon A_{n-1}[-n] \to \sigma_{\geq n} X
\end{equation}
in $\Tw^-(\cat A)$, simply defined using the twisted differentials of $X$:
\begin{equation*}
({\beta_n})_n^j = q_{n-1}^j, \quad j > n-1.
\end{equation*}
By definition it has degree $0$, and
\begin{align*}
(d\beta_n)_n^j &= (-1)^j d(q_{n-1}^j) + q^j_k (\beta_n)_n^k \\
&= (-1)^j d(q^j_{n-1}) + q^j_k q^k_{n-1} = 0.
\end{align*}
By definition
\begin{equation*}
\cone(\beta_n) = (A_{n-1}[-n+1] \oplus \sigma_{\geq n} X, \begin{psmallmatrix} 0 & 0 \\ \beta_n \shiftid{A_{n-1}}{-n+1}{-n} & \sigma_{\geq n} q \end{psmallmatrix}) = \sigma_{\geq n-1} X.
\end{equation*}
The map $\varphi_{n,n-1}$ is precisely the natural inclusion $\sigma_{\geq n} X \to \cone(\beta_n)$. In other words, for any $n \in \mathbb Z$, there is a pretriangle in $\Tw^-(\cat A)$:
\begin{equation} \label{eq:twtruncation_triangle}
A_{n-1}[-n] \xrightarrow{\beta_n} \sigma_{\geq n} X \xrightarrow{\varphi_{n,n-1}} \sigma_{\geq n-1}X \to A_{n-1}[-n+1].
\end{equation}
\end{remark}
The twisted complex $X$ can be reconstructed from the truncations $\sigma_{\geq n} X$ by taking the colimit. More precisely, we have the following:
\begin{prop} \label{prop:tot_twmin_colim}
Let $n \in \mathbb Z$. The totalisation $\Tot(X)$, with the maps $(\Tot(\varphi_{n-p}))_{p \geq 0}$, is the colimit of $(\Tot(\sigma_{\geq n-p}X), \Tot(\varphi_{n-p,n-p-1}))_{p \geq 0}$ in $\comp{\cat A} = Z^0(\compdg{\cat A})$:
\begin{equation}
\Tot(X) \cong \varinjlim_{p \geq 0} \Tot(\sigma_{\geq n-p} X).
\end{equation}
\end{prop}
\begin{proof}
Assume we are given a dg-module $M$ and closed degree $0$ morphisms $\alpha_{n-p} \colon \Tot(\sigma_{\geq n-p} X) \to M$ for all $p \geq 0$, such that $\alpha_{n-p} = \alpha_{n-p-1} \circ \Tot(\varphi_{n-p,n-p-1})$ for all $p \geq 1$. We want to define a unique $\alpha \colon \Tot(X) \to M$ such that $\alpha \circ \Tot(\varphi_{n-p}) = \alpha_{n-p}$ for all $p \geq 0$. We observe that
\begin{equation*}
\Tot(X) = \bigcup_{p \geq 0} \varphi_{n-p}(\Tot(\sigma_{\geq n-p} X))
\end{equation*}
and the differential on $\Tot(\varphi_{n-p})(\Tot(\sigma_{\geq n-p} X))$ is the restriction of the differential on $\Tot(X)$. So, for all $y \in \Tot(X)(A)$, $y=\Tot(\varphi_{n-p})(y')$ for a unique $y'$, for some $p \geq 0$. Hence, we are forced to set
\begin{equation*}
\alpha(y) = \alpha_{n-p}(y').
\end{equation*}
Now it is easy to verify that $\alpha$ is well-defined and satisfies the required properties.
\end{proof}

It is known that the totalisation $\Tot(X) \cong \varinjlim_p \Tot(\sigma_{\geq n-p}X)$, as a directed colimit, lies in the following short exact sequence in $\comp{\cat A}$:
\begin{equation} \label{eq:ses_TotX}
0 \to \bigoplus_{p \geq 0} \Tot(\sigma_{\geq n-p} X) \xrightarrow{1-\mu} \bigoplus_{p \geq 0} \Tot(\sigma_{\geq n-p}X) \xrightarrow{\oplus_p \Tot(\varphi_{n-p})} \Tot(X) \to 0,
\end{equation}
where $\mu$ is the morphism induced by
\begin{equation*}
\Tot(\sigma_{\geq n-p} X) \xrightarrow{\Tot(\varphi_{n-p,n-p-1})} \Tot(\sigma_{\geq n-p-1} X) \to \bigoplus_{p \geq 0} \Tot(\sigma_{\geq n-p} X).
\end{equation*}
Recall from Definition \ref{def:holim_dgmod} that the (strictly dg-functorial) homotopy colimit $\hocolim_{p \geq 0} \Tot(\sigma_{\geq n-p} X)$ is the cone in the following pretriangle in $\compdg{\cat A}$:
\begin{equation}
\bigoplus_{p \geq 0} \Tot(\sigma_{\geq n-p} X) \xrightarrow{1-\mu} \bigoplus_{p \geq 0} \Tot(\sigma_{\geq n-p}X) \to \hocolim_{p \geq 0} \Tot(\sigma_{\geq n-p} X).
\end{equation}
In order to compare $\Tot(X) \cong \varinjlim_p \Tot(\sigma_{\geq n-p} X)$ and $\hocolim_p \Tot(\sigma_{\geq n-p} X)$, we use the following general result:
\begin{lemma} \label{lemma:dg-coker-iso-cone}
Let
\begin{equation*}
A \xrightarrow{f} B \xrightarrow{g} C
\end{equation*}
a degreewise split exact sequence of maps in $Z^0(\cat A)$, namely, there are degree $0$ (not necessarily closed) maps $\sigma \colon C \to B$ and $\rho \colon B \to A$ such that $B$, together with those maps, is a biproduct of $A$ and $C$:
\begin{equation*}
gf =0, \quad \rho \sigma = 0, \quad g \sigma = 1, \quad \rho f= 1, \quad \sigma g + f \rho = 1.
\end{equation*}
Then, the closed degree $0$ morphism $\varphi \colon \cone(f) \to C$ defined by $\varphi = (0, g)$ (with respect to the direct sum decomposition $\cone(f) = A[1] \oplus B$) is an isomorphism in $H^0(\cat A)$. In particular, $\varphi$ makes the following diagram commute in $Z^0(\cat A)$:
\begin{equation*}
\begin{gathered}
\xymatrix{
A \ar[r]^f \ar@{=}[d] & B \ar[r] \ar@{=}[d] & \cone(f) \ar[d]^\varphi \\
A \ar[r]^f & B \ar[r]^g & C.
}
\end{gathered}
\end{equation*}
\end{lemma}
\begin{proof} 
We define a homotopy inverse of $\varphi$ as follows. Observe that $d(g \sigma) = g d\sigma = 0$, so $f \rho d\sigma = d\sigma - \sigma g d\sigma = d\sigma$. We set
\begin{equation*}
\psi = \begin{pmatrix} -\delta \\ \sigma  \end{pmatrix} \colon C \to \cone(f),
\end{equation*}
where $\delta = \shiftid{A}{0}{1} g d\sigma$. This is a closed  degree $0$ morphism which serves as a homotopy inverse to $\varphi$.
\end{proof}
Now, we check that the short exact sequence \eqref{eq:ses_TotX} is degreewise split. To do so, it is sufficient to check that $\oplus_p \Tot(\varphi_{n-p})$ has a degree $0$ section $\sigma \colon \Tot(X) \to \bigoplus_p \Tot(\sigma_{\geq n-p}X)$. We define $\sigma$ on $A_m[-m]$ to be the inclusion on the first summand:
\begin{align*}
& A_m[-m] \to \bigoplus_{p \geq 0} A_m[-m] \to \bigoplus_{p \geq 0} \Tot(\sigma_{\geq n-p}X), \quad m \geq n, \\
& A_{n-k}[-n+k] \to \bigoplus_{p \geq k} A_{n-k}[-n+k] \to \bigoplus_{p \geq 0} \Tot(\sigma_{\geq n-p} X), \quad k > 0.
\end{align*}
It is immediate to check that $\oplus_p \Tot(\varphi_{n-p}) \circ \sigma = 1$. Hence, from Lemma \ref{lemma:dg-coker-iso-cone} we deduce:
\begin{prop} \label{prop:tot_hocolim_iso_lim}
The morphism $\oplus_p \Tot(\varphi_{n-p})$ induces an isomorphism
\begin{equation}
\varphi \colon \hocolim_{p \geq 0} \Tot(\sigma_{\geq n-p} X) \to \Tot(X)
\end{equation}
in $\hocomp{\cat A}$. Moreover, the following diagram is commutative in $\comp{\cat A}$ for all $p \geq 0$:
\begin{equation} \label{eq:tot_hocolim_commutative}
\begin{gathered}
\xymatrix{
\Tot(\sigma_{\geq n-p} X) \ar[d] \ar[dr]^{\Tot(\varphi_{n-p})} \\
\hocolim_{p \geq 0} \Tot(\sigma_{\geq n-p} X) \ar[r]^-\varphi & \Tot(X),
}
\end{gathered}
\end{equation}
where the vertical arrow is the canonical morphism to the homotopy colimit.
\end{prop}

We have seen how an object $X \in \Tw^-(\cat A)$ can be reconstructed from its truncations $\sigma_{\geq n} X$. On the other hand, a twisted complex in $\Tw^-(\cat A)$ can be constructed from a suitable ``increasing sequence'':
\begin{prop} \label{prop:resolution_twminus_colimit}
Let $\cat A$ be a dg-category. Assume there is a sequence $(A_n)_{n \leq M}$ of objects of $\cat A$ and a sequence $(X_n)_{n \leq M}$ of twisted complexes in $\Tw^-_\bd (\cat A)$ such that:
\begin{align*}
X_M &= A_M[-M], \\
X_{n-1} &= \cone(A_{n-1}[-n] \xrightarrow{\beta_n} X_n),
\end{align*}
for suitable closed degree $0$ morphisms $\beta_n \colon A_{n-1}[-n] \to X_n$, so that there is a pretriangle
\begin{equation*}
A_{n-1}[-n] \xrightarrow{\beta_n} X_n \xrightarrow{\varphi_{n,n-1}} X_{n-1} \to A_{n-1}[n+1]
\end{equation*}
in $\Tw^-(\cat A)$.
Then, there is a twisted complex $X \in \Tw^-(\cat A)$ such that $\sigma_{\geq n} X = X_n$. In particular, the totalisation $\Tot(X)$ (together with the natural inclusions $\Tot(X_{M-p}) \to \Tot(X)$) is a colimit of the sequence $(\Tot(X_{M-p}) \xrightarrow{\Tot(\varphi_{M-p,M-p-1})} \Tot(X_{m-p-1}))_p$:
\begin{equation*}
\Tot(X) \cong \varinjlim_p \Tot(X_{M-p}).
\end{equation*}
\end{prop}
\begin{proof}
By construction we have $X_{n-1} = (A_{n-1}[-n+1] \oplus X_n, \begin{psmallmatrix} 0 & 0 \\ \beta_n \shiftid{A_{n-1}}{-n+1}{-n} & q_{X_n} \end{psmallmatrix})$. Hence, we may set
\begin{equation*}
X = \left( \bigoplus_{i \leq M} A_i[-i],q \right),
\end{equation*}
where $q_i^j \colon Q_i \to Q_j$ is set to be $({q_{X_k}})_i^j$, for $k \leq \min(i,j)$. This is well defined and by construction $\sigma_{\geq n} X = X_n$. The last part of the claim follows from Proposition \ref{prop:tot_twmin_colim}.
\end{proof}
If $X \in \Tw^-(\cat A)$, then $\sigma_{\geq n-p} X \in \Tw^-_\bd(\cat A)$ for all $p \geq 0$. From \eqref{eq:twtruncation_triangle} it is easy to see that $\sigma_{\geq n-p} X$, as any object in $\Tw^-_\bd (\cat A)$, is obtained as an iterated cone of (shifts of) twisted complexes of the form $A=(A,0)$ with $A \in \cat A$. In particular, $\Tw^-_\bd (\cat A)$ is strongly pretriangulated and the totalisation functor restricts to
\begin{equation} \label{eq:twbd_pretr}
\Tot \colon \Tw^-_\bd (\cat A) \to \pretr{\cat A} \subset \hproj{\cat A},
\end{equation}
and $\Tot(\sigma_{\geq n-p} X) \in \pretr{\cat A}$. Since $\hproj{\cat A}$ is closed under direct sums, cones and isomorphisms in $\hocomp{\cat A}$, it is also closed under homotopy colimits, and we deduce:
\begin{prop} \label{prop:twminus_hproj}
For all $X \in \Tw^-(\cat A)$, the totalisation $\Tot(X) \approx \hocolim_p \Tot(\sigma_{\geq n-p} X)$ is an h-projective dg-module. In particular, $\Tot$ induces a fully faithful dg-functor
\begin{equation} \label{eq:tot_twminus_hproj}
\Tot \colon \Tw^-(\cat A) \to \hproj{\cat A}
\end{equation}
and hence also a fully faithful functor
\begin{equation}
H^0(\Tot) \colon H^0(\Tw^-(\cat A)) \to \dercomp{\cat A}.
\end{equation}
\end{prop}
It is well-known that the derived category $\dercomp{\cat A}$ of a dg-category $\cat A$ with cohomology concentrated in nonpositive
degrees has a t-structure whose heart is the category $\Mod{H^0(\cat A)}$ (see \cite[Lem 2.2, Prop 2.3]{amiot-cluster} for a proof when $\cat A$ is a dg-algebra). This allows us to use the results of \S \ref{subsection_tstruct_hocolim} and describe the cohomology of the sequence of truncations $(\Tot(\sigma_{\geq n-p} X))_p$.
\begin{lemma} \label{lemma:truncation_cohomology_constant}
Let $X \in \Tw^-(\cat A)$ be of the form $X=(\bigoplus_{j \leq M} A[-j], q)$, and consider the sequence $(\sigma_{\geq M-p} X \xrightarrow{\varphi_{M-p,M-p-1}} \sigma_{\geq M-p-1} X)_{p \geq 0}$ and the natural maps $\varphi_{M-p} \colon \sigma_{\geq M-p} X \to X$ (see \eqref{eq:inclusion_truncations}). Then, the induced maps in cohomology
\begin{align*}
H^i(\Tot(\varphi_{M-p,M-p-1})) \colon H^i(\Tot(\sigma_{\geq M-p} X)) & \to H^i(\Tot(\sigma_{\geq M-p-1} X)), \\
H^i(\Tot(\varphi_{M-p})) \colon H^i(\Tot(\sigma_{\geq M-p} X)) & \to H^i(\Tot(X)) 
\end{align*}
are isomorphisms for $i > M-p$ and epimorphisms for $i=M-p$.
\end{lemma}
\begin{proof}
Upon shifting. assume $M=0$ so that $X= (\bigoplus_{j\leq 0} A_j[-j],q)$. By \eqref{eq:twtruncation_triangle}, we have a pretriangle in $\compdg{\cat A}$
\begin{equation*}
A_{-p-1}[p] \to \Tot(\sigma_{\geq -p} X) \to \Tot(\sigma_{-p-1} X) \to A_{-p-1}[p+1],
\end{equation*}
where we identify $A_{-p}$ with $\cat A(-,A_{-p})$. Taking $i$-th cohomology, we get the following exact sequence:
\begin{align*}
H^{i+p}(A_{-p-1}) \to & H^i(\Tot(\sigma_{\geq -p} X)) \to H^i(\Tot(\sigma_{\geq -p-1} X)) \\ & \to H^{i+p+1}(A_{-p-1}).
\end{align*}
$\cat A$ has cohomology concentrated in nonpositive degrees, hence if $i > - p$ we have that both $H^{i+p}(A_{-p-1})=0$ and $H^{i+p+1}(A_{-p-1})=0$, and if $i=-p$ then only $H^{i+p+1}(A_{-p-1})=0$. So,  the morphism
\begin{equation*}
H^i(\Tot(\sigma_{\geq M-p} X)) \to H^i(\Tot(\sigma_{\geq M-p-1}))
\end{equation*}
is an isomorphism if $i > -p$ and an epimorphism if $i=-p$, and our first claim is proved.

Now, we can apply Lemma \ref{lemma:tcohom_holim} to the induced sequence $(\Tot(\sigma_{-p} X) \to \Tot(\sigma_{-p-1} X))_p$ in $\dercomp{\cat A}$: we deduce that the natural map
\begin{equation*}
\Tot(\sigma_{\geq -p} X) \to \hocolim_p \Tot(\sigma_{-p} X)
\end{equation*}
into the (strictly dg-functorial) homotopy colimit $\hocolim_p \Tot(\sigma_{-p} X)$ is such that
\begin{equation*}
H^i(\Tot(\sigma_{\geq -p} X)) \to H^i(\hocolim_p \Tot(\sigma_{-p} X))
\end{equation*}
is an isomorphism for $i > -p$ and an epimorphism for $i=p$. Hence, our second claim follows from Proposition \ref{prop:tot_hocolim_iso_lim}: there is an isomorphism $\hocolim_p \Tot(\sigma_{-p} X) \to \Tot(X)$ in $\hocomp{\cat A}$ such that the diagram \eqref{eq:tot_hocolim_commutative} is commutative.
\end{proof}
\subsection{Twisted complexes and quasi-equivalences} \label{subsection:twcomp_qeq}
It is well known that a dg-functor $F \colon \cat A \to \cat B$ induces a dg-functor
\begin{equation}
\Ind_F \colon \compdg{\cat A} \to \compdg{\cat B},
\end{equation}
which is left adjoint to the restriction functor
\begin{equation*}
\begin{split}
\res_F \colon \compdg{\cat B} & \to \compdg{\cat A}, \\
Y &\mapsto Y \circ F.
\end{split}
\end{equation*}
We refer to \cite{drinfeld-dgcat} for its definition and we recall from there some of its relevant properties:
\begin{itemize}
\item $\Ind_F$ is left adjoint to the restriction functor $\res_F \colon \compdg{\cat B} \to \compdg{\cat A}$ and it preserves representable modules. Namely, there is an isomorphism of complexes
\begin{equation*}
\Ind_F(\cat A(-,A)) \cong \cat B(-,F(A)),
\end{equation*}
natural in $A \in \cat A$.
\item $\Ind_F$ preserves h-projective modules and hence induces a dg-functor
\begin{equation*}
\Ind_F \colon \hproj{\cat A} \to \hproj{\cat B}.
\end{equation*}
If $F$ is fully faithful, the same is true for $\Ind_F$; if $F$ is a quasi-equivalence, the same is true for $\Ind_F \colon \hproj{\cat A} \to \hproj{\cat B}$.
\item $\Ind_F$ preserves cones and shifts, hence it induces a dg-functor
\begin{equation*}
\Ind_F \colon \pretr{\cat A} \to \pretr{\cat B}.
\end{equation*}
If $F$ is a quasi-equivalence, the same is true for $\Ind_F \colon \pretr{\cat A} \to \pretr{\cat B}$. 

Moreover, if $i \colon \cat A \to \pretr{\cat A}$ is the natural inclusion induced by the Yoneda embedding, then $\Ind_i \colon \compdg{\cat A} \to \compdg{\pretr{\cat A}}$ is an equivalence of dg-categories.
\end{itemize} 
\begin{remark} \label{remark:ind_qfun_adj}
For all $X \in \hproj{\cat A}$ and $Y \in \hproj{\cat B}$, we have the adjunction isomorphism:
\begin{equation*}
\hproj{\cat B}(\Ind_F(X),Y) \cong \compdg{\cat A}(X,Y \circ F).
\end{equation*}
Hence, setting
\begin{equation} \label{eq:res_qfun}
(\res_F)_Y^X = \compdg{\cat A}(X, Y \circ F)
\end{equation}
we have a candidate quasi-functor $\res_F \colon \hproj{\cat B} \to \hproj{\cat A}$ which is right adjoint to the dg-functor $\Ind_F \colon \hproj{\cat A} \to \hproj{\cat B}$. Indeed, take an h-projective resolution $Q(Y \circ F) \to Y \circ F$; for all $X \in \hproj{\cat A}$, it induces a quasi-isomorphism
\begin{equation*}
\hproj{\cat A}(X, Q(Y \circ F)) \to \compdg{\cat A}(X,Y \circ F) = (\res_F)^X_Y,
\end{equation*}
natural in $X$.
\end{remark}
In a precise sense, the functor $\Tw^-(-)$ (see Remark \ref{remark:Twmin_functorial}) can be viewed as a restriction of $\Ind$:
\begin{prop}
Let $F \colon \cat A \to \cat B$ be a dg-functor between dg-categories with cohomology concentrated in nonpositive degrees. Then, the following diagram is commutative (up to natural isomorphism):
\begin{equation} \label{eq:Tw-_hproj_commute}
\begin{gathered}
\xymatrix{
\Tw^-(\cat A) \ar[r]^{\Tw^-(F)} \ar[d]_{\Tot_{\cat A}} & \Tw^-(\cat B) 
\ar[d]^{\Tot_{\cat B}} \\
\hproj{\cat A} \ar[r]^{\Ind_F} & \hproj{\cat B},
}
\end{gathered}
\end{equation}
where $\Tot_{\cat A}$ and $\Tot_{\cat B}$ are the (quasi-fully faithful) totalisation functors \eqref{eq:tot_twmin_def}, which have values in h-projective modules thanks to Proposition \ref{prop:twminus_hproj}.
\end{prop}
\begin{proof}
Since $\Ind_F \dashv \res_F$, it is sufficient to find an isomorphism of complexes
\begin{equation*}
\compdg{\cat B}(\Tot(\Tw^-(F)(X)),M) \cong \compdg{\cat A}(\Tot(X),\res_F(M)),
\end{equation*}
natural in $X \in \Tw^-(\cat A)$ and $M \in \compdg{\cat B}$. This can be explicitly written down; the details are left to the reader.
\end{proof}
The functor $\Tw^-(-)$ preserves quasi-equivalence, as $\Ind$ does.
\begin{prop} \label{prop:Twmin_qeq}
Let $F \colon \cat A \to \cat B$ be a dg-functor between dg-categories with cohomology concentrated in nonpositive degrees. If $F$ is a quasi-equivalence, then $\Tw^-(F)$ is a quasi-equivalence.
\end{prop}
\begin{proof}
We notationally identify both categories $\cat A$ and $\cat B$ with their images in $\compdg{\cat A}$ and $\compdg{\cat B}$ under the Yoneda embedding. We know that $\Ind_F$ is a quasi-equivalence and both $\Tot_{\cat A}$ and $\Tot_{\cat B}$ are quasi-fully faithful, so by the commutativity of \eqref{eq:Tw-_hproj_commute} we immediately deduce that $\Tw^-(F)$ is quasi-fully faithful. 

In order to prove essential surjectivity in $H^0$, let $Y \in \Tw^-(\cat B)$, and upon a suitable shift assume it is of the form:
\begin{equation*}
Y= \left( \bigoplus_{i \leq 0} B_i[-i], r   \right).
\end{equation*}
From \eqref{eq:twtruncation_triangle} we get a pretriangle in $\hproj{\cat B}$, for all $p \geq 0$:
\begin{equation*}
B_{-p-1}[p] \to \Tot_{\cat B}(\sigma_{\geq -p}Y) \to \Tot_{\cat B}(\sigma_{\geq -p-1} Y) \to B_{-p-1}[p+1],
\end{equation*}
Since $F$ is a quasi-equivalence, we can find $X_0=A_0 \in \cat A$ such that $F(A_0) \approx B_0$. Next, assume inductively that $\sigma_{\geq -p}Y \cong \Tw^-(F)(X_{-p})$ in $H^0(\Tw^-(\cat B))$, where
\begin{equation*}
X_{-p} = \left(\bigoplus_{i=-p}^0 A_i[-i], q_p \right) \in \Tw^-_\bd(\cat A).
\end{equation*}
We have that $B_{-p-1} \cong F(A_{-p-1})$ in $H^0(\cat B)$ for some $A_{-p-1} \in \cat A$; consider the diagram in $\hproj{\cat B}$:
\begin{equation*}
\begin{gathered}
\xymatrix@C=2.2cm{
B_{-p-1}[p] \ar[r] \ar[d]^\approx & \Tot_{\cat B}(\sigma_{\geq -p}Y) \ar[d]^\approx \ar[r] & \Tot_{\cat B}(\sigma_{\geq -p-1} Y) \ar@{.>}[d]^\approx \\
F(A_{-p-1}[p]) \ar@{-->}[r]^-{\Tot_{\cat B}(\Tw^-(F)(\beta_{-p}))} & \Tot_{\cat B}(\Tw^-(F)(X_{-p})) \ar[r] & \Tot_{\cat B}((\Tw^-(F)(X_{-p-1})).
}
\end{gathered}
\end{equation*}
By the inductive hypothesis the first two vertical arrows on the left are homotopy equivalences, so that we can find a closed degree $0$ map
\begin{equation*}
F(A_{-p-1}[p]) \to \Tot_{\cat B}(\Tw^-(F)(X_p))
\end{equation*}
in $\hproj{\cat B}$, such that the leftmost square commutes up to homotopy. Following our convention \eqref{eq:twcompl_inclusion}, we have
\begin{equation*}
F(A_{-p-1}[p]) = \Tot_{\cat B}(\Tw^-(F)(A_{-p-1}[p]))
\end{equation*}
Since $\Tot_{\cat B} \circ \Tw^-(F)$ is quasi-fully faithful, the above map is (up to homotopy) of the form $\Tot_{\cat B}(\Tw^-(F))(\beta_{-p})$, for some closed degree $0$ morphism $\beta_{-p} \colon A_{-p-1}[p] \to X_{-p}$ in $\Tw^-(\cat A)$. We define $X_{-p-1} = \cone(\beta_{-p})$ in $\Tw^-(\cat A)$, and clearly we can find the dotted vertical homotopy equivalence which makes the above diagram commute in $H^0(\hproj{\cat B})$.

We can now apply Proposition \ref{prop:resolution_twminus_colimit} and find $X \in \Tw^-(\cat A)$ such that $\sigma_{\geq -p} X = X_{-p}$ for all $p \geq 0$. Recall from Proposition \ref{prop:tot_hocolim_iso_lim} that
\begin{align*}
\Tot_{\cat B}(Y) & \approx \hocolim_p \Tot_{\cat B}(\sigma_{\geq -p} Y), \\
\Tot_{\cat A}(X) & \approx \hocolim_p \Tot_{\cat A}(X_{-p}).
\end{align*}
Then, the commutative square in $H^0(\hproj{\cat B})$
\begin{equation*}
\begin{gathered}
\xymatrix{
\Tot_{\cat B}(\sigma_{\geq -p}Y) \ar[d]^\approx \ar[r] & \Tot_{\cat B}(\sigma_{\geq -p-1} Y) \ar[d]^\approx \\
\Tot_{\cat B}(\Tw^-(F)(X_p)) \ar[r] & \Tot_{\cat B}((\Tw^-(F)(X_{p+1})).
}
\end{gathered}
\end{equation*}
tells us that
\begin{align*}
\Tot_{\cat B}(Y) &\approx \hocolim_p \Tot_{\cat B}(\sigma_{\geq -p} Y) \\
&\approx \hocolim_p \Tot_{\cat B}(\Tw^-(F)(X_p)),
\end{align*}
and moreover
\begin{align*}
\hocolim_p \Tot_{\cat B}(\Tw^-(F)(X_p)) & \cong \hocolim_p \Ind_F(\Tot_{\cat A}(X_p)) \\
& \cong \Ind_F (\hocolim_p \Tot_{\cat A}(X_p)) \\
& \approx \Ind_F(\Tot_{\cat A}(X)) \\
& \cong \Tot_{\cat B}(\Tw^-(X)),
\end{align*}
whence $Y \cong \Tw^-(X)$ in $H^0(\Tw^-(\cat B))$, as we wanted. In the above chain of  homotopy equivalences and isomorphisms, we used the commutativity of \eqref{eq:Tw-_hproj_commute} and the fact that $\Ind_F$ commutes with homotopy colimits (we invite the reader to check this using that $\Ind_F$ is a dg-functor which preserves direct sums).
\end{proof}

\section{Twisted complexes on homotopically locally coherent dg-categories}
It is well-known that the derived category $\dercomp{\cat A}$ of a dg-category $\cat A$ with cohomology concentrated in nonpositive
degrees has a (non-degenerate) t-structure whose heart is the category $\Mod{H^0(\cat A)}$ (see \cite[Lem 2.2, Prop 2.3]{amiot-cluster} for a proof when $\cat A$ is a dg-algebra). In this section, we give conditions on $\cat A$ in order that the triangulated category $H^0(\Tw^-(\cat A))$ naturally inherits this t-structure.
\subsection{Finitely presented modules and coherent categories}
We start by briefly recalling the notion of \emph{(right) coherent category} and some related results we shall need. For this subsection, we fix a $\basering k$-linear category $\cat C$.
\begin{defin}
A (right) $\cat C$-module $M \in \Mod{\cat C}$ is \emph{finitely presented} if there is an exact sequence
\begin{equation*}
\bigoplus_{j=1}^m \cat C(-,C'_j) \to \bigoplus_{i=1}^n \cat C(-,C_i) \to M \to 0,
\end{equation*}
for some objects $C_1,\ldots,C_n$ and $C'_1,\ldots,C'_m$ in $\cat C$. The full subcategory of finitely presented modules of $\Mod{\cat C}$ is denoted by $\Modfp{\cat C}$.
\end{defin}
The following result is true without any additional hypothesis on $\cat C$:
\begin{prop} \label{prop:fp_closed_coker_ext}
The category $\Modfp{\cat C}$ is closed under cokernels, extensions and direct summands in $\Mod{\cat C}$.
\end{prop}
\begin{proof}
It follows from \cite[Tag 0517]{stacks-project}.
\end{proof}
The definition of \emph{coherent category} is as follows:
\begin{defin}
$\cat C$ is \emph{(right) coherent} if $\Modfp{\cat C}$ is an abelian category.
\end{defin}
Since $\Modfp{\cat C}$ has cokernels and it can be shown that the inclusion $\Modfp{\cat C} \hookrightarrow \Mod{\cat C}$ preserves kernels, we deduce that $\cat C$ is coherent if and only if $\Modfp{\cat C}$ is closed under kernels in $\Mod{\cat C}$. Next, we give a very useful characterisation of coherent additive categories:
\begin{defin}
Let $f \colon C \to C'$ be a morphism in $\cat C$. A \emph{weak kernel} of $f$ is a morphism $g \colon D \to C$ such that the sequence
\begin{equation*}
\cat C(-,D) \xrightarrow{g_*} \cat C(-,C) \xrightarrow{f_*} \cat C(-,C')
\end{equation*}
is exact in $\Mod{\cat C}$. If every morphism in $\cat C$ has a weak kernel, we say that \emph{$\cat C$ admits weak kernels}.
\end{defin}
\begin{prop}[{\cite[Lemma 1]{krause-brown-coherent}}] \label{prop:cohcat_wker}
Assume that $\cat C$ is additive. Then $\cat C$ is coherent if and only if it admits weak kernels.
\end{prop}
\subsection{Homotopically locally coherent dg-categories}
By definition, the category of finitely presented modules on a coherent category is an abelian subcategory of the category of modules. In analogy, we now give a more general and ``homotopically relevant'' notion of coherence for dg-categories: this will have the key property that a suitable category of h-projective and ``homotopically finitely presented'' dg-modules will inherit both the property of being pretriangulated and the t-structure from the dg-category of h-projective dg-modules.
\begin{remark} \label{remark:subcat_leftbd_rightbd}
A t-structure on a pretriangulated dg-category $\cat A$ is by definition a t-structure on the homotopy category $H^0(\cat A)$. We shall denote by $\cat A_{\leq n}$ and $\cat A_{\geq n}$ the full dg-subcategories of $\cat A$ with the same objects as the aisles $H^0(\cat A)_{\leq n}$ and $H^0(\cat A)_{\geq n}$. Moreover, we shall denote by $\cat A^+$ and $\cat A^-$ the full dg-subcategories of $\cat A$ whose objects are the same as $H^0(\cat A)^+$ and $H^0(\cat A)^-$. Given a quasi-functor $F \colon \cat A \to \cat B$ between dg-categories with t-structures, we say that it is $\emph{t-exact}$ (or that it \emph{preserves the t-structures}) if $H^0(F)$ does, namely if it commutes with the truncation functors (or, equivalently, if it preserves the aisles).  If $\cat A$ is strongly pretriangulated, then the same is true for $\cat A^+$ and $\cat A^-$. For a given dg-category $\cat B$, we shall sometimes write $\dercompmin{\cat B}$ and $\hprojmin{\cat B}$ instead of $\dercomp{\cat B}^-$ and $\hproj{\cat B}^-$.
\end{remark}
\begin{defin}
Let $\cat Q$ be a dg-category. A dg-module $M \in \compdg{\cat Q}$ is called \emph{homotopically finitely presented} (in short, \emph{hfp}) if $H^i(M)$ is a finitely presented $H^0(\cat Q)$-module for all $i \in \mathbb Z$:
\[
H^i(M) \in \Modfp{H^0(\cat Q)}, \quad \forall\, i \in \mathbb Z.
\]
We denote by $\hproj{\cat Q}^\mathrm{hfp}$ and $\dercomp{\cat Q}^{\mathrm{hfp}}$ the full subcategories of respectively $\hproj{\cat Q}$ and $\dercomp{\cat Q}$ whose objects are the homotopically finitely presented $\cat Q$-dg-modules:
\begin{align} 
\hproj{\cat Q}^\mathrm{hfp} &= \{M \in \hproj{\cat Q} : H^i(M) \in \Modfp{H^0(\cat Q)} \ \ \forall\, i\}, \label{eq:hproj_tsubcat_fp} \\
\dercomp{\cat Q}^{\mathrm{hfp}} &= \{M \in \dercomp{\cat Q} : H^i(M) \in \Modfp{H^0(\cat Q)} \ \ \forall\, i \}. \label{eq:dercat_tsubcat_fp}
\end{align}

We shall also set:
\begin{align} 
\hprojmin{\cat Q}^\mathrm{hfp} &= \{M \in \hproj{\cat Q}^\mathrm{hfp} : H^i(M) = 0 \quad \forall\, i \gg 0 \}, \\ \label{eq:tot_twminus_essimg}
\dercompmin{\cat Q}^\mathrm{hfp} &= \{M \in \dercomp{\cat Q}^\mathrm{hfp} : H^i(M) = 0 \quad \forall\, i \gg 0 \}.
\end{align}
\end{defin}
\begin{defin}
\label{def:hlc}
A dg-category $\cat Q$ is called \emph{(right) homotopically locally coherent} (in short, \emph{hlc}) if:
\begin{itemize}
\item $\cat Q$ is cohomologically concentrated in nonpositive degrees: for all $A,A' \in \cat Q$, we have $H^i(\cat Q(A,A'))=0$ for all $i>0$.
\item $H^0(\cat Q)$ is an additive and (right) coherent $\basering k$-linear category.
\item For all $A \in \cat Q$, the represented dg-module $\cat Q(-,A)$ is homotopically finitely presented, in other words the $H^0(\cat Q)$-module $H^i(\cat Q(-,A))$ is finitely presented for all $i \in \mathbb Z$: $H^i(\cat Q(-,A)) \in \Modfp{H^0(\cat Q)}$.
\end{itemize}
\end{defin}
There is a nice cohomological characterisation of the dg-category $\Tw^-(\cat Q)$ when $\cat Q$ is a hlc dg-category, which will be proven in \S \ref{subsection:proof_thm_hlc_tw_equiv}.
\begin{thm} \label{thm:hlc_tw_equiv}
Let $\cat Q$ be a hlc dg-category.
\begin{enumerate}
\item \label{item:hlc_tw_equiv_1}  The dg-category $\hproj{\cat Q}^\mathrm{hfp}$ of homotopically finitely presented (hfp) modules is strongly pretriangulated and has a non-degenerate t-structure which is induced from $\hproj{\cat Q}$; its heart is $\Modfp{H^0(\cat Q)}$.

In other words, the category $\dercomp{\cat Q}^\mathrm{hfp}$ is a triangulated subcategory of $\dercomp{\cat Q}$ and it has a non-degenerate t-structure induced from $\dercomp{\cat Q}$; its heart is $\Modfp{H^0(\cat Q)}$.
\item  \label{item:hlc_tw_equiv_2} The totalisation dg-functor \eqref{eq:tot_twminus_hproj} induces a quasi-equivalence
\begin{equation*}
\Tot \colon \Tw^-(\cat Q) \xrightarrow{\approx} \hprojmin{\cat Q}^\mathrm{hfp}.
\end{equation*}
In particular, $\Tw^-(\cat Q)$ has a unique non-degenerate right bounded t-structure with heart $\Modfp{H^0(\cat Q)}$ and such that the totalisation functor is t-exact. Moreover, $\hprojmin{\cat Q}^\mathrm{hfp}$ is essentially $\mathcal U$-small.
\end{enumerate}
\end{thm}

\subsection{The resolution and the proof of Theorem \ref{thm:hlc_tw_equiv}} \label{subsection:proof_thm_hlc_tw_equiv}
The proof of part \ref{item:hlc_tw_equiv_1} of Theorem \ref{thm:hlc_tw_equiv} follows from the following two lemmas.
\begin{lemma} \label{lemma:hproj_tsubcat_fp_pretr}
The dg-category $\hproj{\cat Q}^\mathrm{hfp}$ is strongly pretriangulated. Moreover, its full dg-subcategory $\hprojmin{\cat Q}^{\mathrm{hfp}}$ is strongly pretriangulated and contains the representables $\cat Q(-,A)$.
\end{lemma}
\begin{proof}
By definition, $\hproj{\cat Q}^\mathrm{hfp}$ is closed under shifts, so to see that it is strongly pretriangulated we only have to show that it is closed under cones. Let $f \colon M \to N$ a closed degree $0$ morphism in $\hproj{\cat Q}^{\mathrm{hfp}}$. It fits in the following pretriangle in $\hproj{\cat Q}$:
\begin{equation*}
M \xrightarrow{f} N \xrightarrow{j} \cone(f) \xrightarrow{p} M[1].
\end{equation*}
Taking cohomology, we obtain the following exact sequence:
\begin{equation*}
H^i(M) \to H^i(N) \to H^i(\cone(f)) \to H^{i+1}(M) \to H^{i+1}(N)
\end{equation*}
which also gives the following short exact sequence:
\begin{equation*}
0 \to \coker(H^i(f)) \to H^i(\cone(f)) \to \ker(H^{i+1}(f)) \to 0.
\end{equation*}
Since $\cat Q$ is hlc, $H^i(f)$ and $H^{i+1}(f)$ are maps between objects in $\Modfp{H^0(\cat Q)}$. Also, since $H^0(\cat Q)$ is coherent, both $\coker(H^i(f)), \ker(H^{i+1}(f)) \in \Modfp{H^0(\cat Q)}$. Since the category $\Modfp{H^0(\cat Q)}$ is closed under extensions, we deduce that
\begin{equation*}
H^i(\cone(f)) \in \Modfp{H^0(\cat Q)}.
\end{equation*}

Finally, since $\cat Q$ is by hypothesis concentrated in nonpositive degrees, we immediately deduce that $\hprojmin{\cat Q}^\mathrm{hfp}$ contains all the representables.
\end{proof}
\begin{lemma} \label{lemma:hfp_tstruct}
Let $\cat Q$ be a hlc dg-category. Then, $\dercomp{\cat Q}^{\mathrm{hfp}}$ is a triangulated subcategory of $\dercomp{\cat Q}$ stable under truncations, hence it has a non-degenerate t-structure induced from $\dercomp{\cat Q}$; its heart is the category $\Modfp{H^0(\cat Q)}$.
\end{lemma}
\begin{proof}
$\dercomp{\cat Q}^{\mathrm{hfp}}$ is clearly stable under truncations.  To see that it is closed under direct summands, we directly apply Proposition \ref{prop:fp_closed_coker_ext}. Its heart is then given by the intersection of $\dercomp{\cat Q}^\mathrm{hfp}$ with the heart $\Mod{H^0(\cat Q)}$, which is immediately seen to be precisely $\Modfp{H^0(\cat Q)}$.
\end{proof}

Next, we prove part \ref{item:hlc_tw_equiv_2} of Theorem \ref{thm:hlc_tw_equiv}. We already know from Proposition \ref{prop:twminus_hproj} that the totalisation $\Tot \colon \Tw^-(\cat Q) \to \hproj{\cat Q}$ is quasi-fully faithful, so we need to focus on its essential image. First, we prove that totalisations of twisted complexes in $\Tw^-(\cat Q)$ land in the subcategory $\hprojmin{\cat Q}^{\mathrm{hfp}}$:
\begin{lemma} \label{lemma:hlc_twminus_essimg}
Let $\cat Q$ be a hlc dg-category, and let $X =(\bigoplus_{i \geq M} A_i[-i],q) \in \Tw^-(\cat Q)$. Then, $H^i(\Tot(X))=0$ for $i>M$ and $H^i(\Tot(X)) \in \Modfp{H^0(\cat Q)}$ for all $i \in \mathbb Z$. In particular, the totalisation functor restricted to $\Tw^-(\cat Q)$ has image in $\hprojmin{\cat Q}^\mathrm{hfp}$:
\begin{equation*}
\Tot \colon \Tw^-(\cat Q) \to \hprojmin{\cat Q}^\mathrm{hfp}.
\end{equation*}
\end{lemma}
\begin{proof}
Without loss of generality, we assume $M=0$. By Lemma \ref{lemma:truncation_cohomology_constant}, we have for $i \in \mathbb Z$ and $-p < i$ that
\begin{equation*}
H^i(\Tot(X)) \cong H^i(\Tot(\sigma_{\geq -p} X)).
\end{equation*}
So, it is enough to prove the statement for $\sigma_{\geq -p} X$, for all $p \geq 0$. We argue by induction. For the base step, we have $\sigma_{\geq 0} X = A_0 \in \cat Q$ and the claim follows since $\cat Q$ is hlc. Next, assume the thesis is true for $\sigma_{\geq -p} X$. From \eqref{eq:twtruncation_triangle} we obtain a pretriangle in $\compdg{\cat Q}$:
\begin{equation}
\cat Q(-,A_{-p-1})[p] \to \Tot(\sigma_{\geq -p} X) \to \Tot(\sigma_{\geq -p-1} X) \to \cat Q(-, A_{-p-1})[p+1].
\end{equation}
Taking $i$-th cohomology, we get an exact sequence in $\Mod{H^0(\cat Q)}$:
\begin{align*}
H^{i+p}&(\cat Q(-,A_{-p-1})) \xrightarrow{s} H^i(\Tot(\sigma_{\geq -p}X)) \to H^i(\Tot(\sigma_{\geq -p-1}X)) \\ & \to H^{i+p+1}(\cat Q(-,A_{-p-1})) \xrightarrow{t} H^{i+1}(\Tot(\sigma_{\geq -p} X)).
\end{align*}
Now, if $i>0$ we have
\[
H^{i+p}(\cat Q(-,A_{-p-1})) = 0, \quad H^{i+p+1}(\cat Q(-,A_{-p-1})) = 0,
\]
so we have an isomorphism
\begin{equation*} 
H^i(\Tot(\sigma_{\geq -p}X)) \xrightarrow{\sim} H^i(\Tot(\sigma_{\geq -p-1}X))
\end{equation*}
and from the inductive hypothesis we conclude that $H^i(\Tot(\sigma_{\geq -p-1}X)) = 0$. For the other claim, we observe that the above long exact sequence induces the following short exact sequence:
\begin{equation*}
0 \to \coker(s) \to H^i(\Tot(\sigma_{\geq -p-1}X)) \to \ker(t) \to 0.
\end{equation*}
By the inductive hypothesis and since $\cat Q$ is hlc, $s$ and $t$ are maps between objects in $\Modfp{H^0(\cat Q)}$. Also, since $H^0(\cat Q)$ is coherent, both $\coker(s), \ker(t) \in \Modfp{H^0(\cat Q}$. Since $\Modfp{H^0(\cat Q)}$ is closed under extensions, we deduce that $H^i(\Tot(\sigma_{\geq -p-1}X)) \in \Modfp{H^0(\cat Q)}$, as claimed.
\end{proof}
The following is a key technical result which allows us to ``resolve'' objects by means of twisted complexes. We write it down in some generality.
\begin{prop} \label{prop:dginj_resolution}
Let $\cat D$ be a strongly pretriangulated dg-category; we identify it with its pretriangulated hull: $\pretr{\cat D} = \cat D$. Assume that $\cat D$ has a t-structure $(\cat D_{\leq 0}, \cat D_{\geq 0})$. We denote as usual with $H^i(-)$ the i-th cohomology functor $H^0(\cat D) \to H^0(\cat D)^\heartsuit$; we shall write $H^i(f)$ instead of the more precise $H^i([f])$, if $[f]$ is the cohomology class of a closed degree $0$ morphism $f$.

Assume moreover there is a full dg-subcategory $\cat Q \subseteq \cat D_{\leq 0}$ with cohomology concentrated in nonpositive degrees. We recall that the inclusion $\cat Q \subset \cat D$ induces a fully faithful dg-functor $\pretr{\cat Q} \hookrightarrow \pretr{\cat D} = \cat D$ (see the properties of $\Ind$ in \S \ref{subsection:twcomp_qeq}); the totalisation functor \eqref{eq:twbd_pretr} restricted to $\Tw^-_\bd (\cat Q)$ induces a (quasi-fully faithful) dg-functor:
\begin{equation} \label{eq:TQ_tot}
T_{\cat Q} \colon \Tw^-_\bd (\cat Q) \xrightarrow{\Tot} \pretr{\cat Q} \hookrightarrow \cat D.
\end{equation}
Moreover, assume that for any object $A \in \cat D^-$ there is an object $Q \in \cat Q$ and a closed degree $0$ morphism $\alpha \colon Q \to A$ with the property that $H^0(\alpha) \colon H^0(Q) \to H^0(A)$ is an epimorphism in $H^0(\cat D)^\heartsuit$. 

Fix $A \in \cat D_{\leq M}$. There is a sequence $(Q_n)_{n \leq M}$ of objects of $\cat Q$ and a sequence $(X_n)_{n \leq M}$ of twisted complexes in $\Tw^-_\bd (\cat Q)$ such that
\begin{align*}
X_M &= Q_M[-M], \\
X_{n-1} &= \cone(Q_{n-1}[-n] \to X_n),
\end{align*}
for suitable closed degree $0$ morphisms $Q_{n-1}[-n] \to X_n$ in $\Tw^-_\bd (\cat Q)$, so that there are pretriangles in $\cat D$
\begin{equation*}
Q_{n-1}[-n] \to T_{\cat Q}(X_n) \xrightarrow{j_{n,n-1}} T_{\cat Q}(X_{n-1}) \to Q_{n-1}[-n+1]
\end{equation*}
and $X_n$ is concentrated in degrees between $n$ and $M$:
\begin{equation*}
X_n = \left(\bigoplus_{k=n}^M Q_k[-k], q_{X_n} \right).
\end{equation*}
Also, $T_{\cat Q}(X_n) \in \cat D_{\leq M}$ for all $n$. Moreover, there exist closed degree $0$ morphisms $\alpha_n \colon T_{\cat Q}(X_n) \to A$ in $\cat D$, such that the following diagram is (strictly) commutative:
\begin{equation} \label{eq:res_Xn_commutative}
\begin{gathered}
\xymatrix{
T_{\cat Q}(X_n) \ar[r]^-{\alpha_n} \ar[d]_{j_{n,n-1}} & A \\
T_{\cat Q}(X_{n-1}) \ar[ur]_{\alpha_{n-1}},
}
\end{gathered}
\end{equation}
The morphism $\alpha_n$ induces a map in $H^0(\cat D)^\heartsuit$ for all $i \in \mathbb Z$:
\begin{equation*}
H^i(\alpha_n) \colon H^i(T_{\cat Q}(X_n)) \to H^i(A)
\end{equation*}
which is an isomorphism for $i > n$ and an epimorphism for $i=n$. Also the induced map
\begin{equation*}
H^i(j_{n,n-1}) \colon H^i(T_{\cat Q}(X_n)) \to H^i(T_{\cat Q}(X_{n-1}))
\end{equation*}
is an isomorphism for $i > n$ and an epimorphism for $i=n$.
\end{prop}
\begin{proof}
Upon replacing $A$ with a suitable shift, we assume that $M=0$, so that $A \in \cat D_{\leq 0}$ and in particular $H^i(A)=0$ for all $i<0$. We construct the sequences $(Q_n)_{n \geq 0}$ and $(X_n)_{n \geq 0}$ inductively, together with the maps $\alpha_n \colon T_{\cat Q}(X_n) \to A$. For notational ease, we shall drop $T_{\cat D}$ when taking cohomology, writing for instance $H^i(\alpha_n) \colon H^i(X_n) \to H^i(A)$. 

\emph{Base step.} By hypothesis, we can find a closed degree $0$ map in $\cat D$
\begin{equation}
\alpha_0 \colon Q_0 \to A,
\end{equation}
where $Q_0 \in \cat Q \subset \cat D_{\leq 0}$. Clearly, we have $Q_0 = T_{\cat Q}(X_0)$, where $X_0 =(Q_0,0)=Q_0 \in \Tw^-_\bd (\cat Q)$. The map $H^0(\alpha_0)$ is an epimorphism by hypothesis, and $H^i(\alpha_0)=0$ is an isomorphism for all $i < 0$, since (again by hypothesis) both $A$ and $X_0$ have cohomology concentrated in nonpositive degrees.

\emph{Inductive step.} Assume we have the objects $Q_k$, the twisted complexes $X_k$ and the maps $\alpha_k \colon A \to X_k$ with the required properties for $k \geq n$ ($n \leq 0$). Now, set
\begin{equation*}
C_n = \cone(\alpha_n \colon T_{\cat Q}(X_n) \to A)[-1] \in \cat D^-.
\end{equation*}
By hypothesis, we find a closed degree $0$ map
\begin{equation*}
Q_{n-1} \to C_n[n]
\end{equation*}
which is an epimorphism in $H^0(\cat D)^\heartsuit$ upon taking $H^0$. Shifting, we find a closed degree $0$ map
\begin{equation*}
p_n \colon Q_{n-1}[-n] \to C_n
\end{equation*}
such that $H^n(p_n) \colon H^n(Q_{n-1}[-n]) \to H^n(C_n)$ is an epimorphism. Consider the following diagram in $\cat D$:
\begin{equation} \label{eq:alphan_def}
\begin{gathered}
\xymatrix{
Q_{n-1}[-n] \ar[r]^-{\beta_n} \ar[d]_{p_n} & T_{\cat Q}(X_n) \ar[r]^{j_{n,n-1}} \ar@{=}[d] & T_{\cat Q}(X_{n-1}) \ar@{.>}[d]^{\alpha_{n-1}} \\
C_n \ar[r] & T_{\cat Q}(X_n)  \ar[r]^-{\alpha_n} & A.
}
\end{gathered}
\end{equation}
The morphism $\beta_n$ is defined as the composition making the left square (strictly) commute. Now, by hypothesis $X_n = (\bigoplus_{k=n}^0 Q_k[-k], q_{X_n})$ is concentrated in degrees between $n$ and $0$, and $Q_{n-1}[-n]$ is concentrated in degree $n$ as a twisted complex. Hence, the closed degree $0$ map $\beta_n$ necessarily comes from a unique one-sided morphism $b_n \colon Q_{n-1}[n] \to X_n$ in $\Tw^-_\bd (\cat Q)$. By definition, $X_{n-1} = \cone(Q_{n-1}[-n] \xrightarrow{b_n} X_n)$. It is the twisted complex defined by
\begin{equation}
X_{n-1} = (Q_{n-1}[-n+1] \oplus X_n, \begin{psmallmatrix} 0 & 0 \\ b_n \shiftid{Q_{n-1}}{-n+1}{-n} & q_{X_n} \end{psmallmatrix}).
\end{equation}
We notice here that $T_{\cat Q}(X_{n-1}) \in \cat D_{\leq 0}$, since it is the cone of a map between objects in $\cat D_{\leq 0}$. The morphism $j_{n,n-1}$ is induced by the natural inclusion $X_n \to X_{n-1}$. Next, notice that $\alpha_n \circ \beta_n$ is $0$ in the homotopy category (the rows of the above diagram induce distinguished triangles in $H^0(\cat D)$), hence we can find a degree $0$ morphism $c_n \colon Q_{n-1}[-n+1] \to A$ such that
\begin{equation*}
dc_n - \alpha_n \beta_n \shiftid{Q_{n-1}}{-n+1}{-n} = 0.
\end{equation*}
So, we may define the closed degree $0$ morphism $\alpha_{n-1} \colon T_{\cat Q}(X_{n-1}) \to A$ as
\begin{equation}
\alpha_{n-1} = (c_n, \alpha_n),
\end{equation}
and by construction this makes the right square of the above diagram (strictly) commute.

Next, consider the following diagram induced by \eqref{eq:alphan_def} in cohomology (the rows are exact):
\begin{equation*}
\begin{gathered}
\xymatrix{
 H^i(Q_{n-1}[-n]) \ar[r] \ar[d]^{H^i(p_n)} &  H^i(X_n) \ar[r]^-{H^i(j_{n,n-1})} \ar@{=}[d] & H^i(X_{n-1}) \ar[r] \ar[d]^{{H^i(\alpha_{n-1})}} & H^{i+1}(Q_{n-1}[-n]) \ar[r] \ar@{=}[d] & H^{i+1}(X_n) \ar@{=}[d] \\
H^i(C_n) \ar[r] & H^i(X_n) \ar[r]^{H^i(\alpha_n)} & H^i(A) \ar[r] & H^{i+1}(C_n) \ar[r] & H^{i+1}(X_n).
}
\end{gathered}
\end{equation*}
Since $H^i(\alpha_n)$ is an isomorphism for $i > n$ and $H^n(\alpha_n)$ is an epimorphism, we deduce by exactness of the lower  row that $H^i(C_n) \cong 0$ for all $i > n$. Next, observe that $H^i(Q_{n-1}[-n]) \cong 0$ for all $i > n$, since the objects in $\cat Q$ have cohomology concentrated in nonpositive degrees; notice that this implies that $H^i(j_{n,n-1})$ is an isomorphism for $i>n$ and a monomorphism for $i=n$, as required. Moreover, $H^i(j_n) = 0$ is trivially an isomorphism for $i > n$, and recall that $H^n(j_n)$ is an epimorphism by construction. So, we easily deduce (for example, using the Five Lemma) that $H^i(\alpha_{n-1})$ is an isomorphism for all $i \geq n$.

In order to show that $H^{n-1}(\alpha_{n-1})$ is an epimorphism, consider the following diagram:
\begin{equation*}
\begin{gathered}
\xymatrix{
H^{n-1}(X_n) \ar[r] \ar@{=}[d] & H^{n-1}(X_{n-1}) \ar[r] \ar[d]^{H^{n-1}(\alpha_{n-1})} & H^n(Q_{n-1}[-n]) \ar[r] \ar@{->>}[d]^{H^n(p_n)} & H^n(X_n) \ar@{=}[d] \\
H^{n-1}(X_n) \ar[r] & H^{n-1}(A) \ar[r] & H^n(C_n) \ar[r] & H^n(X_n). 
}
\end{gathered}
\end{equation*}
We can now conclude by the version of the 5-lemma recalled in Lemma \ref{lemma:diagramchase}.
\end{proof}
\begin{lemma} \label{lemma:diagramchase}
Consider the following commutative diagram in any abelian category:
\begin{equation*}
\begin{gathered}
\xymatrix{
A \ar[r]^\alpha \ar@{=}[d] & B \ar[r]^\beta \ar[d]^f & C \ar[r]^\gamma \ar@{->>}[d]^g & D \ar@{=}[d] \\
A \ar[r]^{\alpha'} & B' \ar[r]^{\beta'} & C' \ar[r]^{\gamma'} & D. 
}
\end{gathered}
\end{equation*}
Assume that the rows are exact and that $g$ is a epimorphims. Then, $f$ is an epimorphism.
\end{lemma}
\begin{remark}
The above proof, with the suitable changes, applies to a more general setting where we are given a cohomological functor $H^0 \colon H^0(\cat D) \to \mathfrak A$ with values in any abelian category (not necessarily coming from a t-structure on $\cat D$) and a full dg-subcategory $\cat Q \subset \cat D$ concentrated in nonpositive degrees such that $H^i(Q) = 0$ for any $i>0$ and $Q \in \cat Q$. In that case, we may resolve any $A \in \cat D$ such that $H^i(A)=0$ for $i>M$ with a sequence of twisted complexes $(X_n)$ of objects of $\cat Q$ such that $H^i(X_n)=0$ for all $i >M$. We won't need the result in such a generality, and we leave this to the interested reader.
\end{remark}
In the case where $\cat D$ has a non-degenerate right bounded t-structure which is closed under countable coproducts (Definition \ref{def:tstruct_closed_directsums}), we may actually reconstruct any object $A \in \cat D^- = \cat D$ as the homotopy colimit of its resolution.
\begin{coroll} \label{coroll:res_tstruct}
In the setting of the above Proposition \ref{prop:dginj_resolution}, assume that $\cat D$ has a non-degenerate right bounded t-structure which is closed under countable coproducts (Definition \ref{def:tstruct_closed_directsums}). The object $A \in \cat D_{\leq M}$, with the maps $\alpha_{M-p} \colon T_{\cat Q}(X_{M-p}) \to A$ and the ``trivial homotopies'' $h_{M-p} = 0 \colon T_{\cat Q}(X_{M-p}) \to A$, is the homotopy colimit of the sequence $(T_{\cat Q}(X_{M-p}) \xrightarrow{j_{M-p,M-p-1}} T_{\cat Q}(X_{M-p-1}))_p$. Namely, the induced morphism (recall Remark \ref{remark:hocolim_univprop})
\begin{equation}
\holim_p (\alpha^*_{M-p},0) \colon \cat D(A,-) \to \holim_p \cat D(T_{\cat Q}(X_{M-p}),-)
\end{equation}
is an isomorphism in $\dercomp{\opp{\cat D}}$. In particular, any $A \in \cat D$ can be reconstructed as a homotopy colimit as above.
\end{coroll}
\begin{proof}
Upon shifting, assume $M=0$. First, we notice that $\holim_p (\alpha^*_{-p},0)$ is well defined, since \eqref{eq:res_Xn_commutative} is strictly commutative. By hypothesis, the aisle $H^0(\cat D)_{\leq 0}$ is closed under countable coproducts and $T_{\cat Q}(X_{-p}) \in \cat D_{\leq 0}$ for all $p$; hence, by Lemma \ref{lemma:dgcat_hocolim}, we know that $(T_{\cat Q}(X_{-p}) \xrightarrow{j_{-p,-p-1}} T_{\cat Q}(X_{-p-1}))_p$ admits a homotopy colimit $\hocolim_p T_{\cat Q}(X_{-p})$ in $\cat D$. Namely, we have closed degree $0$ maps $j_{-p} \colon T_{\cat Q}(X_{-p}) \to \hocolim_p T_{\cat Q}(X_{-p})$ and homotopies $h_{-p}$ which induce an isomorphism in $\dercomp{\opp{\cat D}}$:
\begin{equation*}
\holim_p (j^*_{-p}, h^*_{-p}) \colon \cat D(\hocolim_p T_{\cat Q}(X_{-p}), -) \to \holim_p \cat D(T_{\cat Q}(X_{-p}),-).
\end{equation*}
By the universal property \eqref{eq:hocolim_univprop_diagram}, we find a closed degree $0$ map $\alpha \colon \hocolim_p T_{\cat Q}(X_{-p}) \to A$ such that the following diagram is commutative in $\dercomp{\opp{\cat D}}$:
\begin{equation*}
\begin{gathered}
\xymatrix{
\cat D(A,-) \ar[r]^-{\holim_n (\alpha^*_{-p},0)} \ar@{.>}[d]^-{\alpha^*} & \holim_n \cat D(T_{\cat Q}(X_{-p}),-) \\
\cat D(\hocolim_p T_{\cat Q}(X_{-p}),-), \ar[ur]_{\holim_n (j^*_n,h^*_n)}
}
\end{gathered}
\end{equation*}
and moreover $[\alpha] [j_{-p}] = [\alpha_{-p}]$ in $H^0(\cat D)$. Now, we know from the above Proposition \ref{prop:dginj_resolution} that both $H^i(j_{-p,-p-1})$ and $H^i(\alpha_{-p})$ are isomorphisms for $i>-p$ and epimorphisms for $i=p$. So, we may apply Corollary \ref{coroll:tcohom_holim} and find that $[\alpha] \colon \hocolim_p T_{\cat Q}(X_{-p}) \to A$ is an isomorphism in $H^0(\cat D)$. Hence, $\alpha^* \colon \cat D(A,-) \to \cat D(\hocolim_p T_{\cat Q}(X_{-p}),-)$ is an isomorphism in $\dercomp{\opp{\cat D}}$, and we conclude that $\holim_p (\alpha^*_{-p},0)$ is also an isomorphism in $\dercomp{\opp{\cat D}}$, as claimed.
\end{proof}
Finally, we prove the essential surjectivity of $H^0(\Tot) \colon H^0(\Tw^-(\cat Q)) \to H^0(\hprojmin{\cat Q}^{\mathrm{hfp}})$, which completes the proof of Theorem \ref{thm:hlc_tw_equiv}:
\begin{prop} \label{prop:Twmin_hprojhfp_essimg}
Let $\cat Q$ be a hlc dg-category, and let $M \in \hprojmin{\cat Q}^{\mathrm{hfp}}$. Then, there exists $X \in \Tw^-(\cat Q)$ such that $\Tot(X) \cong M$ in $H^0(\hproj{\cat Q})$. 
\end{prop}
\begin{proof}
It is sufficient to show that there exist $X \in \Tw^-(\cat Q)$ and a quasi-isomorphism $\Tot(X) \to M$, for both $\Tot(Y)$ and $M$ are h-projective. We recall Lemma \ref{lemma:hproj_tsubcat_fp_pretr} and we apply Proposition \ref{prop:dginj_resolution} with:
\begin{itemize}
\item $\cat D = \hprojmin{\cat Q}^{\mathrm{hfp}}$, which is a strongly pretriangulated full dg-subcategory of $\hproj{\cat Q}$ and has a t-structure (Lemma \ref{lemma:hfp_tstruct}).
\item $\cat Q = \cat Q$ viewed as a full subcategory of $\hprojmin{\cat Q}^{\mathrm{hfp}}$ via the Yoneda embedding.
\end{itemize}
Notice that $\cat D = \cat D^-$, and the dg-functor $T_{\cat Q} \colon \Tw^-_\bd (\cat Q) \to \cat D$ is precisely the totalisation $\Tot \colon \Tw^-_\bd (\cat Q) \to \hprojmin{\cat Q}^\mathrm{hfp}$. The above data satisfy the hypotheses of Proposition \ref{prop:dginj_resolution}: $\cat Q$ is cohomologically concentrated in negative degrees, it lies in the aisle $\hprojmin{\cat Q}^\mathrm{hfp}_{\leq 0}$, and every $N \in \cat D$ is such that $H^0(N) \in \Modfp{H^0(\cat Q)}$ is finitely generated, so that we have an epimorphism
\begin{equation*}
H^0(\cat Q(-,B)) \to H^0(N),
\end{equation*}
and by the  Yoneda Lemma this is induced by a closed degree $0$ morphism $\cat Q(-,B) \to N$.

Now, for simplicity, assume that $H^i(M) = 0$ for all $i>0$. By Proposition \ref{prop:dginj_resolution}, we find a sequence $(X_{-p})_{p \geq 0}$ of twisted complexes and maps $\alpha_{-p} \colon \Tot(X_{-p}) \to M$ such that the diagram
\begin{equation*}
\begin{gathered}
\xymatrix{
\Tot(X_{-p}) \ar[r]^-{\alpha_{-p}} \ar[d]_{j_{-p,-p-1}} & M \\
\Tot(X_{-p-1}) \ar[ur]_{\alpha{-p-1}}
}
\end{gathered}
\end{equation*}
is (strictly) commutative in $\cat D$. 

We may apply Proposition \ref{prop:resolution_twminus_colimit} and find $X \in \Tw^-(\cat Q)$ and such that $X_{-p} = \sigma_{\geq -p} X$. We know from Proposition \ref{prop:tot_twmin_colim} that $\Tot(X)$ together with the natural inclusions $j_{-p} \colon \Tot(X_{-p}) \to \Tot(X)$ is the colimit of the sequence $(\Tot(X_{-p}) \xrightarrow{j_{-p,-p-1}} \Tot(X_{-p-1}))_p$, so the maps $\alpha_{-p}$ induce a map $\alpha \colon \Tot(X) \to M$ such that $\alpha \circ j_{-p} = \alpha_{-p}$. Given $i \in \mathbb Z$, these relations give in particular commutative diagrams:
\begin{equation*}
\begin{gathered}
\xymatrix{
H^i(\Tot(\sigma_{\geq -p} X)) \ar[r]^-{H^i(\alpha_{-p})} \ar[d]_{H^i(j_{-p})} & H^i(M) \\
H^i(\Tot(X)). \ar[ur]_{H^i(\alpha)}
}
\end{gathered}
\end{equation*}
Now, choose any $p$ such that $i > -p$. we know from Lemma \ref{lemma:truncation_cohomology_constant} that $H^i(j_{-p})$ is an isomorphism. Moreover, we know from Proposition \ref{prop:dginj_resolution} that $H^i(\alpha_{-p})$ is also an isomorphism. Hence, $H^i(\alpha)$ is an isomorphism; since $i \in \mathbb Z$ is arbitrary, we conclude that $\alpha \colon \Tot(X) \to M$ is a quasi-isomorphism, as claimed.
\end{proof}

\section{Derived projectives and injectives} \label{section:setting}
\subsection{Basic definitions and properties}
\begin{defin}[{\cite[\S 5.1]{rizzardo-vdb-nonFM}}] \label{defin_dginj_dgproj}
Let $\cat T$ be a triangulated category with t-structure, and as usual denote
\begin{equation*}
H^0 = \tau_{\leq 0}\tau_{\geq 0} \colon \cat T \to \cat T^\heartsuit.
\end{equation*}
\begin{enumerate}
\item \label{item:dgproj_def} Assume that for all projectives $P \in \Proj(\cat T^\heartsuit)$ in the heart the cohomological functor $\cat T^\heartsuit(P,H^0(-)) \colon \cat T \to \Mod{\basering k}$ is corepresentable:
\begin{equation*}
\cat T^\heartsuit(P,H^0(-)) \cong \cat T(S(P),-).
\end{equation*}
We say that $S(P)$ is the \emph{derived projective} associated to $P$. In this case, we also say that \emph{$\cat T$ has derived projectives}. If $\cat A$ is a pretriangulated dg-category such that $H^0(\cat A)$ has a t-structure, we shall say that \emph{$\cat A$ has derived projectives} if $H^0(\cat A)$ has this property.
\item \label{item:dginj_def} Dually, assume that for all injectives $I \in \Inj(\cat T^\heartsuit)$ in the heart, the cohomological functor $\cat T^\heartsuit(H^0(-),I) \colon \opp{\cat T} \to \Mod{\basering k}$ is representable:
\begin{equation*}
\cat T^\heartsuit(H^0(-),I) \cong \cat T(-,L(I)).
\end{equation*}
We say that $L(I) \in \cat T$ is the \emph{derived injective} associated to $I$. In this case, we also say that \emph{$\cat T$ has derived injectives}. If $\cat A$ is a pretriangulated dg-category such that $H^0(\cat A)$ has a t-structure, we shall say that \emph{$\cat A$ has derived injectives} if $H^0(\cat A)$ has this property.  
\end{enumerate}
\end{defin}

\begin{remark}
Derived injectives over non-positively graded dg-algebras, with respect to the standard t-structure, have been investigated in \cite{shaul}.
\end{remark}

We now sum up from \cite[\S 5.1]{rizzardo-vdb-nonFM} some basic properties of derived injectives and projectives:
\begin{prop}
Let $\cat T$ be a triangulated category with t-structure.
\begin{enumerate}
\item Assume that $\cat T$ has derived projectives as in Definition \ref{defin_dginj_dgproj} part \ref{item:dgproj_def}. Then, for any $P \in \Proj(\cat T^\heartsuit)$, we have:
\begin{align}
S(P) & \in \cat T_{\leq 0}, \\
H^0(S(P)) & \cong P.
\end{align}
If $Q \in \Inj(\cat T^\heartsuit)$ is another projective, we have
\begin{equation}
\cat T(S(P),S(Q)[i]) \cong \begin{cases} \cat T^\heartsuit(P,Q) & \text{if $i=0$}, \\ 0 & \text{if $i > 0$,} \end{cases}
\end{equation}
for $i \in \mathbb Z$. In particular, $S \colon \Proj(\cat T^\heartsuit) \to \cat T, P \mapsto S(P)$ defines a fully faithful functor. Its essential image, which is the full subcategory of derived projectives of $\cat T$, is denoted by $\DGProj(\cat T)$. If $\cat T=H^0(\cat A)$ for some dg-category $\cat A$, we simplify notation and write $\DGProj(H^0(\cat A)) = \DGProj(\cat A)$, viewing it as a full dg-subcategory of $\cat A$.
\item Dually, assume that $\cat T$ has derived injectives as in Definition \ref{defin_dginj_dgproj} part \ref{item:dginj_def}. Then, for any $I \in \Inj(\cat T^\heartsuit)$, we have:
\begin{align}
L(I) & \in \cat T_{\geq 0}, \\
H^0(L(I)) & \cong I.
\end{align}
If $J \in \Inj(\cat T^\heartsuit)$ is another injective, we have
\begin{equation}
\cat T(L(I),L(J)[i]) \cong \begin{cases} \cat T^\heartsuit(I,J) & \text{if $i=0$}, \\ 0 & \text{if $i > 0$,} \end{cases}
\end{equation}
for $i \in \mathbb Z$. In particular, $L \colon \Inj(\cat T^\heartsuit) \to \cat T, I \mapsto L(I)$ defines a fully faithful functor. Its essential image, which is the full subcategory of derived injectives of $\cat T$, is denoted by $\DGInj(\cat T)$. If $\cat T=H^0(\cat A)$ for some dg-category $\cat A$, we simplify notation and write $\DGInj(H^0(\cat A)) = \DGInj(\cat A)$, viewing it as a full dg-subcategory of $\cat A$.
\end{enumerate}
\end{prop}
\begin{remark} \label{remark:dgproj_stable_tminus}
If $\cat T$ has derived projectives as in Definition \ref{defin_dginj_dgproj} part \ref{item:dginj_def}, then the same is true for $\cat T^-$ and $\DGProj(\cat T^-) = \DGProj(\cat T)$. This follows immediately from the fact that $(\cat T^-)^\heartsuit = \cat T^\heartsuit$ and the fact that $S(P) \in \cat T^-$ for all injectives $P \in \Proj(\cat T^\heartsuit)$. Dually, if $\cat T$ has derived injectives as in Definition \ref{defin_dginj_dgproj} part \ref{item:dgproj_def}, then the same is true for $\cat T^+$ and $\DGProj(\cat T^+) = \DGProj(\cat T)$.
\end{remark}
\begin{remark} \label{remark:dgproj_coprod}
Assume that $\cat T$ has derived projectives, and let $\{P_i : i \in I \}$ be a family of objects in $\Proj(\cat T^\heartsuit)$ such that $P=\bigoplus_i P_i$ exists in $\cat T$. Then, $H^0(P)$ is a coproduct of the $P_i$ in $\cat T^\heartsuit$ and in particular is in $\Proj(\cat T^\heartsuit)$. To see this, first note that $\bigoplus_i P_i \in \cat T_{\leq 0}$, because
\begin{equation*}
\cat T(\bigoplus_i P_i, Y) \cong \prod_i \cat T(P_i,Y) = 0,
\end{equation*}
for all $Y \in \cat T_{\geq 1}$, since $P_i \in \cat T^\heartsuit \subseteq \cat T_{\leq 0}$. Then, for all $A \in \cat T^\heartsuit$ we have:
\begin{align*}
\prod_i \cat T^\heartsuit(P_i,A) & \cong \cat T(P, A) \\
& \cong \cat T(\tau_{\geq 0}(P), A) \\
& \cong \cat T^\heartsuit(H^0(P), A),
\end{align*} 
naturally in $A$. Being a coproduct of projectives in $\cat T^\heartsuit$, the object $H^0(P)$ is itself projective.

Moreover, we have
\begin{align*}
\cat T(S(H^0(P)),X) & \cong \cat T^\heartsuit(H^0(P),H^0(X)) \\
& \cong \prod_i \cat T^\heartsuit(P_i,H^0(X)) \\
& \cong \prod_i \cat T(S(P_i),X) \\
& \cong \cat T(\bigoplus_i S(P_i),X),
\end{align*}
naturally in $X \in \cat T$. We deduce that $S(H^0(P)) \cong \bigoplus_i S(P_i)$; in other words, $S(-)$ commutes with coproducts. Dually, one can prove that $L(-)$ commutes with direct products if $\cat T$ has derived injectives.
\end{remark}
Derived injectives or projectives can be used to make resolutions of objects of triangulated categories with t-structures, much like projectives are used to resolve objects in abelian categories. The starting point for this is the following definition:
\begin{defin}
Let $\cat T$ be a triangulated category with t-structure. We say that \emph{$\cat T$ has enough derived projectives} if $\cat T$ has derived projectives as in Definition \ref{defin_dginj_dgproj} and moreover the heart $\cat T^\heartsuit$ has enough projectives. Dually, we say that \emph{$\cat T$ has enough derived injectives} if $\cat T$ has derived injectives and $\cat T^\heartsuit$ has enough injectives. 

If $\cat A$ is a pretriangulated dg-category with a t-structure, we say that \emph{$\cat A$ has enough derived projectives (or injectives)} if $H^0(\cat A)$ has this property.
\end{defin}
\begin{remark} \label{remark:derinj_firststep}
If $\cat T$ has enough derived projectives, then for any given $A \in \cat T$ we can find a projective $P \in \Proj(\cat T^\heartsuit)$ and an epimorphism $\overline{\alpha} \colon P \twoheadrightarrow H^0(A)$. From this, directly applying the definition, we find a morphism $\alpha \colon S(P) \to A$ in $\cat T$ such that $H^0(\alpha)$ is identified with $\overline{\alpha}$ via the isomosphism $H^0(S(P)) \cong P$.
\end{remark}
In the following example, we explicitly characterise the derived projectives of $\hproj{\cat Q}^{\mathrm{hfp}}$.
\begin{example} \label{example:derproj_yoneda}
Let $\cat Q$ be a hlc dg-category such that $H^0(\cat Q)$ is Karoubian. Then, we know from Theorem \ref{thm:hlc_tw_equiv} that the triangulated categories $\dercomp{\cat Q}^{\mathrm{hfp}}$
and $\dercompmin{\cat Q}^{\mathrm{hfp}}$ have natural t-structures with the same heart given by $\Modfp{H^0(\cat Q)}$. By \cite[Proposition A.14]{lowen-vandenbergh-deformations-abelian}, we know that the projectives of $\Modfp{H^0(\cat Q)}$ are precisely given by (the essential image of) $H^0(\cat Q)$, namely the modules isomorphic to the representables $H^0(\cat Q)(-,A)$ for some $A \in \cat Q$. By the Yoneda lemma, we have for all $M \in \hproj{\cat Q}^{\mathrm{hfp}}$:
\begin{align*}
\Modfp{H^0(\cat Q)}(H^0(\cat Q)(-,A),H^0(M)) &\cong H^0(M)(A)\\
&\cong \dercomp{\cat Q}(\cat Q(-,A), M)
\end{align*}
We conclude that $\DGProj(\hproj{\cat Q}^\mathrm{hfp})$ is precisely given by the dg-modules isomorphic in $H^0(\hproj{\cat Q})$ to the representable dg-modules. The same is true for $\DGProj(\hprojmin{\cat Q}^\mathrm{hfp})$ by Remark \ref{remark:dgproj_stable_tminus}. A little more precisely, we have that the Yoneda embedding $\cat Q \hookrightarrow \hprojmin{\cat Q}^\mathrm{hfp}$ induces a quasi-equivalence
\begin{equation}
\cat Q \xrightarrow{\approx} \DGProj(\hprojmin{\cat Q}^\mathrm{hfp}) = \DGProj(\hproj{\cat Q}^\mathrm{hfp}).
\end{equation}
Since $\Tw^-(\cat Q)$ is quasi-equivalent to $\hprojmin{\cat Q}^\mathrm{hfp}$ via the totalisation functor, we also deduce that $\DGProj(\Tw^-(\cat Q))$ is the closure in $H^0(\Tw^-(\cat Q))$ of the objects of the form $Q=(Q,0) \in \Tw^-(\cat Q)$.
\end{example}
The following lemma is an improvement of Lemma \ref{lemma:hproj_tsubcat_fp_pretr}.
\begin{lemma} \label{lemma:hprojhfp_enoughderivedproj}
Let $\cat Q$ be a hlc dg-category such that $H^0(\cat Q)$ is Karoubian. Then, the full dg-subcategory $\hproj{\cat Q}^{\mathrm{hfp}}$ of $\hproj{\cat Q}$ defined in \eqref{eq:hproj_tsubcat_fp} has enough derived projectives. Moreover, if $H^0(\cat Q)$ is closed under countable coproducts, then the same is true for $H^0(\hproj{\cat Q}^{\mathrm{hfp}})$
\end{lemma}
\begin{proof}
Clearly the homotopy category of $\hproj{\cat Q}^{\mathrm{hfp}}$ is equivalent to the triangulated category $\dercomp{\cat Q}^{\mathrm{hfp}}$ (defined in \eqref{eq:dercat_tsubcat_fp}), and from Theorem \ref{thm:hlc_tw_equiv} we know that it has a t-structure induced from $\dercomp{\cat Q}$ with heart $\Modfp{H^0(\cat Q)}$. Since $H^0(\cat Q)$ is Karoubian, we know that the representable $H^0(\cat Q)$-modules are precisely the projectives of the heart $\Modfp{H^0(\cat Q)}$, and by the above Example \ref{example:derproj_yoneda} we know that for every projective $H^0(\cat Q)(-,A)$ the associated derived projective is $\cat Q(-,A)$. Moreover, the heart $\Modfp{H^0(\cat Q)}$ has enough projectives, hence by definition $\hproj{\cat Q}^{\mathrm{hfp}}$ has enough derived projectives.

Next, we assume that $H^0(\cat Q)$ has countable coproducts, and we show that the same is true for $\dercomp{\cat Q}^{\mathrm{hfp}}$. Let $(M_j)_{j\in J}$ be a countable family of objects there. Then, for all $i \in \mathbb Z$, we have projective presentations
\begin{equation*}
H^0(\cat Q)(-,A_j) \to H^0(\cat Q)(-,B_j) \to H^i(M_j).
\end{equation*}
Taking direct sums in $\dercomp{\cat Q}$ we find an exact sequence:
\begin{equation*}
\bigoplus_{j \in J} H^0(\cat Q)(-,A_j) \to \bigoplus_{j \in J} H^0(\cat Q)(-,B_j) \to \bigoplus_{j \in J} H^i(M_j) \cong H^i (\bigoplus_{j \in J} M_j ).
\end{equation*}
By hypothesis, $\oplus_j H^0(\cat Q)(-,A_j)$ and $\oplus_j H^0(\cat Q)(-,B_j)$ are representable (say, respectively by objects $A$ and $B$ in $H^0(\cat Q)$) and we get an exact sequence:
\begin{equation*}
H^0(\cat Q)(-,A) \to H^0(\cat Q)(-,B) \to H^i(\bigoplus_{j \in J} M_j),
\end{equation*}
namely $H^i(M) \in \Modfp{H^0(\cat Q)}$, as desired.
\end{proof}

For a given dg-category, the property of being homotopically locally coherent and with Karoubian $H^0$ is precisely what makes it a ``dg-category of derived projectives of a dg-category with enough derived projectives'', in virtue of Example \ref{example:derproj_yoneda}
and the following result:
\begin{lemma} \label{lemma:dgproj_hlc_karoubian}
Let $\cat A$ be a pretriangulated dg-category with a t-structure and enough derived projectives. Then, the dg-category $\DGProj(\cat A)$ is hlc and $H^0(\DGProj(\cat A))$ is Karoubian. 

Moreover, if the t-structure on $\cat A$ is closed under countable direct sums (Definition \ref{def:tstruct_closed_directsums}), then the coproduct $\bigoplus_i S(P_i)$ in $H^0(\cat A)_{\leq 0}$ of any countable collection of objects in $H^0(\DGProj(\cat A))$ lies in $H^0(\DGProj(\cat A))$. In particular, $H^0(\DGProj(\cat A))$ is closed under countable direct sums.
\end{lemma}
\begin{proof}
For simplicity, set $\cat Q = \DGProj(\cat A)$. We know that $S(-) \colon \Proj(H^0(\cat A)^\heartsuit) \to H^0(\cat Q)$ is an equivalence, hence $H^0(\cat Q)$ is additive, coherent and Karoubian since $\Proj(H^0(\cat A)^\heartsuit)$ is such (by the dual of \cite[Remark A.13]{lowen-vandenbergh-deformations-abelian}. To go on to show that $\cat Q$ is hlc, we let $Q \in \cat Q$ and consider $H^i(\cat Q(S(P),Q))$ for any given projective $P$ in the heart. We have:
\begin{align*}
H^i(\cat Q(S(P),Q)) &\cong H^0(\cat A(S(P),Q[i])) \\
&\cong H^0(\cat A)^\heartsuit(P,H^i(Q)).
\end{align*}
Next, consider a projective presentation of $H^i(Q)$, which exists since $H^0(\cat A)^\heartsuit$ has enough projectives:
\begin{equation*}
P_1 \to P_0 \to H^i(Q) \to 0.
\end{equation*}
Since $P$ is projective, we get an exact sequence
\begin{equation*}
H^0(\cat A)^\heartsuit(P,P_1) \to H^0(\cat A)^\heartsuit(P,P_0) \to H^0(\cat A)^\heartsuit(P,H^i(Q)) \to 0.
\end{equation*}
Next, recalling that $S(-) \colon \Proj(H^0(\cat A)^\heartsuit) \to H^0(\cat Q)$ is an equivalence, we get an exact sequence:
\begin{equation*}
H^0(\cat Q)(S(P),S(P_1)) \to H^0(\cat Q)(S(P),S(P_0)) \to H^i(\cat Q(S(P),Q)) \to 0.
\end{equation*}
This sequence is natural in $P$, hence also in $S(P) \in H^0(\cat Q)$, since $S(-)$ is fully faithful. We conclude that $H^i(\cat Q(-,Q))$ is finitely presented, as desired.

Finally, the second part of the claim follows from Remark \ref{remark:dgproj_coprod}. Indeed, the coproduct $\bigoplus_i P_i$ exists in $H^0(\cat A)_{\leq 0}$ since this aisle is closed under direct sums, and then we have that
\begin{equation*}
S(H^0(P)) \cong \bigoplus_i S(P_i) \in H^0(\cat Q). \qedhere
\end{equation*}
\end{proof}
\subsection{Functors preserving derived injectives or projectives}
In this part we give sufficient conditions in order that a given exact functor $F \colon \cat T \to \cat S$ between categories which have derived projectives or injectives actually preserves the subcategories of derived projectives or injectives. To this purpose, we start by investigating the behaviour of adjoints with respect to t-structures. Recall that $F$ is right (left) t-exact if 
$F(\cat S_{\le 0})\subset \cat T_{\le 0} $ 
($F(\cat S_{\ge 0})\subset \cat T_{\ge 0} $) holds. Furthemore we say that $F$ is t-exact if it is both left and right t-exact. We recall the following standard result.
\begin{lemma} \label{lemma:tstruct_adjoint_leftright}
Assume we are given an adjunction
$
F \dashv G \colon \cat T \leftrightarrows \cat S
$
of exact functors between triangulated categories with t-structures. Then $F$ is right t-exact if and only if $G$ left t-exact.
\end{lemma}
\begin{proof}
Assume that $G(\cat S_{\geq 0}) \subseteq \cat T_{\geq 0}$. Playing with shifts, we notice that this implies that $G(\cat S_{\geq n}) \subseteq \cat T_{\geq n}$ for all $n \in \mathbb Z$. Now, let $A \in \cat T_{\leq 0}$. Recall that $F(A) \in \cat S_{\leq 0}$ is equivalent to
\begin{equation*}
\cat S(F(A),B) = 0, \quad \forall\, B \in \cat S_{>0}.
\end{equation*}
On the other hand, for any $B \in \cat S_{>0}$, we have
\begin{equation*}
\cat S(F(A),B) \cong \cat T(A,G(B)) \cong 0,
\end{equation*}
since by hypothesis $G(B) \in \cat T_{>0}$. The other implication is proved in the same fashion.
\end{proof}
Any exact functor $F \colon \cat T \to \cat S$ between triangulated categories with t-structures induces a functor between the hearts:
\begin{equation}
\begin{split}
F^\heartsuit \colon \cat T^\heartsuit & \to \cat S^\heartsuit, \\
A & \mapsto H^0(F(A)) = \tau_{\leq 0} \tau_{\geq 0}F(A).
\end{split}
\end{equation}
This formula simplifies to $\tau_{\geq 0} F(A)$ ($\tau_{\leq 0} F(A)$) is $F$ if right (left) t-exact (in which case $F^{\heartsuit}:\cat T^{\heartsuit} \to \cat S^{\heartsuit}$ is right (left) exact)
and to $A \mapsto F(A)$ if $F$ is t-exact (in which case $F^{\heartsuit}$ is exact).
\begin{lemma} \label{lemma:tstruct_adjoint_hearts}
Assume we are given an adjunction
\begin{equation*}
F \dashv G \colon \cat T \leftrightarrows \cat S
\end{equation*}
of exact functors between triangulated categories with t-structures. Then, if  $F$ is right t-exact (or equivalently
$G$ is left t-exact) then the above adjunction induces an adjunction
\begin{equation*}
F^\heartsuit \dashv G^\heartsuit \colon \cat T^\heartsuit \leftrightarrows \cat S^\heartsuit
\end{equation*}
of the functors induced between the hearts.
\end{lemma}
\begin{proof}
We may compute, for any $A \in \cat T^{\heartsuit}$ and $B \in \cat S^{\heartsuit}$:
\begin{align*}
\cat S^\heartsuit(F^\heartsuit(A),B)
&\cong \cat S^\heartsuit(\tau_{\geq 0}F(A),B) && \text{($F$ is right t-exact)}\\
&\cong \cat S(F(A),B) &&(B \in \cat S^\heartsuit \subseteq \cat S_{\geq 0}) \\
&\cong \cat T(A,G(B)) \\
&\cong \cat T(A,\tau_{\le 0}G(B)) &&(A \in \cat T^\heartsuit \subseteq \cat T_{\leq 0}) \\ 
&\cong \cat T^\heartsuit(A,G^\heartsuit(B)) && \text{($G$ is left t-exact).}
 \qedhere
\end{align*}
\end{proof}
It is well known that a functor between abelian categories preserves projectives whenever it has an exact right adjoint. A similar result is true in the framework of t-structures and derived projectives and injectives:
\begin{prop} \label{prop:restr_dgproj_preserve}
Let $F \colon \cat T \to \cat S$ be an exact functor between triangulated categories with t-structures. If $\cat T$ and $\cat S$ have derived projectives and $F$ has a t-exact right adjoint $G$ then $F$ preserves the derived projectives, namely it restricts to a functor
\begin{equation*}
F|_{\DGProj(\cat T)} \colon \DGProj(\cat T) \to \DGProj(\cat S).
\end{equation*}

Dually, if $\cat T$ and $\cat S$ have derived injectives and $F$ has a t-exact left adjoint $G'$, then $F$ preserves the derived injectives, namely it restricts to a functor
\begin{equation*}
F|_{\DGInj(\cat T)} \colon \DGInj(\cat T) \to \DGInj(\cat S).
\end{equation*}
\end{prop}
\begin{proof}
We prove the statement about derived projectives, the other one being dual. Let $S(P) \in \DGProj(\cat T)$ be a fixed derived projective, associated to some $P \in \Proj(\cat T^\heartsuit)$. We are going to prove that $F(S(P)) \cong S(F^\heartsuit(P))$. For any $B \in \cat S$, we have:
\begin{align*}
\cat S(S(F^\heartsuit(P)),B) &= \cat S^\heartsuit(F^\heartsuit(P),H^0(B)) \\
&\cong \cat T^\heartsuit(P,G^\heartsuit(H^0(B))) \qquad (F^\heartsuit \dashv G^\heartsuit) \\
&\cong \cat T^\heartsuit(P,H^0(G(B))) \qquad \text{($G$ is t-exact, hence $G^\heartsuit \circ H^0 \cong H^0 \circ G$)} \\
&\cong \cat T(S(P),G(B)) \\
&\cong \cat T (F(S(P)),B).
\end{align*}
We applied the above Lemma \ref{lemma:tstruct_adjoint_hearts} and the fact that $F^\heartsuit(P)$ is projective since $P \in \Proj(\cat T^\heartsuit)$ and $F^\heartsuit$ has an exact right adjoint $G^\heartsuit$.
\end{proof}

\section{The (re)construction} \label{section:reconstruction}
The results of the previous sections (see  Theorem \ref{thm:hlc_tw_equiv}, Lemma
  \ref{lemma:hprojhfp_enoughderivedproj}, Example
  \ref{example:derproj_yoneda}, see also Remark \ref{remark:tcohom_holim_leftsep_applicable})
 are summarized in the following theorem which provides a method for constructing triangulated categories with a t-structure and with enough derived projectives.
\begin{thm}[Construction] \label{thm:construction} Let $\cat Q$ be a hlc dg-category such that $H^0(\cat Q)$ is Karoubian. Then, the dg-category $\hproj{\cat Q}^\mathrm{hfp}$ defined by
\begin{equation*}
\hproj{\cat Q}^{\mathrm{hfp}} = \{M \in \hproj{\cat Q} : H^i(M) \in \Modfp{H^0(\cat Q)} \ \ \forall\, i\}
\end{equation*}
has a non-degenerate t-structure whose heart is the category $\Modfp{H^0(\cat Q)}$, has enough derived
projectives, and $\DGProj(\hproj{\cat Q}^\mathrm{hfp})$ is the closure of $\cat Q \hookrightarrow \hproj{\cat Q}^{\mathrm{hfp}}$ under isomorphisms in $H^0(\hproj{\cat Q}^\mathrm{hfp})$; also, $H^0(\hproj{\cat Q}^\mathrm{hfp})$ is closed under countable copruducts if $H^0(\cat Q)$ is. Moreover, the totalisation dg-functor induces a quasi-equivalence:
\begin{equation*}
\Tot \colon \Tw^-(\cat Q) \to \hprojmin{\cat Q}^{\mathrm{hfp}}.
\end{equation*}
In particular, $\Tw^-(\cat Q)$ has a non-degenerate right bounded t-structure. This t-structure is closed under countable coproducts (Definition \ref{def:tstruct_closed_directsums}) if $H^0(\cat Q)$ is closed under countable coproducts.
\end{thm}
A natural question is whether the above Theorem \ref{thm:construction}
can be ``inverted''. Namely, given a pretriangulated category $\cat A$ with a non-degenerate right bounded t-structure with enough derived projectives, can we reconstruct $\cat A^-$ as $\Tw^-(\DGProj(\cat A))$?  Theorem \ref{thm:comparison} below provides positive answers to these question, provided that we also assume closure under countable coproducts.
\subsection{Reconstruction} \label{subsection:reconstructing} We fix a  pretriangulated dg-category $\cat A$ with a non-degenerate right bounded t-structure, with enough derived projectives, and which is closed under countable coproducts. Let $\cat Q = \DGProj(\cat A)$, and let $j \colon \cat Q \hookrightarrow \cat A$ be the inclusion. We can compose the Yoneda embedding $\cat A \hookrightarrow \hproj{\cat A}$ with the restriction quasi-functor $\res_j \colon \hproj{\cat A} \to \hproj{\cat Q}$ (recall \eqref{eq:res_qfun}), hence obtaining a quasi-functor
\begin{equation} \label{eq:qfun_comparison}
\cat A \to \hproj{\cat Q},
\end{equation}
which in $H^0$ gives the following exact functor between triangulated categories:
\begin{equation} \label{eq:exfun_comparison}
\begin{split}
H^0(\cat A) &\to \dercomp{\cat Q}, \\
A & \mapsto \cat A(j(-),A).
\end{split}
\end{equation}
Notice that $\cat Q$ is hlc and its $H^0$ is Karoubian by Lemma \ref{lemma:dgproj_hlc_karoubian}, so $\Tw^-(\cat Q)$ is quasi-equivalent to $\hprojmin{\cat Q}^{\mathrm{hfp}}$ by the above Theorem \ref{thm:construction}. The key result we are going to show is the following:
\begin{thm}[Reconstruction] \label{thm:comparison} Let $\cat A$ be a pretriangulated dg-category with a non-degenerate, right bounded t-structure with enough derived projectives, and which is closed under countable coproducts. Let $\cat Q = \DGProj(\cat A)$. The quasi-functor \eqref{eq:qfun_comparison}
\begin{equation*}
\cat A^- \to \hproj{\cat Q}
\end{equation*}
is t-exact and induces an isomorphism in the homotopy category $\Hqe$ between the dg-categories $\cat A^-$ and $\hprojmin{\cat Q}^{\mathrm{hfp}} \approx \Tw^-(\cat Q)$.
\end{thm}

\subsection{Theorem \ref{thm:comparison}: preparations}
The proof of Theorem \ref{thm:comparison} is achieved by applying Proposition \ref{prop:dginj_resolution} in order to resolve $A \in \cat A$ with a sequence of twisted complexes of derived projectives. Before going on with the actual proof, we take care of the setting and the preparatory results.

We fix a pretriangulated dg-category $\cat A$ with a non-degenerate right bounded t-structure with enough derived projectives, which is closed under countable coproducts (so that the aisles $H^0(\cat A)_{\leq M}$ are closed under countable coproducts). We set $\cat Q=\DGProj(\cat A)$, which is a hlc dg-category such that $H^0(\cat Q)$ is Karoubian. Moreover, $\cat Q \subseteq \cat A_{\leq 0}$.
\begin{itemize}
\item We can assume that $\cat A$ is strongly pretriangulated, by replacing and identifying it with its pretriangulated hull: $\pretr{\cat A} = \cat A$.
\item Consider the inclusion dg-functor $j \colon \cat Q \to \cat A$. Since $\cat A$ is strongly pretriangulated, we have an induced fully faithful dg-functor $j' \colon \pretr{\cat Q} \hookrightarrow \cat A$. Recalling the properties of $\Ind$ in \S \ref{subsection:twcomp_qeq}, we know that the restriction along the natural embedding $\cat Q \hookrightarrow \pretr{\cat Q}$ induces a dg-equivalence
\begin{equation} \label{eq:pretr_dgeq}
\compdg{\pretr{\cat Q}} \xrightarrow{\sim} \compdg{\cat Q},
\end{equation}
which gives also an equivalence between the derived categories:
\begin{equation}
\dercomp{\pretr{\cat Q}} \xrightarrow{\sim}  \dercomp{\cat Q}.
\end{equation}
Clearly, this equivalence maps
\begin{equation}
\cat A(j'(-),A) \mapsto \cat A(j(-),A),
\end{equation}
for all $A \in \cat A$.
\item Recalling \eqref{eq:twbd_pretr}, the totalisation dg-functor $\Tot_{\cat Q}$ restricts to a dg-functor
\begin{equation*}
\Tot_{\cat Q} \colon \Tw^-_\bd(\cat Q) \to \pretr{\cat Q},
\end{equation*}
and composing with $j' \colon \pretr{\cat Q} \hookrightarrow \cat A$, we get a dg-functor (recall also \eqref{eq:TQ_tot})
\begin{equation*}
T_{\cat Q} \colon \Tw^-_\bd (\cat Q) \to \cat A
\end{equation*}
with the property that for all $Y \in \Tw^-_\bd(\cat Q)$ we have an isomorphism 
\begin{equation} \label{eq:tot_notational_iso}
\Tot_{\cat Q}(Y) \cong \cat A(j(-),T_{\cat Q}(Y)),
\end{equation}
in $\compdg{\cat Q}$, natural in $Y$.
\end{itemize}
\begin{lemma} \label{lemma:functor_preserves_tstruct}
Choose an inverse $S^{-1}$ of the equivalence $S(-) \colon \Proj(H^0(\cat A)^\heartsuit) \to \DGProj(\cat A)$, and let $A \in \cat A$. For all $i \in \mathbb Z$, we have an isomorphism
\begin{equation} \label{eq:dgproj_tstruct_compare}
H^i(\cat A(j(-), A)) \cong H^0(\cat A)^\heartsuit(S^{-1}(-),H^i(A)).
\end{equation}
In particular, the functor
\begin{equation} \label{eq:nerve_dgproj}
\begin{split}
H^0(\cat A) & \to \dercomp{\cat Q}, \\
A & \mapsto \cat A(j(-),A)
\end{split}
\end{equation}
is t-exact.
\end{lemma}
\begin{proof}
The isomorphism \eqref{eq:dgproj_tstruct_compare} follows from the very definition of derived projectives. t-exactness now follows from the fact that the t-structure on $\dercomp{\cat Q}$ is non-degenerate. Indeed, let $A \in H^0(\cat A)_{\leq n}$. By \eqref{eq:aisles_incl_cohom}, we know that $H^i(A)=0$ for all $i> n$. Thanks to \eqref{eq:dgproj_tstruct_compare}, we find out that $H^i(\cat A(j(-), A))=0$ for $i>n$, and this means that $\cat A(j(-), A) \in \dercomp{\cat Q}_{\leq n}$. A similar argument shows that $H^0(\cat A)_{\geq n}$ is mapped to $\dercomp{\cat Q}_{\geq n}$.
\end{proof}
\begin{remark} \label{remark:induced_hearts}
The functor induced by the above \eqref{eq:nerve_dgproj} between the hearts is precisely
\begin{equation}
\begin{split}
H^0(\cat A)^\heartsuit & \to \Mod{H^0(\cat Q)}, \\
A & \mapsto H^0(\cat A)^\heartsuit(S^{-1}(-),A).
\end{split}
\end{equation}
This is proven in \cite[Proposition 6.25]{lowen-vandenbergh-deformations-abelian} to induce an equivalence between $H^0(\cat A)^\heartsuit$ and $\Modfp{H^0(\cat Q)}$. That result will follow from the proof of Theorem \ref{thm:comparison}.
\end{remark}
\begin{lemma} \label{lemma:prep1}
Let $A \in \cat A_{\leq M}$. Then, there is a sequence $(X_{M-p} \to X_{M-p-1})_{p \geq 0}$ in $\Tw^-_\bd (\cat Q)$, with $T_{\cat Q}(X_{M-p}) \in \cat A_{\leq M}$, and closed degree $0$ maps $\alpha_{M-p} \colon T_{\cat Q}(X_{M-p}) \to A$ such that the diagram
\begin{equation*}
\begin{gathered}
\xymatrix{
T_{\cat Q}(X_{M-p}) \ar[r]^-{\alpha_{M-p}} \ar[d]_{j_{M-p,M-p-1}} & A \\
T_{\cat Q}(X_{M-p-1}) \ar[ur]_{\alpha_{M-p-1}}
}
\end{gathered}
\end{equation*}
is strictly commutative in $\cat A$, and
\begin{equation} \label{eq:resol_dgproj_holim}
\holim_p (\alpha^*_{M-p}, 0) \colon \cat A(A,-) \to \holim_p \cat A(T_{\cat Q}(X_{M-p}), -)
\end{equation}
is an isomorphism in $\dercomp{\opp{\cat A}}$. In other words, $A$ together with the maps $\alpha_{M-p}$ is the homotopy colimit of $(T_{\cat Q}(X_{M-p}) \xrightarrow{j_{M-p,M-p-1}} T_{\cat Q}(X_{M-p-1})) _p$.
\end{lemma}
\begin{proof}
Proposition \ref{prop:dginj_resolution}, Corollary \ref{coroll:res_tstruct}.
\end{proof}
\begin{lemma} \label{lemma:prep2}
Let $A \in \cat A_{\leq M}$, and consider the sequence $(X_{M-p} \to X_{M-p-1})_p$ and the maps $\alpha_{M-p} \colon T_{\cat Q}(X_{M-p}) \to A$ given by Lemma \ref{lemma:prep1}. There exists $X \in \Tw^-(\cat Q)$ such that $\sigma_{\geq M-p} X = X_{M-p}$, and the morphisms 
\begin{equation*}
(\alpha_{M-p})_* \colon \Tot_{\cat Q}(X_{M-p}) \cong \cat A(j(-),T_{\cat Q}(X_{M-p})) \to \cat A(j(-),A)
\end{equation*}
induce a closed degree $0$ morphism in $\compdg{\cat Q}$
\begin{equation} \label{eq:TotX_dgproj_qis}
\Tot_{\cat Q}(X) \to \cat A(j(-),A)
\end{equation}
which is an isomorphism in $\dercomp{\cat Q}$. Moreover, the induced morphism (recall \eqref{eq:hocolim_inducedmap})
\begin{equation} \label{eq:hocolim_pretr_resol}
\hocolim_p ((\alpha_{M-p})_*,0) \colon \hocolim_p \cat A(j'(-),T_{\cat Q}(X_{M-p})) \to \cat A(j'(-),A)
\end{equation}
is an isomorphism in $\dercomp{\pretr{\cat Q}}$.
\end{lemma}
\begin{proof}
Upon shifting, assume $M=0$. The sequence $(X_{-p} \to X_{-p-1})_p$ is constructed using Proposition \ref{prop:dginj_resolution}, so Proposition \ref{prop:resolution_twminus_colimit} is applicable and  gives $X \in \Tw^-(\cat Q)$ such that $\sigma_{-p} X = X_{-p}$, and moreover 
\begin{equation*}
\Tot_{\cat Q}(X) \cong \varinjlim_p \Tot_{\cat Q}(X_{-p}) \approx \hocolim_p \Tot_{\cat Q}(X_{-p}) .
\end{equation*}
Thanks to the commutative diagram \eqref{eq:tot_hocolim_commutative}, the equivalence \eqref{eq:pretr_dgeq} and the isomorphism \eqref{eq:tot_notational_iso}, we only need to check that \eqref{eq:TotX_dgproj_qis} is a quasi-isomorphism. Let $i \in \mathbb Z$ and recall Lemma \ref{lemma:truncation_cohomology_constant}: $H^i(\Tot_{\cat Q}(X_{-p})) \to H^i(\Tot_{\cat Q}(X))$ is an isomorphism for $-p < i$, so it is enough to prove that
\begin{equation*}
H^i(\Tot_{\cat Q}(X_{-p})) \cong H^i(\cat A(j(-),T_{\cat Q}(X_{-p}))) \xrightarrow{H^i((\alpha_{-p})_*)} H^i(\cat A(j(-),A))
\end{equation*}
is an isomorphism for $-p < i$. Thanks to \eqref{eq:dgproj_tstruct_compare}, this is equivalent to proving that
\begin{equation*}
H^0(\cat A)^\heartsuit(S^{-1}(-),H^i(T_{\cat Q}(X_{-p}))) \xrightarrow{H^i(\alpha_{-p})_*} H^0(\cat A)^\heartsuit (S^{-1}(-),H^i(A))
\end{equation*}
is an isomorphism for $-p < i$. This follows from Proposition \ref{prop:dginj_resolution}, where we prove that
\begin{equation*}
H^i(\alpha_{-p}) \colon H^i(T_{\cat Q}(X_{-p})) \to H^i(A)
\end{equation*}
is an isomorphism for $-p < i$.
\end{proof}

\subsection{Theorem \ref{thm:comparison}: proof}
In order to prove Theorem \ref{thm:comparison}, it is enough to show that the functor \eqref{eq:exfun_comparison}
\begin{align*}
H^0(\cat A) & \to \dercomp{\cat Q}, \\
A & \mapsto \cat A(j(-),A)
\end{align*}
is fully faithful and its essential image is $\dercompmin{\cat Q}^\mathrm{hfp} \cong H^0(\hprojmin{\cat Q}^\mathrm{hfp})$, since t-exactness follows from Lemma \ref{lemma:functor_preserves_tstruct}.

\subsubsection*{Fully faithfulness} Since $\pretr{\cat Q} \hookrightarrow \cat A$ and $\dercomp{\pretr{\cat Q}} \cong \dercomp{\cat Q}$ via the restriction dg-functor, fully faithfulness of \eqref{eq:exfun_comparison} is equivalent to fully faithfulness of 
\begin{equation}
\begin{split}
H^0(\cat A) & \to \dercomp{\pretr{\cat Q}}, \\
A & \mapsto \cat A(j'(-),A).
\end{split}
\end{equation}
We compute, given $A, B \in \cat A$ (assuming for simplicity that $H^i(A)=0$ for all $i>0$):
\begin{align*}
\cat A(A,B) & \qis \holim_p \cat A(T_{\cat Q}(X_{-p}),B) \quad \text{(from \eqref{eq:resol_dgproj_holim})} \\
& \cong \holim_p \compdg{\pretr{\cat Q}}(\cat A(j'(-),T_{\cat Q}(X_{-p})), \cat A(j'(-),B)) \quad \text{(Yoneda)} \\
& \cong \compdg{\pretr{\cat Q}}(\hocolim_p \cat A(j'(-),T_{\cat Q}(X_{-p})), \cat A(j'(-),B)) \quad \text{(from \eqref{eq:strictholim_dgmod})}.
\end{align*}
Hence:
\begin{align}
H^0(\cat A(A,B)) & \cong \hocomp{\pretr{\cat Q}}(\hocolim_p \cat A(j'(-),T_{\cat Q}(X_{-p})), \cat A(j'(-),B)) \nonumber \\
& \cong \dercomp{\pretr{\cat Q}}(\hocolim_p \cat A(j'(-),T_{\cat Q}(X_{-p})), \cat A(j'(-),B)) \quad \tag{$\ast$} \label{tag:hocolim_hproj}  \\
& \cong \dercomp{\pretr{\cat Q}}(\cat A(j'(-),A), \cat A(j'(-),B)). \quad \text{(from \eqref{eq:hocolim_pretr_resol})} \nonumber
\end{align}
The isomorphism \eqref{tag:hocolim_hproj} follows from \eqref{eq:hproj_loc_iso} because
\begin{align*}
\hocolim_p \cat A(j'(-),T_{\cat Q}(X_{-p})) &= \hocolim_p \cat A(j'(-),j'(\Tot_{\cat Q}(X_{-p})) \\
&\cong \hocolim_p \pretr{\cat Q}(-,\Tot_{\cat Q}(X_{-p}))
\end{align*}
is h-projective, being a homotopy colimit of representables (hence h-projectives). \qed
\subsubsection*{Essential image} Recall from Theorem \ref{thm:construction} that the totalisation functor $\Tot_{\cat Q}$ induces an equivalence $H^0(\Tw^-(\cat Q)) \cong \dercompmin{\cat Q}^\mathrm{hfp}$, so it is enough to prove that the essential image of \eqref{eq:exfun_comparison} coincides with the objects of the form $\Tot_{\cat Q}(X) \in \dercompmin{\cat Q}^\mathrm{hfp}$, for some $X \in \Tw^-(\cat Q)$. 

Given $A \in \cat A$, we know from Lemma \ref{lemma:prep2} that $\cat A(j(-),A)$ is isomorphic to $\Tot_{\cat Q}(X)$ in $\dercomp{\cat Q}$, for some $X \in \Tw^-(\cat Q)$, hence the essential image of \eqref{eq:exfun_comparison} is contained in $\dercompmin{\cat Q}^\mathrm{hfp}$. 

On the other hand, take an object $\Tot_{\cat Q}(X) \in \dercompmin{\cat Q}^\mathrm{hfp}$, for some $X \in \Tw^-(\cat Q)$. Upon shifting, we may assume that $X$ is of the form
\begin{equation*}
X= (\bigoplus_{h \leq 0} Q_h[-h],q).
\end{equation*}
By Lemma \ref{lemma:hlc_twminus_essimg}, we know that $\Tot_{\cat Q}(X) \in \dercomp{\cat Q}_{\leq 0}$. Let $X_{-p} = \sigma_{\geq -p} X$ and consider the sequence $(\Tot_{\cat Q}(X_{-p}) \to \Tot_{\cat Q}(X_{-p-1}))_p$ (recall \S \ref{subsec:twcompl_colim}). This clearly gives a sequence in $\cat A$:
\begin{equation*}
(T_{\cat Q}(X_{-p}) \xrightarrow{j_{-p,-p-1}} T_{\cat Q}(X_{-p-1}))_p
\end{equation*}
We claim that $T_{\cat Q}(X_{-p}) \in \cat A_{\leq 0}$ for all $p$. Indeed, we have that $T_{\cat Q}(X_0) = X_0 \in \cat Q \subseteq \cat A_{\leq 0}$, and then there is a pretriangle
\begin{equation*}
Q_{-p-1}[p] \to T_{\cat Q}(X_{-p}) \xrightarrow{j_{-p,-p-1}} T_{\cat Q}(X_{-p-1}) \to Q_{-p-1}[p+1],
\end{equation*}
from where we see that, assuming inductively that $T_{\cat Q}(X_{-p}) \in \cat A_{\leq 0}$, we get $T_{\cat Q}(X_{-p-1}) \in \cat A_{\leq 0}$. Now, the sequence $(\Tot_{\cat Q}(X_{-p}) \to \Tot_{\cat Q}(X_{-p-1}))_p$ is isomorphic (recall \eqref{eq:tot_notational_iso}) to
\begin{equation*}
(\cat A(j(-),T_{\cat Q}(X_{-p})) \xrightarrow{(j_{-p,-p-1})_*} \cat A(j(-), T_{\cat Q}(X_{-p}))_p.
\end{equation*}
By Proposition \ref{prop:tot_hocolim_iso_lim}, we know that
\begin{equation*}
\Tot_{\cat Q}(X) \approx \hocolim_p \Tot_{\cat Q}(X_{-p}) \cong \hocolim_p \cat A(j(-),T_{\cat Q}(X_{-p})).
\end{equation*}

Now, set
\begin{equation*}
A = \hocolim_p T_{\cat Q}(X_{-p}),
\end{equation*}
which lies in $\cat A_{\leq 0}$ by Remark \ref{remark:hocolim_aisle}, and comes with closed degree $0$ maps
\begin{equation*}
j_{-p} \colon T_{\cat Q}(X_{-p}) \to A
\end{equation*}
such that $[j_{-p-1}] = [j_{-p,-p-1} \circ j_{-p}]$ in $H^0(\cat A)$. In particular, the diagram
\begin{equation*}
\begin{gathered}
\xymatrix{
\cat A(j(-),T_{\cat Q}(X_{-p})) \ar[r]^-{(j_{-p})_*} \ar[d]_{(j_{-p,-p-1})_*} & \cat A(j(-),A) \\
\cat A(j(-),T_{\cat Q}(X_{-p-1})) \ar[ur]_{(j_{-p-1})_*}
}
\end{gathered}
\end{equation*}
is commutative in $\dercomp{\cat Q}$, and recalling the ``weak universal property'' of the homotopy colimit (see \S \ref{subsec:hocolim_trcat}) we get a morphism in $\dercomp{\cat Q}$:
\begin{equation*}
\Tot_{\cat Q}(X) \to \cat A(j(-),A)
\end{equation*} 
and for all $i$ a commutative diagram
\begin{equation*}
\begin{gathered}
\xymatrix{
H^i(\Tot_{\cat Q}(X)) \ar[r]  & H^i(\cat A(j(-),A)) \\
H^i(j(-), T_{\cat Q}(X_{-p})). \ar[u]  \ar[ur]_{H^i((j_{-p})_*)}
}
\end{gathered}
\end{equation*}

Now, fix $i \in \mathbb Z$ and recall from Lemma \ref{lemma:truncation_cohomology_constant} that
\begin{equation*}
H^i(\Tot_{\cat Q}(X_{-p})) \cong H^i(\cat A(j(-), T_{\cat Q}(X_{-p})) \to H^i(\Tot_{\cat Q}(X))
\end{equation*}
is an isomorphism if $i > -p$. Moreover, by \eqref{eq:dgproj_tstruct_compare} the map $H^i((j_{-p})_*)$ can be identified with
\begin{equation*}
H^0(\cat A)^\heartsuit(S^{-1}(-),H^i(T_{\cat Q}(X_{-p}))) \xrightarrow{H^i(j_{-p})_*} H^0(\cat A)^\heartsuit(S^{-1}(-),H^i(A))
\end{equation*}

So, in order to show that $H^i(\Tot_{\cat Q}(X)) \to H^i(\cat A(j(-),A))$ is an isomorphism, it is now sufficient to show that $H^i(T_{\cat Q}(X_{-p})) \to H^i(A)$ is an isomorphism if $i > -p$. The idea is to show that
\begin{equation*}
H^i(T_{\cat Q}(X_{-p})) \to H^i(T_{\cat Q}(X_{-p-1}))
\end{equation*}
is an isomorphism for $i > -p$ and an epimorphism for $i=p$, and then apply Lemma \ref{lemma:tcohom_holim} (see also Remark \ref{remark:tcohom_holim_leftsep_applicable}). In fact, we may apply $T_{\cat Q} \colon \Tw^-_\bd (\cat Q) \to \cat A$ to the pretriangle \eqref{eq:twtruncation_triangle}, obtaining a pretriangle in $\cat A$:
\begin{equation*}
T_{\cat Q}(Q_{-p-1})[p] \to T_{\cat Q}(X_{-p}) \to T_{\cat Q}(X_{-p-1}) \to T_{\cat Q}(Q_{-p-1})[p+1].
\end{equation*}
Notice that $T_{\cat Q}(Q_{-p-1}) = Q_{-p-1} \in \cat Q \subset \cat A$, and it has cohomology concentrated in nonpositive degrees. Hence, applying the t-structure cohomology of $\cat A$ and arguing as in the proof of Lemma \ref{lemma:truncation_cohomology_constant}, we see that $H^i(T_{\cat Q}(X_{-p})) \to H^i(T_{\cat Q}(X_{-p-1}))$ is indeed an isomorphism for $i > -p$ and an epimorphism for $i=p$. Putting everything together, we conclude that
\begin{equation*}
\Tot_{\cat Q}(X) \to \cat A(j(-),A)
\end{equation*}
is an isomorphism in $\dercomp{\cat Q}$.

\subsection{The correspondence}
Theorems \ref{thm:construction} and \ref{thm:comparison} tell us that the dg-categories of the form $\Tw^-(\cat Q)$ when $\cat Q$ is a hlc dg-category such that $H^0(\cat Q)$ is Karoubian and closed under countable coproducts are precisely the dg-categories $\cat A$ is any dg-category with a non-degenerate right bounded t-structure with enough derived projectives and which is closed under countable coproducts. This correspondence can be made into an equivalence of categories, as we are going to show.
\begin{defin} \label{defin:Hqe_DGProj}
The category $\Hqe^{\mathrm{DGProj}}$ is defined as follows: 
\begin{itemize}
\item The objects are the homotopically locally coherent dg-categories $\cat Q$ such that $H^0(\cat Q)$ is Karoubian.
\item The morphisms are the morphisms $F: \cat Q \to \cat Q'$ in $\Hqe$ (isomorphism classes of quasi-functors) with the property that for any $Q' \in \cat Q'$, the restricted module along the functor $H^0(F)$:
\begin{equation*}
H^0(\cat Q')(F(-),Q') = H^0(\cat Q')(H^0(F)(-),Q')
\end{equation*}
lies in $\Modfp{H^0(\cat Q)}$.
\end{itemize}

We also denote by $\Hqe^{\mathrm{DGProj}}_{\oplus}$ the full subcategory of $\Hqe^{\mathrm{DGProj}}$ of dg-categories $\cat Q$ such that $H^0(\cat Q)$ is closed under countable coproducts.
\end{defin} 
\begin{lemma}
$\Hqe^{\mathrm{DGProj}}$ is a subcategory of $\Hqe$.
\end{lemma}
\begin{proof}
Clearly, for any dg-category $\cat Q \in \Hqe^{\mathrm{DGProj}}$ and for any $Q \in \cat Q$, we have that $H^0(\cat Q)(-,Q)$ lies in $\Modfp{H^0(\cat Q)}$, hence $\Hqe^{\mathrm{DGProj}}$ is closed under identities.

To show closure under compositions, let
\begin{equation*}
\cat Q \xrightarrow{F} \cat Q' \xrightarrow{G} \cat Q''
\end{equation*}
be morphisms in $\Hqe^{\mathrm{DGProj}}$. Let $Q'' \in \cat Q''$; by hypothesis, we have an exact sequence:
\begin{equation*}
H^0(\cat Q'')(-,Q_1) \to H^0(\cat Q'')(-,Q_0) \to H^0(\cat Q'')(G(-),Q'') \to 0,
\end{equation*}
from which we restrict along $H^0(F)$ and get an exact sequence:
\begin{equation*}
H^0(\cat Q'')(F(-),Q_1) \to H^0(\cat Q'')(F(-),Q_0) \to H^0(\cat Q'')(GF(-),Q'') \to 0.
\end{equation*}
Now, since both $H^0(\cat Q'')(F(-),Q_1)$ and $H^0(\cat Q'')(F(-),Q_0)$ are finitely presented by hypothesis, the same is true for the cokernel $H^0(\cat Q'')(GF(-),Q'')$, as desired.
\end{proof}
\begin{defin} \label{defin:Hqe_tmin}
The category $\Hqe^{\mathrm{t-}}$ is defined as follows:
\begin{itemize}
\item The objects are dg-categories $\cat A$ endowed with non-degenerate right bounded t-structure with enough derived projectives.
\item The morphisms are the morphisms $F \colon \cat A \to \cat B$ in $\Hqe$ (isomorphism classes of quasi-functors) such that they admit a t-exact right adjoint $G \colon \cat B \to \cat A$ .
\end{itemize}

We also denote by $\Hqe^{\mathrm{t-}}_{\oplus}$ the full subcategory of $\Hqe^{\mathrm{t-}}$ of dg-categories $\cat A$ such that the t-structure is closed under countable coproducts.
\end{defin}
Notice that two dg-categories $\cat A, \cat B \in \Hqe^{\mathrm{t-}}$ are isomorphic in $\Hqe^{\mathrm{t-}}$ if and only if there is an isomorphism $\cat A \cong \cat B$ in $\Hqe$ which preserves the t-structures.
\begin{remark} \label{remark:hprojminhfp_Twmin_ind}
If $\cat Q, \cat Q'$ are hlc dg-categories with Karoubian $H^0$ and $F \colon \cat Q \to \cat Q'$ is any dg-functor, then the dg-functor $\Ind_F \colon \hproj{\cat Q} \to \hproj{\cat Q'}$ actually restricts to a dg-functor
\begin{equation}
\Ind_F \colon \hprojmin{\cat Q}^\mathrm{hfp} \to \hprojmin{\cat Q'}^\mathrm{hfp},
\end{equation}
and the following diagram is commutative, with vertical arrows being quasi-equivalences:
\begin{equation} \label{eq:Tw-_hproj-_commute}
\begin{gathered}
\xymatrix{
\Tw^-(\cat Q) \ar[r]^{\Tw^-(F)} \ar[d]^{\approx}_{\Tot_{\cat Q}} & \Tw^-(\cat Q') 
\ar[d]_{\approx}^{\Tot_{\cat Q'}} \\
\hprojmin{\cat Q}^\mathrm{hfp} \ar[r]^-{\Ind_F} & \hprojmin{\cat Q'}^\mathrm{hfp}.
}
\end{gathered}
\end{equation}
Indeed, if $M \in \hprojmin{\cat Q}^\mathrm{hfp}$, we know from Proposition \ref{prop:Twmin_hprojhfp_essimg} that $M \cong \Tot(Y)$ in $H^0(\hproj{\cat Q})$. By the commutative diagram \eqref{eq:Tw-_hproj_commute}, we have
\begin{equation*}
\Ind_F(M) \cong \Ind_F(\Tot(Y)) \cong \Tot(\Tw^-(F)(Y))
\end{equation*}
in $H^0(\hproj{\cat Q'})$, and we know that $\Tot(\Tw^-(F)(Y)) \in \hprojmin{\cat Q'}^\mathrm{hfp}$ by Lemma \ref{lemma:hlc_twminus_essimg}.
\end{remark}
Now, recall from Proposition \ref{prop:Twmin_qeq} that the functor $\cat B \mapsto \Tw^-(\cat B)$ preserves quasi-equivalences, hence it induces a functor
\begin{align*}
\Tw^-  \colon \Hqe & \to \Hqe, \\
\cat B &\mapsto \Tw^-(\cat B).
\end{align*}
If $\cat Q$ is an hlc dg-category with Karoubian $H^0$, we shall identify $\Tw^-(\cat Q)$ with $\hprojmin{\cat Q}^\mathrm{hfp}$ via the totalisation $\Tot_{\cat Q}$. If $F \colon \cat Q \to \cat Q'$ is a dg-functor between such dg-categories, then $\Tw^-(F)$ is identified with $\Ind_F$ thanks to by \eqref{eq:Tw-_hproj-_commute}.
\begin{lemma}
The functor $\Tw^- \colon \Hqe \to \Hqe$ induces a functor
\begin{equation} \label{eq:correspondence}
\begin{split}
 \Tw^- \colon \Hqe^{\mathrm{DGProj}} & \to \Hqe^{\mathrm{t-}}, \\
\cat Q &\mapsto \Tw^-(\cat Q) \equiv \hprojmin{\cat Q}^\mathrm{hfp},
\end{split}
\end{equation}
where we endow $\Tw^-(\cat Q)$ with the natural t-structure of Theorem \ref{thm:construction}.
\end{lemma}
\begin{proof}
We already know from Theorem \ref{thm:construction} that if $\cat Q \in \Hqe^{\mathrm{DGProj}}$ then $\hproj{\cat Q}^\mathrm{hfp} \in \Hqe^{\mathrm{t-}}$. It only remains to show that for any morphism $F \colon \cat Q \to \cat Q'$ in $\Hqe^{\mathrm{DGProj}}$, the morphism $\Tw^-(F) \colon \Tw^-(\cat Q) \to \Tw^-(\cat Q')$ has a right adjoint which preserves the t-structures. Since $\Tw^-$ preserves quasi-equivalences, without loss of generality we can assume that $\cat Q$ is cofibrant, so that $F$ can be represented by a dg-functor -- which, abusing notation, we also denote by $F$. Since we are identifying $\Tw^-(-)$ with $\hprojmin{-}^\mathrm{hfp}$ via the totalisation dg-functor, by \eqref{eq:Tw-_hproj-_commute} we see that $\Tw^-(F)$ is identified with $\Ind_F \colon \hprojmin{\cat Q}^\mathrm{hfp} \to \hprojmin{\cat Q'}^\mathrm{hfp}$. From Remark \ref{remark:ind_qfun_adj}, we know that $\res_F \colon \hproj{\cat Q'} \to \hproj{\cat Q}$ is the right adjoint quasi-functor of $\Ind_F \colon \hproj{\cat Q} \to \hproj{\cat Q'}$; being a lift of the restriction functor $\dercomp{\cat Q'} \to \dercomp{\cat Q}$, it is readily seen that $\res_F$ preserves the t-structures. To conclude, it is sufficient to show that $\res_F$ restricts to a quasi-functor
\begin{equation*}
\res_F \colon \hprojmin{\cat Q'}^\mathrm{hfp} \to \hprojmin{\cat Q}^\mathrm{hfp},
\end{equation*}
which will be the right adjoint of $\Ind_F \colon \hprojmin{\cat Q}^\mathrm{hfp} \to \hprojmin{\cat Q'}^\mathrm{hfp}$. Hence, we need to show that for any $M \in \dercompmin{\cat Q'}^\mathrm{hfp} \cong H^0(\hprojmin{\cat Q'}^\mathrm{hfp})$, the restriction $M \circ F$ lies in $\dercompmin{\cat Q}^\mathrm{hfp} \cong H^0(\hprojmin{\cat Q}^\mathrm{hfp})$. First, since $H^i(M) = 0$ for $i \gg 0$, we immediately see that $H^i(M\circ F) =0$ for $i \gg 0$. Moreover, by hypothesis we have an exact sequence
\begin{equation*}
H^0(\cat Q')(-,Q_1) \to H^0(\cat Q')(-,Q_0) \to H^i(M) \to 0,
\end{equation*}
which by restriction induces an exact sequence
\begin{equation*}
H^0(\cat Q')(F(-),Q_1) \to H^0(\cat Q')(F(-),Q_0) \to H^i(M \circ F) \to 0.
\end{equation*}
By assumption, $H^0(\cat Q')(F(-),Q_1)$ and $H^0(\cat Q')(F(-),Q_0)$ are finitely presented, so the same is true for the cokernel $H^i(M \circ F)$, as desired.
\end{proof}
Finally, we prove:
\begin{thm} \label{thm:correspondence}
The functor \eqref{eq:correspondence}
\begin{equation*}
\Tw^- \colon \Hqe^{\mathrm{DGProj}} \to \Hqe^{\mathrm{t-}}.
\end{equation*}
is fully faithful, and induces an equivalence of categories:
\begin{equation*}
\Tw^- \colon \Hqe^{\mathrm{DGProj}}_{\oplus} \to \Hqe^{\mathrm{t-}}_{\oplus}.
\end{equation*}
The inverse is given by
\begin{equation}
\begin{split}
\mathrm{DGProj} \colon \Hqe^{\mathrm{t-}}_{\oplus} & \to \Hqe^{\mathrm{DGProj}}_{\oplus}, \\
\cat A & \to \DGProj(\cat A).
\end{split}
\end{equation}
\end{thm}
The proof of the above theorem requires some care with the technical details, using the language of quasi-functors.
\begin{lemma} \label{lemma:qfun_hproj_qis}
Let 
\begin{equation*}
T \colon \hprojmin{\cat Q}^\mathrm{hfp} \to \hprojmin{\cat Q'}^\mathrm{hfp}
\end{equation*}
be a quasi-functor. Then, for all $Y \in \hprojmin{\cat Q}^\mathrm{hfp}$, the restriction $T^{h_{\cat Q'}(-)}_Y$ lies in $\dercompmin{\cat Q'}^\mathrm{hfp}$, and  there is an isomorphism in $\dercomp{\basering k}$
\begin{equation}
T^X_Y \qis \compdg{\cat Q'}(X,T^{h_{\cat Q'}(-)}_Y),
\end{equation}
``natural'' in $X$ and $Y$ in the sense that it lifts to an isomorphism in the derived category $\dercomp{\opp{\hprojmin{\cat Q}^\mathrm{hfp}} \otimes \hprojmin{\cat Q'}^\mathrm{hfp}}$. In particular, given another quasi-functor
\begin{equation*}
T' \colon \hprojmin{\cat Q}^\mathrm{hfp} \to \hprojmin{\cat Q'}^\mathrm{hfp},
\end{equation*}
we have that $T \cong T'$ (in the derived category) if and only if $T^{h_{\cat Q'}(-)} \cong T^{h_{\cat Q'}(-)}$ in the derived category $\dercomp{\opp{\hprojmin{\cat Q}^\mathrm{hfp}} \otimes \cat Q'}$.

Moreover, if $T_Y$ is an h-projective dg-module for all $Y$, the same is true for $T^{h_{\cat Q'}(-)}_Y$.
\end{lemma}
\begin{proof}
Since $T$ is a quasi-functor, for all $Y \in \hprojmin{\cat Q}^\mathrm{hfp}$ we have
\begin{equation*}
T_Y \qis \hprojmin{\cat Q'}^\mathrm{hfp}(-,F(Y)),
\end{equation*}
for some $F(Y) \in \hprojmin{\cat Q'}^\mathrm{hfp}$. In particular, by the Yoneda lemma:
\begin{equation*}
T^{h_{Q'}(-)}_Y \qis \compdg{\cat Q'}(h_{\cat Q'}(-),F(Y)) \cong F(Y) \in \hprojmin{\cat Q'}^\mathrm{hfp},
\end{equation*}
hence $T^{h_{Q'}(-)}_Y \in \dercompmin{\cat Q'}^\mathrm{hfp}$.

There is a natural morphism
\begin{equation*}
T^X_Y \to \compdg{\cat Q'}(X, T^{h_{\cat Q'}(-)}_Y)
\end{equation*}
which is induced by the action
\begin{equation*}
T^X_Y \otimes \hprojmin{\cat Q'}^{\mathrm{hfp}}(h_{\cat Q'}(-),X) \to T_Y^{h_{\cat Q'}(-)},
\end{equation*}
and we have a commutative diagram:
\begin{equation*}
\begin{gathered}
\xymatrix{
T^X_Y \ar[r] \ar[d]^\approx & \compdg{\cat Q'}(X,T^{h_{\cat Q'}(-)}_Y) \ar[d]^\approx \\
\hprojmin{\cat Q'}^\mathrm{hfp}(X,F(Y)) \ar[r]^-\sim & \hprojmin{\cat Q'}^\mathrm{hfp}(X,\compdg{\cat Q'}(h_{\cat Q'}(-),F(Y))).
}
\end{gathered}
\end{equation*}
The rightmost vertical arrow is an isomorphism in $\dercomp{\basering k}$ because $X$ is h-projective; the lower horizontal arrow is a (strict) isomorphism by the Yoneda lemma. This implies that the upper horizontal arrow is an isomorphism in $\dercomp{\basering k}$, as we wanted. Next, let $T' \colon \hprojmin{\cat Q}^\mathrm{hfp} \to \hprojmin{\cat Q'}^\mathrm{hfp}$ be another quasi-functor. If $T \cong T'$ then $T^{h_{\cat Q'}(-)} \cong (T')^{h_{\cat Q'}(-)}$ by restriction. On the other hand, assume there is an isomorphism $T^{h_{\cat Q'}(-)} \xrightarrow{\approx} (T')^{h_{\cat Q'}(-)}$. For $X \in \hprojmin{\cat Q'}^\mathrm{hfp}$ and $Y \in \hprojmin{\cat Q}^\mathrm{hfp}$, this induces
\begin{equation*}
T^X_Y \approx \compdg{\cat Q'}(X, T^{h_{\cat Q'}(-)}_Y) \xrightarrow{\approx} \compdg{\cat Q'}(X,T'^{h_{\cat Q'}(-)}_Y) \approx (T')^X_Y,
\end{equation*}
which gives an isomorphism of quasi-functors.

Now, assume that $T_Y$ is h-projective for all $Y$. Notice that $T^{h_{\cat Q'}}_Y$ is the image of $T_Y$ with respect to the restriction functor
\begin{equation*}
\compdg{\hprojmin{\cat Q'}^\mathrm{hfp}} \to \compdg{\cat Q'}
\end{equation*}
along the Yoneda embedding $\cat Q' \hookrightarrow \hprojmin{\cat Q'}^\mathrm{hfp}$. This restriction functor maps any representable $\hprojmin{\cat Q'}^\mathrm{hfp}(-,M)$ to
\begin{equation*}
\hprojmin{\cat Q'}^\mathrm{hfp}(h_{\cat Q'}(-),M) = \compdg{\cat Q'}(h_{\cat Q'}(-),M) \cong M
\end{equation*}
by the Yoneda lemma. By \cite[Proposition 3.2, (4)]{canonaco-stellari-internalhoms} we deduce that this restriction functor preserves h-projective dg-modules, hence we have that $T^{h_{\cat Q'}(-)}_Y$ is actually h-projective, as claimed.
\end{proof}
\begin{lemma} \label{lemma:qfun_hproj_adj}
Giving an adjunction of quasi-functors
\begin{equation*}
T \dashv S \colon \hprojmin{\cat Q}^\mathrm{hfp} \to \hprojmin{\cat Q'}^\mathrm{hfp}
\end{equation*}
is the same as giving an isomorphism in $\dercomp{\basering k}$
\begin{equation}
\compdg{\cat Q'}(Q(T)^{h_{\cat Q'}(-)}_Y, X) \qis \compdg{\cat Q}(Y,S^{h_{\cat Q}(-)}_X) \approx S^Y_X,
\end{equation}
``natural'' in $X$ and $Y$ in the sense that it lifts to an isomorphism in the derived category $\dercomp{\opp{\hprojmin{\cat Q}^\mathrm{hfp}} \otimes \hprojmin{\cat Q'}^\mathrm{hfp}}$. Here $Q(T)$ is, as usual, an h-projective resolution of $T$ as a bimodule.
\end{lemma}
\begin{proof}
For simplicity, assume that $T$ is h-projective, and identify $Q(T) = T$. We have:
\begin{align*}
\compdg{\cat Q'}(T^{h_{\cat Q'}(-)}_Y,X) & \cong  \compdg{\hprojmin{\cat Q'}^\mathrm{hfp}}(h_{T^{h_{\cat Q'}(-)}_Y}, h_X) \qquad \text{(Yoneda)} \\
&\qis \compdg{\hprojmin{\cat Q'}^\mathrm{hfp}}(T_Y,h_X),
\end{align*}
where the second isomorphism follows from the above Lemma \ref{lemma:qfun_hproj_qis} and the fact that both $T_Y$ and the representable module $h_{T^{h_{\cat Q'}(-)}_Y}$ are h-projective (recall \cite[Lemma 3.4]{canonaco-stellari-internalhoms}). Now, since
\begin{equation*}
\compdg{\cat Q}(Y,S^{h_{\cat Q}(-)}_X) \qis  S^Y_X \approx \compdg{\cat Q}(h_Y, S_X)
\end{equation*}
again by Lemma \ref{lemma:qfun_hproj_qis}, we conclude recalling the definition of adjoint quasi-functors \eqref{eq:qfunct_adjoints}. Every isomorphism above is ``natural'' in $X$ and $Y$ in the sense that it lifts to an isomorphism in the convenient derived category of bimodules.
\end{proof}

Finally, we can prove Theorem \ref{thm:correspondence}:
\begin{proof}[Proof of Theorem \ref{thm:correspondence}]
Essential surjectivity of $\Tw^- \colon \Hqe^\mathrm{DGProj}_{\oplus} \to \Hqe^{t-}_{\oplus}$ follows directly from Theorem \ref{thm:comparison} and the fact that if $\cat A \in \Hqe^\mathrm{DGProj}_{\oplus}$ then $\DGProj(\cat A)$ is closed under countable coproducts (Lemma \ref{lemma:dgproj_hlc_karoubian}); we need to show fully faithfulness. Assume without loss of generality that $\cat Q$ is cofibrant, so that any quasi-functor defined on $\cat Q$ is actually isomorphic to a (strict) dg-functor. We are going to prove that the inverse of
\begin{align*}
\Hqe^\mathrm{DGProj}(\cat Q,\cat Q') & \to \Hqe^{t-}(\hprojmin{\cat Q}^\mathrm{hfp}, \hprojmin{\cat Q'}^\mathrm{hfp}), \\
F &\mapsto \Ind_F
\end{align*}
is given by the restriction map
\begin{align*}
\Hqe^{t-}(\hprojmin{\cat Q}^\mathrm{hfp}, \hprojmin{\cat Q'}^\mathrm{hfp}) & \to \Hqe^\mathrm{DGProj}(\cat Q,\cat Q'), \\
T &\mapsto T|_{\cat Q}.
\end{align*}
from which we also see that $\mathrm{DGProj}$ gives actually the inverse functor of $\Tw^- \colon \Hqe^\mathrm{DGProj}_{\oplus} \to \Hqe^{t-}_{\oplus}$. First, since any $T \in \Hqe^{t-}(\hprojmin{\cat Q}^\mathrm{hfp}, \hprojmin{\cat Q'}^\mathrm{hfp})$ has a t-exact right adjoint, by Proposition \ref{prop:restr_dgproj_preserve} we know that $H^0(T)$ preserves the derived projectives, which means that the essential image of $H^0(T|_{\cat Q})$ is $H^0(\cat Q')$; hence $T|_{\cat Q}$ is actually a quasi-functor $\cat Q \to \cat Q'$, and the above restriction map is well defined.

Now, start with a dg-functor $F \colon \cat Q \to \cat Q'$. Then, it is well-known that ${\Ind_F}|_{\cat Q} \cong F$. On the other hand, start with a quasi-functor 
\begin{equation*}
T \colon \hprojmin{\cat Q}^\mathrm{hfp} \to \hprojmin{\cat Q'}^\mathrm{hfp}
\end{equation*}
admitting a right adjoint $S$ which preserves the t-structures. Upon replacing $T$ with an h-projective resolution $Q(T)$, we can assume that $T$ is h-projective as a bimodule. The restriction $T|_{\cat Q} \colon \cat Q \to \cat Q'$ is given by the bimodule $T_{h_{\cat Q}(-)}^{h_{\cat Q'}(-)}$. Since $\cat Q$ is cofibrant, we can assume that there is a dg-functor $F \colon \cat Q \to \cat Q'$ and an isomorphism in the suitable derived category
\begin{equation*}
T_{h_{\cat Q}(-)}^{h_{\cat Q'}(-)} \approx \cat Q'(-,F(-)).
\end{equation*}
Now, we would like to prove that $\Ind_F \cong T$ as quasi-functors. Clearly, this is equivalent to proving that the right adjoints are isomorphic: $\res_F \cong S$. By Lemma \ref{lemma:qfun_hproj_qis}, it is enough to prove that $\res_F^{h_{\cat Q}(-)} \cong S^{h_{\cat Q}(-)}$. We compute:
\begin{align*}
(\res_F)_Y^{h_{\cat Q}(Q)} &= \compdg{\cat Q}(h_Q, Y \circ F) \\
&\cong (Y \circ F)(Q) \qquad \text{(Yoneda)} \\
& \cong \compdg{\cat Q'}(h_{F(Q)}, Y) \\
&\approx \compdg{\cat Q'} (T^{h_{\cat Q'}(-)}_{h_{\cat Q}(Q)}, Y) \\
& \approx S^{h_{\cat Q}(Q)}_Y. \qquad \text{(Lemma \ref{lemma:qfun_hproj_adj})}
\end{align*}
Every isomorphism above is ``natural'' in the sense that it lifts to an isomorphism in the convenient derived category. The isomorphism in $\dercomp{\basering k}$
\begin{equation*}
\compdg{\cat Q'}(h_{F(Q)}, Y) \approx \compdg{\cat Q'} (T^{h_{\cat Q'}(-)}_{h_{\cat Q}(Q)}, Y)
\end{equation*}
holds because $h_{F(Q)}$ and $T^{h_{\cat Q'}(-)}_{h_{\cat Q}(Q)}$ are both h-projective (recall Lemma \ref{lemma:qfun_hproj_qis}). Our proof is complete.
\end{proof}

\bibliographystyle{amsplain}
\bibliography{biblio_publication}
\end{document}